\documentclass{aic}

\mathchardef\standardl=\mathcode`l
\mathcode`l=\ell
\newcommand{\deactivatel}{\mathcode`l=\standardl}
\makeatletter
\edef\operator@font{\operator@font\noexpand\deactivatel}
\makeatother

\newcommand{\mlvar}[1]{\mathit{\deactivatel #1}}

\usepackage{amsmath, wrapfig}
\usepackage{dsfont, amsthm, amssymb, amsfonts, graphicx}
\usepackage{fancyhdr, xspace, psfrag, setspace, supertabular, color, enumerate}
\usepackage{hyperref}

\newtheorem{theorem}{Theorem}[section]
\newtheorem{proposition}[theorem]{Proposition}
\newtheorem{lemma}[theorem]{Lemma}

\newtheorem{claim}{Claim}
\newtheorem{corollary}[theorem]{Corollary}
\newtheorem{conjecture}[theorem]{Conjecture}
\newtheorem{definition}[theorem]{Definition}

\newtheorem{example}[theorem]{Example}
\newtheorem{question}[theorem]{Question}

\def\a{\alpha}

\def\g{\gamma}

\def\F{\mathbb{F}}

\def\R{\mathbb{R}}
\def\C{\mathbb{C}}
\def\E{\mathbb{E}}

\DeclareMathOperator{\rk}{rk}
\DeclareMathOperator{\sprk}{sprk}
\DeclareMathOperator{\erk}{erk}
\DeclareMathOperator{\drk}{drk}
\DeclareMathOperator{\tr}{tr}
\DeclareMathOperator{\etr}{etr}
\DeclareMathOperator{\dtr}{dtr}
\DeclareMathOperator{\sr}{sr}
\DeclareMathOperator{\esr}{esr}
\DeclareMathOperator{\dsr}{dsr}
\DeclareMathOperator{\pr}{pr}
\DeclareMathOperator{\epr}{epr}
\DeclareMathOperator{\dpr}{dpr}
\DeclareMathOperator{\Rrk}{\mathit{R}rk}
\DeclareMathOperator{\eRrk}{e\mathit{R}rk}
\DeclareMathOperator{\dRrk}{d\mathit{R}rk}
\DeclareMathOperator{\frank}{frank}
\DeclareMathOperator{\efrank}{efrank}

\DeclareMathOperator{\trp}{trp}

\newcommand{\Rminusrk}{\mathop{R_{-}\mathrm{rk}}}
\DeclareMathOperator{\Rcomprk}{\mathit{R}_{comp}rk}

\DeclareMathOperator{\Rdashrk}{\mathit{R^\prime}rk}
\DeclareMathOperator{\eRcomprk}{e\mathit{\Rcomp}rk}
\DeclareMathOperator{\eRdashrk}{e\mathit{R^\prime}rk}

\DeclareMathOperator{\Rcomp}{\mathit{R}_{\mathrm{comp}}}
\DeclareMathOperator{\Rnew}{\mathit{R}_{\mathrm{new}}}

\DeclareMathOperator{\Ronerk}{\mathit{R}_{1}rk}
\DeclareMathOperator{\Rtwork}{\mathit{R}_{2}rk}
\DeclareMathOperator{\Rprodrk}{(\mathit{R}_1 \times \mathit{R}_{2})rk}

\DeclareMathOperator{\scc}{sc}
\DeclareMathOperator{\lc}{lc}

\DeclareMathOperator{\ar}{ar}
\DeclareMathOperator{\bias}{bias}

\DeclareMathOperator{\matfrank}{frank}

\DeclareMathOperator{\Id}{Id}

\aicAUTHORdetails{%
  title = {High-Rank Subtensors of High-Rank Tensors}, 
  author = {Thomas Karam},
  plaintextauthor = {Thomas Karam},
  keywords = {ranks of tensors, subtensors, ranks from partitions, disjoint coordinates},
}   

\aicEDITORdetails{%
   year={2026},
   number={10},
   received={24 August 2022},   
   revised={31 May 2026},    
   published={9 October 2026},  
   doi={10.19086/aic/2026.10},      
}   

\begin{document}

\begin{frontmatter}[classification=text]


\author[tk]{Thomas Karam}

\begin{abstract}
Let $d \ge 2$ be an integer. We show that for a class of notions $R$ of rank for order-$d$ tensors, which includes in particular the tensor rank, the slice rank and the partition rank, there exist functions $F_{d,R}$ and $G_{d,R}$ such that if an order-$d$ tensor has $R$-rank at least $G_{d,R}(l)$ then we can restrict its entries to a product of sets $X_1 \times \dots \times X_d$ such that the restriction has $R$-rank at least $l$ and the sets $X_1, \dots, X_d$ each have size at most $F_{d,R}(l)$. Furthermore, our proof methods allow us to show that under a very natural condition we can require the sets $X_1, \dots, X_d$ to be pairwise disjoint.

\end{abstract}
\end{frontmatter}


\tableofcontents

\section{Introduction}\label{section: Introduction}

The last few years have seen a sequence of successes in using notions of ranks for higher-dimensional tensors to solve combinatorial problems. A central idea from the breakthrough solution to the cap-set problem by Ellenberg and Gijswijt \cite{Ellenberg Gijswijt}, which was based on a technique of Croot, Lev, and Pach \cite{Croot Lev Pach}, was reformulated by Tao \cite{Tao} in terms of the notion of slice rank for tensors, leading to what is now known as the slice rank polynomial method. The slice rank was further studied by Sawin and Tao \cite{Sawin and Tao}, and bounds shown there on the slice rank involving orderings on the coordinates were later used by Sauermann \cite{Sauermann} to prove under suitable conditions the existence of solutions with pairwise distinct variables to systems of equations in subsets of $\F_p^n$ that are not exponentially sparse. Another fruitful generalisation of the idea underlying the slice rank has been the partition rank, which was defined by Naslund \cite{Naslund} in order to prove a polynomial upper bound on the size of subsets of $\F_{p^r}^n$ not containing any $k$-right corners (with $p$ a prime integer and $r \ge 1$ a positive integer) and very recently used again by Naslund \cite{Naslund recent} to prove exponential lower bounds on the chromatic number of $\R^n$ with multiple forbidden distances.

In this paper we will focus on high-rank subtensors of tensors: it is a standard fact from linear algebra that if $A$ is a matrix of rank $k$ then $A$ has a $k \times k$ submatrix with rank $k$, and we will study here the extent to which this statement can be generalised to notions of rank for higher-order tensors, in particular to the tensor rank, to the slice rank and to the partition rank. The results that we obtain in this direction as well as the methods that we use in their proofs will also allow us to prove that under a very natural assumption we can find a subtensor such that the coordinates take values in pairwise disjoint sets. As we explain in a few paragraphs, the formulation of this result also arises naturally as an analogue of the standard inequality that every oriented graph has a bipartition such that at least a quarter of the edges go from the first part to the second.

We now define the relevant notions of higher-dimensional ranks for tensors and state our main theorems.

\begin{definition}

Let $d \ge 2$ be an integer and let $\mathbb{F}$ be a field. An \emph{order-$d$ tensor} over $\mathbb{F}$ is a function $T: Q_1 \times \dots \times Q_d \rightarrow \mathbb{F}$ for some finite subsets $Q_1, \dots, Q_d$ of $\mathbb{N}$. \end{definition}

Throughout this paper we shall use the following notation. We write $\F$ for an arbitrary field. If $d \ge 2$ a positive integer, then $Q_1, \dots, Q_d$ will always stand for finite subsets of $\mathbb{N}$, even if this is not explicitly indicated. Given an order-$d$ tensor $T: Q_1 \times \dots \times Q_d \rightarrow \F$ and subsets $X_1 \subset Q_1, \dots, X_d \subset Q_d$, we shall write $T(X_1 \times \dots \times X_d)$ for the restriction $T': X_1 \times \dots \times X_d \rightarrow \F$ of $T$. For each positive integer $n$ we write $\lbrack n \rbrack$ for the set $\{1,2,\dots,n\}$. Given $x\in Q_1 \times \dots \times Q_d$, and $I\subset\lbrack d \rbrack$, we write $x(I)$ for the restriction $(x_{\a}: \a \in I)$ of $x$ to its coordinates in $I$. Given a bipartition $\{I,J\}$ of $\lbrack d \rbrack$ and points $y \in \prod_{\a \in I} Q_{\a}$, $z \in \prod_{\a \in J} Q_{\a}$, we write $T(y,z)$ for the value $T(x)$ where the element $x \in \prod_{\a=1}^d Q_{\a}$ is defined by $x_\a = y_\a$ for each $\a \in I$ and $x_\a = z_\a$ for each $\a \in J$. If $T: \prod_{\a=1}^d Q_{\a} \rightarrow \mathbb{F}$ is an order-$d$ tensor, $I \subset \lbrack d \rbrack$ and $y \in \prod_{\a \in I^c} Q_{\a}$, then we write $T_y: \prod_{\a \in I} Q_{\a} \rightarrow \mathbb{F}$ for the order-$|I|$ tensor defined by $T_y(z) = T(y,z)$, with $T(y,z)$ defined as in the previous sentence. Given arbitrary subsets $I,J$ of $\lbrack d \rbrack$, and tensors $T_1: \prod_{\a \in I} Q_{\a} \rightarrow \mathbb{F}$ and $T_2: \prod_{\a \in J} Q_{\a} \rightarrow \mathbb{F}$, we write $T_1.T_2: \prod_{\a \in I \Delta J} Q_{\a} \rightarrow \mathbb{F}$ for the tensor defined by \[ (T_1.T_2)(y,z) = \sum_{x \in \prod_{\a \in I \cap J} Q_{\a}} T_1(y,x) T_2(z,x) \] for each $y \in \prod_{\a \in I \setminus J} Q_{\a}$ and each $z \in \prod_{\a \in J \setminus I} Q_{\a}$. We will also apply this definition when $T_1$ or $T_2$ is a vector $u: Q_\a \to \F$ for some $\a \in [d]$ rather than a tensor. If $d,s \ge 1$ are positive integers, $T_1, \dots, T_s$ are order-$d$ tensors over $\mathbb{F}$, and $a \in \mathbb{F}^s$, then we write $a.T$ for the linear combination $\sum_{i=1}^s a_i T_i$. (In the case of $u \in \F^{Q_1}$, $Q_1$ refers to one of the $d$ coordinate sets, whereas in the case of $a \in \F^s$, $[s]$ does not refer to any coordinate set.) Whenever this dot product notation is used, which variables are summed over and which variables are not will be clear from the context. 

\begin{definition} \label{tensor, slice, partition rank definition}

Let $d \ge 2$ be an integer, and let $T$ be an order-$d$ tensor. We say that $T$ has \emph{tensor rank at most $1$} if there exist functions $a_{\a}: Q_{\a} \rightarrow \mathbb{F}$ for each $\a \in \lbrack d \rbrack$ such that \[ T(x_1,\dots ,x_d) = a_1(x_1)\dots a_d(x_d) \] for every $(x_1,\dots ,x_d) \in Q_1 \times \dots \times Q_d$.

We say that $T$ has \emph{slice rank at most $1$} if there exist $\a \in \lbrack d \rbrack$ and functions $a: Q_{\a} \rightarrow \mathbb{F}$ and $b: \prod_{\a' \in \lbrack d \rbrack, \a' \neq \a} Q_{\a'} \rightarrow \mathbb{F}$ such that we can write \[ T(x_1,\dots ,x_d) = a(x_{\a})b(x_1,\dots ,x_{\a-1},x_{\a+1},\dots ,x_d) \] for every $(x_1,\dots ,x_d) \in Q_1 \times \dots \times Q_d$.

We say that $T$ has \emph{partition rank at most $1$} if there exist a bipartition $\{I,J\}$ of $\lbrack d \rbrack$ with $I,J$ both non-empty and functions $a: \prod_{\a \in I} Q_\a \rightarrow \mathbb{F}$ and $b: \prod_{\a \in J} Q_\a \rightarrow \mathbb{F}$ such that we can write \[ T(x_1,\dots ,x_d) = a(x(I)) b(x(J)) \] for every $(x_1,\dots ,x_d) \in Q_1 \times \dots \times Q_d$.

We say that the \emph{tensor rank} (resp. \emph{slice rank}, resp. \emph{partition rank}) of $T$ is the smallest nonnegative integer $k$ such that there exist tensors $T_1,\dots ,T_k$ each of tensor rank at most $1$ (resp. slice rank at most $1$, resp. partition rank at most $1$) and such that $T = T_1 + \dots  + T_k$. We denote by $\tr T$ the tensor rank of $T$, by $\sr T$ the slice rank of $T$, and by $\pr T$ the partition rank of $T$.

\end{definition}

Whenever $d \ge 2$ is a positive integer and $T$ is an order-$d$ tensor, we always have  $\pr T \le \sr T \le \tr T$: this follows from the fact that every order-$d$ tensor with tensor rank at most $1$ also has slice rank at most $1$ and every order-$d$ tensor with slice rank at most $1$ also has partition rank at most $1$. For $d=2$ the three notions coincide and are the same as the usual notion of rank for matrices. For $d=3$ the slice and partition ranks are the same, but are smaller than the tensor rank in general. For $d \ge 4$ the three notions are pairwise distinct in general. These definitions can be unified further, as special cases of the following definition.

\begin{definition} \label{Rrk definition}

Let $d \ge 2$ be an integer, and let $R$ be a non-empty family of partitions of $\lbrack d \rbrack$. We say that an order-$d$ tensor $T$ has \emph{$R$-rank at most $1$} if there exist a partition $P \in R$ and for each $I \in P$ a function $a_I: \prod_{\a \in I} Q_{\a} \rightarrow \mathbb{F}$ such that we can write \[ T(x_1,\dots ,x_d) = \prod_{I \in P} a_I(x(I)) \] for every $(x_1,\dots ,x_d) \in Q_1 \times \dots \times Q_d$. We say that the \emph{$R$-rank} of $T$ is the smallest nonnegative integer $k$ such that there exist order-$d$ tensors $T_1,\dots ,T_k$ with $R$-rank at most $1$ such that $T = T_1 + \dots  + T_k$. \end{definition}

We will denote by $\Rrk T$ the $R$-rank of $T$. We can check that for every $d \ge 2$, the $R$-rank specialises to the tensor rank, to the slice rank, and to the partition rank, by taking respectively \begin{align*} 
R & = \{\{\{1\},\{2\},\dots, \{d\}\}\} \\
R & = \{\{\{1\},\{1\}^c\}, \{\{2\},\{2\}^c\}, \dots  , \{\{d\},\{d\}^c\}\} \\
R & = \{\{I,J\}: \{I,J\} \text{ a bipartition of } \lbrack d \rbrack \text{ with } I,J \neq \emptyset\}. 
\end{align*}

This paper will be concerned with subtensors of tensors, and the first main topic will be an extension to the $R$-ranks of tensors of the fact that every matrix of rank $k$ has a $k \times k$ subtensor with rank $k$. In Section \ref{section: Counterexample} we will provide an example of an order-$3$ tensor with slice rank $4$ but that has all its $4 \times 4 \times 4$ subtensors with slice rank at most $3$. Nonetheless, in Section \ref{section: Spanning} we will show that every order-$d$ tensor $T$ with tensor rank $k$ has a $k \times k \times \dots \times k$ ($d$ times) subtensor with tensor rank $k$, and our first main theorem will be the following asymptotic result.


\begin{theorem} \label{Subtensors theorem} Let $d \ge 2$ be an integer, and let $R$ be a non-empty family of partitions of $\lbrack d \rbrack$. There exist functions $F_{d,R}: \mathbb{N} \rightarrow \mathbb{N}$ and $G_{d,R}: \mathbb{N} \rightarrow \mathbb{N}$ such that if $T$ is an order-$d$ tensor with $\Rrk T \ge G_{d,R}(l)$ then there exist $X_1 \subset Q_1, \dots, X_d \subset Q_d$ each with size at most $F_{d,R}(l)$ such that $\Rrk T(X_1 \times \dots \times X_d) \ge l$.\end{theorem}

We emphasize that the field $\F$ is irrelevant in Theorem \ref{Subtensors theorem}, as in all the results of this paper aside from those of Appendix \ref{section: Projection}.

Let $E$ be the set of points $(x_1, \dots, x_d) \in Q_1 \times \dots \times Q_d$ that do \emph{not} have pairwise distinct coordinates. The following definition will be central to our second main result.

\begin{definition} \label{Essential and disjoint rank definition} Let $d \ge 2$ be an integer, let $R$ be a non-empty family of partitions of $\lbrack d \rbrack$. For $T: Q_1 \times \dots \times Q_d \rightarrow \mathbb{F}$ an order-$d$ tensor we define the \emph{essential $R$-rank} \[ \eRrk T = \min_V \Rrk (T+V) \] where the minimum is taken over all order-$d$ tensors $V: Q_1 \times \dots \times Q_d \rightarrow \F$ with support contained inside $E$, and the \emph{disjoint $R$-rank} \[ \dRrk T= \max_{X_1,\dots ,X_d} \Rrk (T(X_1 \times \dots \times X_d)) \] where the maximum is taken over all $X_1 \subset Q_1, \dots, X_d \subset Q_d$ with $X_1,\dots ,X_d$ pairwise disjoint.

\end{definition}

In the $d=2$ case we will write $\erk$ and $\drk$ for respectively the associated essential rank and the disjoint rank corresponding to the usual notion of rank for matrices. For general $d \ge 2$, in the special cases $\Rrk = \tr$, $\sr$, $\pr$, we will respectively write $\etr$, $\esr$, $\epr$ for the associated essential $R$-rank and $\dtr$, $\dsr$, $\dpr$ for the associated disjoint $R$-rank.

It seems worthwhile to compare the essential $R$-rank with the disjoint $R$-rank, as it is straightforward to show that a tensor has essential $R$-rank equal to $0$ if and only if it has disjoint $R$-rank equal to $0$: the corresponding tensors are the tensors supported inside $E$. Moreover, we can show that the disjoint $R$-rank is at most the essential $R$-rank.

\begin{lemma} Let $d \ge 2$ be an integer, and let $R$ be a non-empty family of partitions of $\lbrack d \rbrack$. For every order $d$ tensor $T: Q_1 \times \dots \times Q_d \rightarrow \F$ we have \[\dRrk T \le \eRrk T.\] \end{lemma}

\begin{proof} Let $X_1 \subset Q_1, \dots, X_d \subset Q_d$ be pairwise disjoint sets and let $V: Q_1 \times \dots \times Q_d \rightarrow \F$ be an order-$d$ tensor supported inside $E$. Since the support of $V$ is contained in $E$, which has empty intersection with $X_1 \times \dots \times X_d$, we have \[T(X_1 \times \dots \times X_d) = (T+V)(X_1 \times \dots \times X_d).\] Moreover, taking a restriction of a tensor cannot increase its $R$-rank, so \[\Rrk (T+V)(X_1 \times \dots \times X_d) \le \Rrk (T+V).\] Therefore \[\Rrk T(X_1 \times \dots \times X_d) \le \Rrk (T+V),\] so taking the maximum over the $d$-tuples $(X_1, \dots, X_d)$ of pairwise disjoint sets and the minimum over the tensors $V$ supported inside $E$ we obtain the desired inequality. \end{proof}

Our second main result is a weak converse to this last inequality.

\begin{theorem} \label{Disjoint rank subtensors theorem} Let $d \ge 2$ be an integer, and let $R$ be a non-empty family of partitions of $\lbrack d \rbrack$. There exists a function $G_{d,R}': \mathbb{N} \rightarrow \mathbb{N}$ such that if $T$ is an order-$d$ tensor such that $\eRrk T \ge G_{d,R}'(l)$ then we have $\dRrk T \ge l$. \end{theorem}

Theorem \ref{Disjoint rank subtensors theorem} is also an essential ingredient to the proof of the main result of the paper \cite{Gowers and K equidistribution}, where in joint work with Timothy Gowers we generalise a theorem of Green and Tao (\cite{Green and Tao}, Theorem 1.7) on the approximate equidistribution of polynomials with high rank over finite prime fields to the case where the variables are chosen (uniformly and independently) at random in an arbitrary non-empty subset of the field rather than in the whole field. However, the present paper will not focus on this application.

When $R$ is a family of partitions that corresponds to one of the notions of rank that we have already defined, it will be convenient to use the notation for the notion of rank instead of for the corresponding family of partitions when writing a pair $(d,R)$, and similarly for the indices of $F_{d,R}$, $G_{d,R}$, $G_{d,R}'$. For instance when $R$ corresponds to the partition rank for order-$d$ tensors we will write the pair $(d,R)$ as $(d,\pr)$, and the functions $F_{d,R}$, $G_{d,R}$, $G_{d,R}'$ as $F_{d,\pr}$, $G_{d,\pr}$, $G_{d,\pr}'$ respectively.

The methods involved in our proofs of Theorem \ref{Subtensors theorem} and of Theorem \ref{Disjoint rank subtensors theorem} are similar in several ways: those that we will use to prove the latter can be viewed as a moderate complication of those that we will use to prove the former. The proofs of Theorem \ref{Subtensors theorem} and Theorem \ref{Disjoint rank subtensors theorem} will proceed by a double induction on $d$ and on the sizes of the parts in the partitions of $R$. A key auxiliary ingredient of the proof will be “boosting” result which allows us to transform, for a fixed pair $(d,R)$, Theorem \ref{Subtensors theorem} and Theorem \ref{Disjoint rank subtensors theorem} in their current formulation with one tensor to suitable respective analogues for any fixed number of tensors, and in this way to deduce information on a tensor from information on its slices. The paper is structured as follows. 

Section \ref{section: Counterexample} will begin by providing an example showing that already for the slice rank of order-$3$ tensors, one may not take the functions $F_{d,R}, G_{d,R}$ to both be the identity. In Section \ref{section: Spanning}, we will establish that Theorem \ref{Subtensors theorem} holds for the tensor rank for tensors of all orders, with furthermore $F_{d,R}(l)=l$ and $G_{d,R}=l$ for all $l$. This tensor rank case will constitute the base case of an inductive proof that will later be used to establish Theorem \ref{Subtensors theorem} in the general case of all $R$-rank notions for tensors of all orders. In Section \ref{section: Disjoint rank for matrices}, we will prove Theorem \ref{Disjoint rank subtensors theorem} in the case of matrix rank, as well as a multi-tensor version of that result. Multi-tensor versions of statements of Theorem \ref{Subtensors theorem} and Theorem \ref{Disjoint rank subtensors theorem}, for a given notion pair $(d,R)$, will later play an important role in the induction to establish the respective statements of these theorems for pairs $(d',R')$ with $d'>d$. In Section \ref{section: Disjoint tensor rank} we will prove Theorem \ref{Disjoint rank subtensors theorem} in the case of the tensor rank, which will again play the role of the base case in the inductive proof of Theorem \ref{Disjoint rank subtensors theorem}. 

In Section \ref{section: Slice rank for order-3 tensors} we will prove Theorem \ref{Subtensors theorem} and Theorem \ref{Disjoint rank subtensors theorem} in the case of the slice rank of order-$3$ tensors, using the versions of these theorems in the case of the tensor rank of order-$3$ tensors as well as the simplest forms of two important ideas that will play a major role in the main inductive proof: guarantees that a tensor has high rank from the existence of a family of slices of that tensor that are well-separated in a sense that will be made precise, and equivalences between rank notions under assumptions that lower-order slices have bounded rank. In Section \ref{section: Multi-tensor statements} we will establish a “boosting” result stating that for a given pair $(d,R)$, a multi-tensor version of Theorem \ref{Subtensors theorem} follows from its one-tensor version, and obtain a partially analogous result for Theorem \ref{Disjoint rank subtensors theorem}. 

Section \ref{section: Additional difficulties in order 4 and higher} will then be devoted to a proof of Theorem \ref{Subtensors theorem} for the partition rank of order-$4$ tensors, and will use slight generalisations of the techniques introduced in Section \ref{section: Slice rank for order-3 tensors} together with the “boosted” version of Theorem \ref{Subtensors theorem} for lower-order tensors. In Section \ref{section: Additional difficulties for the order-$4$ tripartition rank}, Theorem \ref{Subtensors theorem} will be proved in the first case – which we will call the tripartition rank of order-$4$ tensors – that did not appear to us to be amenable to a solution based purely on the techniques discussed so far. In Section \ref{section: General case}, we will prove generalised versions of all the tools already introduced – although all ideas involved will already have been covered in earlier sections – followed by the main induction argument that will prove Theorem \ref{Subtensors theorem}. Using variants of these arguments adapted to essential $R$-ranks and disjoint $R$-ranks we will prove Theorem \ref{Disjoint rank subtensors theorem} in that section as well. 

The last numbered section, Section \ref{section: Open problems} will discuss problems still left open by Theorem \ref{Subtensors theorem} and Theorem \ref{Disjoint rank subtensors theorem}. Appendix $A$ and Appendix $B$ will be devoted to proofs of bounds in Theorem \ref{Subtensors theorem} in two special cases: $R$-ranks for families $R$ that arise naturally as products of two families in a suitable sense, and the partition rank over finite fields.

For the reader only interested in a self-contained proof of Theorem \ref{Subtensors theorem}, without any additional accompanying examples and discussion, Sections 3,7,10 will suffice. As for Theorem \ref{Disjoint rank subtensors theorem}, Sections 4,5,7,10 will suffice.

\section{A counterexample to a strong bound for order-3 slice-rank subtensors}\label{section: Counterexample}

We thank Timothy Gowers for discovering Proposition \ref{Counterexample} and constructing the example in this section, which we include here with his permission.

\begin{proposition}\label{Counterexample} Let $\mathbb{F}$ be a field. Then there exists an order-$3$ tensor $T: \lbrack 11 \rbrack \times \lbrack 4 \rbrack \times \lbrack 15 \rbrack \rightarrow \mathbb{F}$ such that $\sr T = 4$ but whenever $X \subset \lbrack 11 \rbrack$, $Y \subset \lbrack 4 \rbrack$, $Z \subset \lbrack 15 \rbrack$ are all of size $4$, $\sr T(X \times Y \times Z) \le 3$. \end{proposition}

For $n_1, n_2, n_3$ positive integers, we say that a subset $U \subset \lbrack n_1 \rbrack \times \lbrack n_2 \rbrack \times \lbrack n_3 \rbrack$ is an \emph{antichain} if whenever $(x',y',z'), (x'',y'',z'') \in U$ are such that $x' \le x''$, $y' \le y''$, $z' \le z''$, necessarily $(x',y',z') = (x'',y'',z'')$. In particular a set of the type \[ \{(x,y,z) \in \lbrack n_1 \rbrack \times \lbrack n_2 \rbrack \times \lbrack n_3 \rbrack: x+y+z=k\} \] for some integer $k$ is an antichain.

If $U$ is a subset of $\lbrack n_1 \rbrack \times \lbrack n_2 \rbrack \times \lbrack n_3 \rbrack$, we say that the \emph{slice covering number} $\scc U$ is the smallest nonnegative integer $k$ such that $U$ can be covered by $k$ slices, i.e., such that there exist nonnegative integers $r,s,t$ with $r+s+t = k$, and $a_1, \dots, a_r \in \lbrack n_1 \rbrack$, $b_1,\dots,b_s \in \lbrack n_2 \rbrack$, $c_1, \dots,c_t \in \lbrack n_3 \rbrack$ satisfying \[ U \subset \left( \bigcup_{1 \le i \le r} \{x=a_i\} \right) \bigcup \left( \bigcup_{1 \le j \le s} \{y=b_j\} \right) \bigcup \left( \bigcup_{1 \le k \le t} \{z=c_k\} \right). \]

It is a special case of a result of Sawin and Tao (Proposition 4 in \cite{Sawin and Tao}) that the slice rank $\sr T$ of an order-$3$ tensor $T$ with support $U$ contained in an antichain is equal to the smallest number $\scc U$ of (order-$2$) slices that suffice to cover its support $U$.

The following definition will also be convenient for us. For $V$ a subset of $\lbrack n_1 \rbrack \times \lbrack n_2 \rbrack$, we say that the \emph{three-point line covering number} $\lc_3 V$ of $V$ is the smallest nonnegative integer $k$ such that $V$ can be covered by $k$ lines, with each line of one of the three types $\{x=a\}$ or $\{y=b\}$ or $\{x+y = c\}$.

Let $S: \lbrack 11 \rbrack \times \lbrack 4 \rbrack \rightarrow \mathbb{F}$ be a matrix with coefficients equal to 1 on a subset $V$ of $\lbrack 11 \rbrack \times \lbrack 4 \rbrack$, to be chosen later, and coefficients equal to $0$ outside $V$. Let $T,T': \lbrack 11 \rbrack \times \lbrack 4 \rbrack \times \lbrack 15 \rbrack \rightarrow \mathbb{F}$ be tensors defined by $T(x,y,z) = S(x,y) 1_{z=x+y}$ and $T':(x,y,z) \mapsto T(x,y,16-z)$, and let $U,U'$ be the respective supports of $T,T'$. The set $U'$ is contained in the antichain $\{x+y+z=16\}$, so it satisfies $\sr T' = \scc U'$. Since $\sr T = \sr T'$ and $\scc U' = \scc U$ we obtain $\sr T = \scc U$.

For any positive integer $k$, the intersections of the slices $\{x=a\}$, $\{y=b\}$, $\{z=c\}$ inside $\lbrack 11 \rbrack \times \lbrack 4 \rbrack \times \lbrack 15 \rbrack$ with the set $\{x+y+z=k\}$ are respectively the subsets $\{x=a\}$, $\{y=b\}$, $\{x+y = k - c\}$ of $\lbrack 11 \rbrack \times \lbrack 4 \rbrack$. Hence the following claim.

\begin{claim} \label{connection to covering numbers}We have $\scc U = \lc_3 V$. More generally, for any $X \subset \lbrack 11 \rbrack$, $Y \subset \lbrack 4 \rbrack$, $Z \subset \lbrack 15 \rbrack$, \[ \scc (U \cap (X \times Y \times Z)) = \lc_3 (V \cap \mu(X,Y,Z)) \] where $\mu(X,Y,Z) = \{(x,y)\in X\times Y: x+y \in Z\} $. \end{claim}

We choose a set \[ V = \{(2,1), (6,1), (11,1), (1,2), (11,3), (1,4), (6,4), (10,4) \} \subset  \lbrack 11 \rbrack \times \lbrack 4 \rbrack \] which we draw below as a matrix with the $x$ coordinates increasing from $1$ to $11$ from left to right, and the $y$ coordinates increasing from $1$ to $4$ from top to bottom). \[\begin{pmatrix}
0&1&0&0&0&1&0&0&0&0&1\\
1&0&0&0&0&0&0&0&0&0&0\\
0&0&0&0&0&0&0&0&0&0&1\\
1&0&0&0&0&1&0&0&0&1&0
\end{pmatrix}\] The description provided by Claim \ref{connection to covering numbers} together with the following two claims provides a proof of Proposition \ref{Counterexample}.

\begin{claim} $\lc_3 V = 4$.  \end{claim}

\begin{proof} The upper bound $\lc_3 V \le 4$ follows from taking the four lines $ y=1, y=2, y=3, y=4$. We now prove the lower bound. Every line (of the type $x=a$, $y=b$, or $x+y=c$) contains at most three points. To cover all eight points with three lines, at least two of the lines must have three points. The only such lines are the lines $y=1$ and $y=4$. This leaves us with the task of covering the two remaining points $(1,2)$ and $(11,3)$ with one line, which cannot be done. \end{proof}

\begin{claim} For $X = \lbrack 11 \rbrack$, $Y = \lbrack 4 \rbrack$, and every $Z \subset \lbrack 15 \rbrack$ of size $4$, $\lc_3 (V \cap \mu(X,Y,Z)) \le 3$.  \end{claim}

\begin{proof} The number of values of $ x+y$ among the eight points of $V$ is six: these values are $3$, $5$, $7$, $10$, $12$ and $14$. It follows that in any subtensor of size $4$ we must remove all the points from at least two of the lines of the type $ x+y=c$.

If we remove the line $x+y=3$, then we can cover the remaining points with the lines $x=11$, $y=1$, $y=4$. Similarly if we remove the line $x+y=14$, then we can cover the remaining points with the lines $x=1$, $y=1$, $y=4$.

If we remove the line $x+y=7$ (i.e. the point $(6,1)$) then we can cover the remaining points with the lines $x=11$, $y=4$, $x+y=3$. Similarly if we remove the line $x+y=10$ (i.e. the point $(6,4)$), then we can cover the remaining points with the lines $x=1$, $y=1$, $x+y=14$.

The only remaining possibility is to remove both the lines $x+y = 5$ and $x+y=12$, i.e. the points $(1,4)$ and $(11,1)$. We can then cover the remaining points with the lines $x=6$, $x+y=3$, $x+y=14$. \end{proof}

By monotonicity, Claim 3 remains true if we replace $X = \lbrack 11 \rbrack$, $Y = \lbrack 4 \rbrack$ by any $X \subset \lbrack 11 \rbrack$ and $Y \subset \lbrack 4 \rbrack$ with size $4$. This establishes Proposition \ref{Counterexample}.

Using the additivity of the slice rank over diagonal sums recently proved by Gowers \cite{Gowers}, we may extend the present example to a class of examples which shows in particular that the gap \[\sr T - \max_{|X|, |Y|, |Z| \le \sr T} \sr T(X \times Y \times Z)\] may be arbitrarily large.

\begin{proposition} For every positive integer $k$, there exists an order-$3$ tensor $T_k: [11k] \times [4k] \times [15k] \rightarrow \F$ such that $\sr T_k = 4k$ but $\sr T_k(X \times Y \times Z) \le (59/15)k$ for every $X \subset Q_1, Y \subset Q_2, Z \subset Q_3$ with size at most $4k$.\end{proposition}

\begin{proof} For each $i \in [k]$ and every positive integer $q$ let $I_{i,q}$ be the set of integers \[\{q(i-1)+1, \dots, qi\}.\] We take $T_k$ to be the diagonal sum of $k$ copies of the tensor $T$ from Proposition \ref{Counterexample}), i.e. we define $T_k$ by \[T_k(x,y,z) = T(x-11(i-1), y-4(i-1), z-15(i-1))\] if there exists $i \in [k]$ such that $x \in I_{i,11}, y \in I_{i,4}, z \in I_{i,15}$, and by $T_k(x,y,z) = 0$ otherwise. For each $i \in [k]$ we consider the translated intersections \begin{align*} X_i & = (X \cap I_{i,11}) - \{11(i-1)\} \\ Y_i & = (Y \cap I_{i,4}) - \{4(i-1)\} \\ Z_i & = (Z \cap I_{i,15}) - \{15(i-1)\}.\end{align*} Let $a_1,a_2,a_3,b$ be the proportions of indices $i \in [k]$ satisfying $|X_i| \le 3$, $|Y_i| \le 3$, $|Z_i| \le 3$, and $|X_i| = |Y_i| = |Z_i| = 4$ respectively. The proportion of indices $i \in [k]$ satisfying at least one of the first three properties is at least $\max(a_1,a_2,a_3) \ge (a_1+a_2+a_3)/3$. For such indices $i \in [k]$ we have \[\sr T(X_i \times Y_i \times Z_i) \le \min(|X_i|, |Y_i|, |Z_i|) \le 3\] and if the fourth property is instead satisfied, then we also have $\sr T(X_i \times Y_i \times Z_i) \le 3$ because $T$ is as in Proposition \ref{Counterexample}. Moreover for every $i \in [k]$ we have \[\sr T(X_i \times Y_i \times Z_i) \le \sr T = 4,\] so using the additivity of the slice rank on diagonal sums we obtain \[\sr T_k(X \times Y \times Z) = \sum_{i=1}^k \sr T(X_i \times Y_i \times Z_i) \le (4 - ((a_1+a_2+a_3)/3+b)) k.\] Because $X,Y,Z$ have size $4k$, the proportions of indices $i \in [k]$ such that $|X_i| \ge 5$, $|Y_i| \ge 5$, $|Z_i| \ge 5$ are at most respectively $4a_1, 4a_2, 4a_3$, so $b \ge 1 - 5(a_1+a_2+a_3)$ by the union bound. Thus, $(a_1+a_2+a_3)/3+b \ge 1/15$. We conclude $\sr T_k(X \times Y \times Z) \le (59/15)k$. \end{proof}

\section{Optimal bounds for tensor rank subtensors using spanning techniques}\label{section: Spanning}

In this section we prove Theorem \ref{Subtensors theorem} in the case of the tensor rank. Given a vector space $V$ over a field $\mathbb{F}$, and a subset $U$ of $V$, we write $ \langle U \rangle$ for the linear subspace of $V$ spanned by $U$. The steps of the proof will be modelled after those of the following standard proof of the existence of full-rank submatrices of matrices. Let $A: Q_1 \times Q_2 \rightarrow \mathbb{F}$ be a matrix. 

\begin{enumerate}

\item We characterise $\rk A$ as the smallest $k \ge 0$ such that there exist vectors $L_1,\dots,L_k \in \F^{Q_2}$ such that $A_x \in \langle L_1, \dots, L_k \rangle$ for each $x \in Q_1$.

\item We extract a basis $ (A_x: x \in X)$ of the linear subspace $\langle A_x: x \in Q_1 \rangle$ spanned by all rows of $A$, with a set $X$ of size $\dim \langle A_x: x \in Q_1 \rangle = \rk A$.

\item It follows from the two previous steps that $\rk A(X \times Q_2) = \rk A$. Iterating again on the second coordinate we are done.
\end{enumerate}

As an analogue of step 1 we begin by expressing the tensor rank of an order-$d$ tensor in terms of a notion of rank for families of order-$(d-1)$ tensors. We recall from the introduction the notation $\tr T$ for the tensor rank of a tensor $T$.

\begin{definition} \label{Spanning rank definition} Let $ V$ be a finite-dimensional vector space over $ \mathbb{F}$ and let $ S,F \subset V$ such that $S$ spans $F$. The \emph{spanning rank} $\sprk_S F$ of $F$ with respect to $S$ is the smallest nonnegative integer $k$ such that there exist vectors $v_1,\dots ,v_k \in S$ satisfying $ F \subset \langle v_1, \dots, v_k \rangle$. \end{definition}

When $ S = V$, we have $\sprk_S F = \dim \langle F \rangle$ for every $F \subset V$. For any fixed $F \subset V$, $S \mapsto \sprk_S F$ is decreasing for inclusion, so in particular we always have $\sprk_S F \ge \dim \langle F \rangle$.

\begin{lemma} \label{Connection between tensor rank and spanning rank} Let $ d \ge 3$ be an integer, and let $ S_{d-1}$ be the family of order-$(d-1)$ tensors $\prod_{\a=2}^d Q_{\a} \rightarrow \mathbb{F}$ with order-$(d-1)$ tensor rank equal to $1$. Then every order-$d$ tensor $T: \prod_{\a=1}^d Q_{\a} \rightarrow \mathbb{F}$ satisfies \[ \tr T = \sprk_{S_{d-1}} \{T_{x_1}: x_1 \in Q_1 \}. \] \end{lemma}

\begin{proof} Let $k$ be an nonnegative integer. If $ \tr T = k$ then there exist functions $ a_{i,\a}: Q_{\a} \rightarrow \mathbb{F}$ for each $i \in \lbrack k \rbrack$ and each $\a \in \lbrack d \rbrack$ such that \[ T(x_1,\dots ,x_d) = \sum_{i=1}^k \prod_{\a=1}^d a_{i,\a}(x_{\a}) \] for every $(x_1, \dots, x_d) \in \prod_{\a=1}^d Q_{\a}$. Then the order-$(d-1)$ tensors $ T^i=\prod_{\a=2}^d a_{i,\a}$ for each $ i \in \lbrack k \rbrack$ each have order-$(d-1)$ tensor rank at most $ 1$ and $ T_{x_1} \in \langle T^1,\dots ,T^k \rangle$ for each $ x_1 \in Q_1$. Conversely if $\sprk_{S_{d-1}} T = k$ then there exist order-$(d-1)$ tensors $T^1,\dots ,T^k: \prod_{\a=2}^d Q_{\a} \rightarrow \mathbb{F}$ with tensor rank at most $1$ such that for each $ x_1 \in Q_1$ there exist $ a_1(x_1),\dots ,a_k(x_1) \in \F$ satisfying $ T_{x_1} = \sum_{i=1}^k a_i(x_1) T^i$, so we can write \[ T(x_1,\dots ,x_d) = \sum_{i=1}^k a_i(x_1) T^i(x_2,\dots ,x_d) \] for every $(x_1,\dots ,x_d) \in \prod_{\a=1}^d Q_{\a}$, which shows that $T$ has tensor rank at most $ k$. \end{proof}

Our next lemma is an adaptation of the linear-algebra fact underlying step 2: the claim that a finite family of vectors of a vector space has a subfamily with the same rank as that of the original family and with the same size as its rank.

\begin{lemma} \label{Subfamilies for spanning rank} Let $k \ge 1$ be an integer, let $V$ be a finite-dimensional vector space over $ \mathbb{F}$, and let $S$ and $F$ be families of elements of $V$ such that $S$ spans $F$. Assume that $\sprk_S F \ge k$. Then there exists a subfamily $ F'$ of $ F$ with size at most $ k$ such that $ \sprk_S F' \ge k$. \end{lemma}

\begin{proof} We distinguish two cases depending on the dimension of the linear subspace $ \langle F \rangle$. If $\dim \langle F \rangle \ge k$ then we take $ F'$ to be a linearly independent family of $ k$ elements of $ F$; because $ \dim \langle F' \rangle = k$, in particular $ \sprk F' \ge k$. If on the other hand $ \dim \langle F \rangle \le k$ then we take $ F'$ to be a maximal linearly independent family of elements of $ F$; the family $ F'$ has size at most $ k$, and since $ \langle F \rangle = \langle F' \rangle$, a family of elements of $ S$ spans all elements of $ F$ if and only if it spans all elements of $ F'$, so $ \sprk F' = \sprk F$.\end{proof}

We are now ready to deduce our subtensors result, Theorem \ref{Subtensors theorem}, for the tensor rank.

\begin{proposition} \label{Subtensors for order-$d$ tensor rank} Let $d \ge 2$, $ k \ge 1$ be integers and let $ T: \prod_{\a=1}^d Q_{\a} \rightarrow \F$ be an order-$d$ tensor. Assume that $ \tr T \ge k$. Then there exist sets $ X_1 \subset Q_1, \dots,X_d \subset Q_d$ with size at most $ k$ such that $ \tr T(X_1 \times \dots  \times X_d) \ge k$. \end{proposition}

\begin{proof} If $d=2$, this is the standard result on full-rank submatrices of matrices. Suppose instead that $d \ge 3$. By Lemma \ref{Connection between tensor rank and spanning rank} we have $ \sprk_{S_{d-1}} \{T_{x_1}: x_1 \in Q_1 \} = \tr T$. Since $ \tr T \ge k$, by Lemma \ref{Subfamilies for spanning rank} there exists a subset $ X_1$ of $Q_1$ with size at most $ k$ such that $ \sprk_{S_{d-1}} \{T_{x_1}: x_1 \in X_1\} \ge k$, so applying Lemma \ref{Connection between tensor rank and spanning rank} again we have \[ \tr T(X_1 \times \prod_{j=2}^d Q_j) \ge k.\] Iterating this argument $ d-1$ more times, which we can, since the roles of the $d$ coordinates are the same in the definition of the tensor rank, we obtain the desired sets $ X_2,\dots ,X_d$. \end{proof}

The remainder of this section is devoted to proving a generalisation of Proposition \ref{Subtensors for order-$d$ tensor rank} to linear subspaces spanned by a fixed number of tensors. We begin by formulating such a generalisation for matrices. In the induction arguments that we will develop throughout later sections and which will prove Theorem \ref{Subtensors theorem}, such multi-tensor versions will play a crucial role.

\begin{proposition} \label{Multi-tensor matrix rank} Let $ s,k \ge 1$ be integers, and let $ A_1,\dots ,A_s: Q_1 \times Q_2 \rightarrow \mathbb{F}$ be matrices. Then there exists a subset $X \subset Q_1$ with size at most $sk$ such that \[\rk (a.A)(X \times Q_2) \ge \min (\rk (a.A), k)\] for every $a \in \F^s$. Iterating a second time, there also exists a subset $Y \subset Q_2$ with size at most $sk$ such that \[ \rk (a.A)(X \times Y) \ge \min (\rk (a.A), k)\] for every $a \in \F^s$.

\end{proposition}

\begin{proof}

We remove rows one by one until there are only at most $sk$ remaining rows, with the inductive step being as follows. Assume that there are still $ m > sk$ remaining rows and let $Q_1'$ be the set of remaining rows. Let $ \Lambda= \{a \in \mathbb{F}^s: \rk (a.A)(Q_1' \times Q_2) \le k\}$. We can take a family $(a^1,\dots,a^r)$ of elements of $\Lambda$ such that $r \le s$ and $\Lambda \subset \langle a^1,\dots,a^r \rangle$. For each $j \in \lbrack r \rbrack$, let $U_j \subset \mathbb{F}^{Q_1'}$ be the linear subspace consisting of all $b \in \mathbb{F}^{Q_1'}$ satisfying $ \sum_{x \in Q_1'} b_x (a^j.A)_x = 0$. The subspaces $U_1,\dots, U_r$ each have dimension at least $ m-k$, so their intersection $\bigcap_{1 \le j \le r} U_j$ has dimension at least $ m-sk > 0$. Taking a non-zero element $ b$ of this intersection and then taking $x \in Q_1'$ such that $ b_x \neq 0$, removing the $x$th row does not change the rank of $(a.A)(Q_1' \times Q_2)$ for any $ a \in \Lambda$. Moreover for all $ a \in \mathbb{F}^s \setminus \Lambda$ we have $ \rk (a.A )(Q_1' \times Q_2) \ge k+1$ so the rank of $(a.A)(Q_1' \times Q_2)$ is still at least $ k$ after removing the $x$th row from $Q_1'$. \end{proof}

We now apply Proposition \ref{Multi-tensor matrix rank} to obtain a multi-tensor version of Lemma \ref{Subfamilies for spanning rank}.

\begin{lemma} \label{Multi-tensor subfamilies for spanning rank} Let $s,k \ge 1$ be integers, let $V$ be a finite-dimensional vector space over $ \mathbb{F}$, let $S$ be a spanning family of elements of $ V$, let $Q$ be a finite subset of $\mathbb{N}$ and for each $ j \in \lbrack s \rbrack$ and $x \in Q$ let $T_{j,x}$ be an element of $V$. Then there exists a subset $X$ of $Q$ with size at most $sk$ such that \[\sprk_S ((a.T)_{x})_{x \in X} \ge \min(k, \sprk_S ((a.T)_{x})_{x \in Q}) \] for every $a \in \F^s$. \end{lemma}

\begin{proof} We fix an arbitrary choice of basis $B = (b_1, \dots, b_{\dim V})$ of $V$. For each $j \in \lbrack s \rbrack$, let $ A_j: Q \times \lbrack \dim V \rbrack \rightarrow \F$ be the matrix such that for each $x \in Q$, the row $(A_j)_x$ is the family of coefficients of $T_{j,x}$ written in the basis $B$. By Proposition \ref{Multi-tensor matrix rank} applied to the matrices $ A_1,\dots ,A_s$ there exists a set $ X$ of size at most $ sk$ such that \[\rk ((a.T)_{x})_{x \in X} \ge \min (k, \rk ((a.T)_{x})_{x \in Q}) \] for every $a \in \mathbb{F}^s$. We now fix $a \in \F^s$. If $ \rk ((a.T)_{x})_{x \in Q} \ge k$ then $ \rk ((a.T)_{x})_{x \in X} \ge k$, so since $ \sprk_S ((a.T)_{x})_{x \in X} \ge \rk ((a.T)_{x})_{x \in X}$ we conclude that \[ \sprk_S ((a.T)_{x})_{x \in X} \ge k. \] If on the other hand $ \rk ((a.T)_{x})_{x \in Q} \le k$ then $ \langle ((a.T)_{x})_{x \in X} \rangle = \langle ((a.T)_{x})_{x \in Q} \rangle$, so since the spanning ranks of two families with the same linear spans are equal we conclude that \[ \sprk_S ((a.T)_{x})_{x \in X} = \sprk_S ((a.T)_{x})_{x \in Q}.\qedhere\] \end{proof}

\begin{proposition} \label{Multi-tensor tensor rank subtensors} Let $ d,s,k \ge 1$ be integers, and $T_1,\dots ,T_s: Q_1 \times \dots \times Q_d \rightarrow \mathbb{F}$ be order-$d$ tensors. Then there exist subsets $X_1 \subset Q_1, \dots, X_d \subset Q_d$ all with size at most $sk$ such that \[ \tr (a.T)(X_1 \times \dots  \times X_d) \ge \min(k, \tr a.T) \] for every $ a \in \mathbb{F}^s$.\end{proposition}

\begin{proof} By Lemma \ref{Multi-tensor subfamilies for spanning rank} there exists $ X_1$ with size at most $ sk$ such that \[ \sprk ((a.T)_{x_1})_{x_1 \in X_1} \ge \min (k, \sprk ((a.T)_{x_1})_{x_1 \in Q_1})\] for every $a \in \mathbb{F}^s$, so applying Lemma \ref{Connection between tensor rank and spanning rank} to both sides we obtain \[ \tr (a.T)(X_1 \times \prod_{\a=2}^d Q_{\a}) \ge \min(k, \tr a.T) \] for every $a \in \mathbb{F}^s$. Iterating this argument $d-1$ more times we get the desired other sets $ X_2,\dots ,X_d$. \end{proof}

For any fixed positive integers $ k,s$ the bound $ sk$ can be seen to be optimal by taking $Q_1 = \dots = Q_d = \lbrack sk \rbrack$ and taking $T_1, \dots, T_s$ to have pairwise disjoint supports each of size exactly $ k$ and all contained in the diagonal $ \{(x_1, \dots, x_d) \in \lbrack sk \rbrack: x_1 = \dots = x_d \}$. 

Proposition \ref{Multi-tensor tensor rank subtensors} suggests the following conjecture, which would if true strengthen Theorem \ref{Subtensors theorem} in two ways: in its statement the lower bound on the $R$-rank of a restriction is the same as the $R$-rank of the original tensor in the regime where the latter is small, and the lower bounds apply to restrictions of linear combinations of several tensors. This conjecture however seems far out of reach of the methods of the present chapter.

\begin{conjecture} \label{Multi-tensor type 2 subtensors for the general case} Let $ d \ge 2, s \ge 1$ be integers and let $ R$ be a non-empty family of partitions on $ \lbrack d \rbrack$. Then there exists a function $ F_{d,R,s, \mathrm{same}}$ such that for every positive integer $l \ge 1$, whenever $ T_1,\dots ,T_s$ are order-$d$ tensors there exist $ X_1,\dots ,X_d$ of size at most $ F_{d,R,s, \mathrm{same}}(l)$ such that
\[ \Rrk (a.T)(X_1 \times \dots  \times X_d) \ge \min(l, \Rrk a.T) \] 
for every $ a \in \mathbb{F}^s$. \end{conjecture}

Conjecture \ref{Multi-tensor type 2 subtensors for the general case} would be false if we furthermore required $F_{d,R,s, \text{same}}(l) = sl$, as Section \ref{section: Spanning} shows for $d=3$, $s=1$, and $R = \{ \{\{1\},\{2,3\}\}, \{\{2\},\{1,3\}\}, \{\{3\},\{1,2\}\} \}$.

\section{Disjoint rank in the matrix case}\label{section: Disjoint rank for matrices}

We thank Lisa Sauermann for a sketch that led to Proposition \ref{Disjoint rank for matrices} and which we use here with her permission. We recall that the essential rank $\erk A$ and disjoint rank $\drk A$ of a matrix $A$ have been defined as the special case of Definition \ref{Essential and disjoint rank definition} corresponding to matrix rank: the essential rank of $A$ is the smallest rank that can be obtained by changing the diagonal entries of $A$, and the disjoint rank of $A$ is the largest rank that can be obtained by restricting $A$ to a product $X \times Y$ of disjoint sets $X,Y$. We note that Proposition \ref{Disjoint rank for matrices} establishes Theorem \ref{Disjoint rank subtensors theorem} in the case $d=2$.

A multi-tensor generalisation, Proposition \ref{Multi-tensor matrix disjoint rank}, will then be proved. Such multi-tensor generalisations play an important role in the inductive proof of Theorem \ref{Disjoint rank subtensors theorem} carried out in Section \ref{section: General case}. This will already be visible from the arguments in Subsection \ref{subsection:Proof of order-3 disjoint slice rank} establishing Theorem \ref{Disjoint rank subtensors theorem} in the case of the slice rank of order-$3$ tensors, and which will directly use Proposition \ref{Multi-tensor matrix disjoint rank}.

\begin{proposition} \label{Disjoint rank for matrices} Let $A: Q_1 \times Q_2 \rightarrow \mathbb{F}$ be a matrix. We have $\drk A \ge (\erk A)/3$. \end{proposition}

\begin{proof} Let $k= \drk A$. There exist disjoint subsets $X \subset Q_1$, $Y \subset Q_2$ such that $\rk A(X \times Y) = k$. By the standard result on full-rank submatrices of matrices we can furthermore require $X,Y$ to have size $k$. Let $z \in Q_1 \setminus (X \cup Y)$ and $w \in Q_2 \setminus (X \cup Y)$. If $z,w$ are distinct then $X \cup \{z\}$ and $Y \cup \{w\}$ are disjoint, so by definition of the disjoint rank, \[ \rk A((X \cup \{z\}) \times (Y \cup \{w\})) = \rk A(X \times Y).\] That the rank does not increase between $A(X \times Y)$ and $A((X \cup \{z\}) \times (Y \cup \{w\}))$ provides a linear system allowing us to obtain the value \[ A(z,w) = A(\{z\} \times Y) A(X \times Y)^{-1} A(X \times \{w\}). \] If $z=w$ then there exists a unique element $d(z) \in \F$ such that \[A(z,z) + d(z) = A(\{z\} \times Y) A(X \times Y)^{-1} A(X \times \{w\}).\] Let $D: Q_1 \times Q_2 \rightarrow \mathbb{F}$ be the matrix defined by $D(z,w) = 0$ if $z \neq w$ and $D(z,w) = d(z)$ if $z=w$. We get that for all $z \in Q_1 \setminus (X \cup Y)$ and all $w \in Q_2 \setminus (X \cup Y)$, \begin{align*}
(A+D)(z,w) & = A(\{z\} \times Y) A(X \times Y)^{-1} A(X \times \{w\}) \\
& = (A+D)(\{z\} \times Y) (A+D)(X \times Y)^{-1} (A+D)(X \times \{w\})
\end{align*} where the last equality holds since $D$ is identically $0$ on all three sets $\{z\} \times Y$, $X \times Y$, $X \times \{w\}$. Therefore \[\rk (A+D)((X \cup (Q_1 \setminus (X \cup Y))) \times (Y \cup (Q_2 \setminus (X \cup Y))) = \rk (A+D)(X \times Y).\] Simplifying both sides we get \[ \rk (A+D)((Q_1 \setminus Y) \times (Q_2 \setminus X)) = \rk A(X \times Y) = k. \] Moreover by subadditivity \[\rk(A+D) \le \rk (A+D)((Q_1 \setminus Y) \times (Q_2 \setminus X)) + \rk (A+D)(Y \times (U_2 \setminus X)) + \rk (A+D)((Q_1 \setminus Y) \times X).\] The first term of the right-hand side is equal to $k$ as we have just shown, and the second and third terms are each at most $k$ since $X,Y$ have size $k$. So $\rk(A+D) \le 3k$, whence $\erk A \le 3 \drk A$ as desired. \end{proof}

We now move to proving a multi-tensor generalisation of Proposition \ref{Disjoint rank for matrices}. As announced at the beginning of this section, Proposition \ref{Multi-tensor matrix disjoint rank} will be an important ingredient in establishing higher-order generalisations of Proposition \ref{Disjoint rank for matrices} and hence Theorem \ref{Disjoint rank subtensors theorem}.

\begin{proposition} \label{Multi-tensor matrix disjoint rank} Let $s,l \ge 1$ be integers, let $A_1,\dots ,A_s: Q_1 \times Q_2 \rightarrow \F$ be matrices. Assume that \[ \erk (a.A) \ge s(s+1)l + l \] for every $a \in \F^s \setminus \{0\}$. Then there exist disjoint subsets $X \subset Q_1$, $Y \subset Q_2$ such that \[ \rk (a.A)(X \times Y) \ge l \] for every $a \in \F^s \setminus \{0\}$. \end{proposition}

\begin{proof} We proceed by induction on $s \ge 1$. Proposition \ref{Disjoint rank for matrices} proves the $s=1$ case. Let $s \ge 2$ be an integer and let us assume that Proposition \ref{Multi-tensor matrix disjoint rank} holds up to $s-1$. By Proposition \ref{Disjoint rank for matrices} (and using that $s^2 + s + 1 \ge 3s$) there exist disjoint subsets $X^0 \subset Q_1$, $Y^0 \subset Q_2$ such that $\rk A_1(X^0 \times Y^0) \ge sl$. By the standard result on full-rank submatrices of matrices we can furthermore find subsets $X' \subset X^0$, $Y' \subset Y^0$ both with size at most $sl$ such that \begin{equation} \rk A_1(X' \times Y') \ge sl. \label{At least sl} \end{equation} If there existed $s$ linearly independent values of $a \in \F^s$ for which \[\rk (a.A)(X' \times Y') < l\] then, as these $s$ values of $a$ would constitute a basis of $\F^s$, then this would contradict \eqref{At least sl} by subadditivity. Therefore, there exists a strict linear subspace $W$ of $\F^s$ such that \[ \rk (a.A)(X' \times Y') \ge l \] for every $a \in \F^s \setminus W$. By the assumption \[ \rk (a.A) \ge s(s+1)l + l \] for every $a \in W \setminus \{0\}$ so (since a single row of column has rank at most $1$) by subadditivity \[ \rk (a.A)((Q_1 \setminus Y') \times (Q_2 \setminus X')) \ge s(s-1)l + l \] for every $a \in W \setminus \{0\}$. By the inductive hypothesis there exist disjoint subsets $X'' \subset Q_1 \setminus Y'$, $Y'' \subset Q_2 \setminus X'$ such that \[ \rk (a.A)(X'' \times Y'') \ge l \] for every $a \in W \setminus \{0\}$. Letting $X = X' \cup X''$ and $Y = Y' \cup Y''$ the sets $X$ and $Y$ are disjoint by construction and \[ \rk (a.A)(X \times Y) \ge l \] for every $a \in \F^s \setminus \{0\}$ as desired. \end{proof}

\section{Disjoint tensor rank using the flattening rank}\label{section: Disjoint tensor rank}

The main goal of this section is to prove Theorem \ref{Disjoint rank subtensors theorem} in the case of the tensor rank, which will constitute the base case of the induction that will later prove Theorem \ref{Disjoint rank subtensors theorem} in full generality in Section \ref{section: General case}.

Let $d \ge 2$ be an integer. For $I \subset \lbrack d \rbrack$, we write $E(I)$ for the set \[\{x(I) \in \prod_{\alpha \in I} Q_{\a}: x_{\a'} = x_{\a''} \text{ for some distinct }\a', \a'' \in I\}.\] In particular, the set $E(\lbrack d \rbrack)$ is the set $E$ of elements that do not have pairwise distinct coordinates.

\subsection{The order-3 case}

Our proof of Theorem \ref{Disjoint rank subtensors theorem} for the tensor rank for order-$3$ tensors will make use of the following notion of rank.

\begin{definition} \label{Flattening rank definition}

For $T: Q_1 \times Q_2 \times Q_3 \rightarrow \F$ an order-$3$ tensor, let the \emph{1-flattening rank} of $T$ be the $R$-rank of $T$ for $R = \{\{\{1\},\{2,3\}\}\}$, and let the \emph{essential 1-flattening rank} of $T$ be the essential $R$-rank of $T$ for that same $R$. We write $\frank_1 T$ and $\efrank_1 T$ for the $1$-flattening rank and essential $1$-flattening rank of $T$ respectively.

\end{definition}

We note that the $1$-flattening rank of $T$ is the rank of the matrix $A: Q_1 \times (Q_2 \times Q_3) \rightarrow \mathbb{F}$ defined by \[A(x,(y,z)) = T(x,y,z).\]

The structure of our proof will be as follows: starting with an order-$3$ tensor $T$ with high essential tensor rank, either the tensor $T$ has an order-$2$ slice with high essential rank, in which case we conclude using Proposition \ref{Disjoint rank for matrices}, or all order-$2$ slices of $T$ have bounded essential rank, in which case the essential tensor rank can be shown to be equivalent to the essential $1$-flattening rank, and we will then show (using again our assumption that all order-$2$ slices of $T$ have bounded essential rank) that because $T$ has high essential $1$-flattening rank, it must have high disjoint $1$-flattening rank, which suffices to ensure that $T$ has high disjoint tensor rank.

We begin by proving our equivalence statement between the essential ranks.

\begin{proposition} \label{Equivalence between the 3D flattening rank and tensor rank} Let $T: Q_1 \times Q_2 \times Q_3 \rightarrow \F$ be an order-$3$ tensor such that $\erk T_x \le m$ for every $x \in Q_1$. If $\efrank_1 T \le l$ then $\etr T \le (m+2)l^2$. \end{proposition}

\begin{proof} By monotonicity we may assume without loss of generality that $\efrank_1 T = l$. There exists a tensor $V$ supported inside $E$, functions $a_1, \dots, a_l: Q_1 \rightarrow \F$ and functions $b_1, \dots, b_l: Q_2 \times Q_3 \rightarrow \F$ such that we can write the decomposition \begin{equation} (T-V)(x,y,z) = \sum_{i=1}^l a_i(x) b_i(y,z) \label{essential 3D flattening rank decomposition}. \end{equation} The functions $a_1, \dots, a_l$ are linearly independent, so by Gaussian elimination there exists a subset $U \subset Q_1$ with size $l$ and functions $a_1^*, \dots, a_l^*: Q_1 \rightarrow \F$ such that $a_{i'}^*.a_i = \delta_{i,i'}$ for all $i,i \in \lbrack l \rbrack$. For a fixed $i \in \lbrack l \rbrack$, taking the dot product with $a_i^*$ in both sides of \eqref{essential 3D flattening rank decomposition} we get $b_i = a_i^*.(T-V)$, i.e.

\begin{equation} b_i = \sum_{x \in U} a_i^*(x) (T-V)_x \label{expression of b_i using T-V in 3D}. \end{equation} For each $x \in U$, the support of $V_x$ is contained in the union of the row $\{y=x\}$, the column $\{z=x\}$ and the diagonal $\{y=z\}$, so (using that a row or column has rank at most $1$) $\erk V_x \le 2$.

We note that it is straightforward to check that $\erk$ and $\etr$ are subadditive (as is $\eRrk$ for general $R$). It follows from the subadditivity of $\erk$ that $\erk (T-V)_x \le m+2$. Using \eqref{expression of b_i using T-V in 3D} and the subadditivity of $\erk$ we obtain $\erk b_i \le l(m+2)$. By the decomposition \eqref{essential 3D flattening rank decomposition} and the subadditivity of $\etr$ we have $\etr (T-V) \le l^2(m+2)$, so $\etr T \le l^2(m+2)$. \end{proof}

We will use the following inequality in the proof of Proposition \ref{Disjoint rank for order-3 tensors}, that is, Theorem \ref{Disjoint rank subtensors theorem} in the case of the tensor rank of order-$3$ tensors.

\begin{lemma} \label{dimension of columns at most tensor rank for 3D tensors} Let $T: Q_1 \times Q_2 \times Q_3 \rightarrow \F$ be an order-$3$ tensor. Then \[\frank_1(T) \le \tr T. \] \end{lemma}

\begin{proof} We write $l$ for $\tr T$. By definition of the tensor rank, we may write a decomposition \[T(x,y,z) = \sum_{i=1}^l a_i(x) b_i(y) c_i(z)\] for some functions $a_i: Q_1 \to \F, b_i: Q_2 \to \F, c_i: Q_3 \to \F$ with $i \in [l]$. For each $(y,z) \in Q_2 \times Q_3$, this provides the expression \[T_{y,z}(x) = \sum_{i=1}^l \sigma_i a_i(x)\] for some $\sigma_i \in \F$. The functions $a_1,\dots,a_l$ therefore span all $T_{y,z}$ with $(y,z) \in Q_2 \times Q_3$, concluding the proof.\end{proof}

Throughout the next proposition and its proof we will write $Q_{2,3}$ for $Q_2 \times Q_3 \setminus E(\{2,3\})$. We show that if an order-$3$ tensor has high essential $1$-flattening rank and all its order-$2$ slices have bounded essential rank then it has high disjoint $1$-flattening rank.

\begin{proposition} \label{Disjoint $frank_1$ assuming all $rk$ of slices of all sizes are bounded}

Let $l,m \ge 1$ be integers, and let $T: Q_1 \times Q_2 \times Q_3 \rightarrow \mathbb{F}$ be an order-$3$ tensor such that $\erk T_x, \erk T_y, \erk T_z \le m$ for every $x \in Q_1$, every $y \in Q_2$, and every $z \in Q_3$, respectively. If $\efrank_1 T \ge 10ml$ then there exist pairwise disjoint $X \subset Q_1$, $Y \subset Q_2$, $Z \subset Q_3$ such that $\frank_1 T(X \times Y \times Z) \ge l$.

\end{proposition}

\begin{proof}

Let $l$ be the largest integer such that there exist $x_1, \dots, x_l \in Q_1$, $y_1, \dots, y_l \in Q_2$, $z_1, \dots, z_l \in Q_3$ all pairwise distinct such that the restrictions \[ T_{y_1,z_1}(\{x_{1},\dots ,x_{l}\}),\dots ,T_{y_l,z_l}(\{x_{1},\dots ,x_{l}\}) \] are linearly independent. The pairwise disjoint sets $X=\{x_1, \dots, x_l\}$, $Y=\{y_1, \dots, y_l\}$, $Z=\{z_1, \dots, z_l\}$ then satisfy $\frank_1 T(X \times Y \times Z) \ge l$. In the remainder of the proof we will show that there exist functions $a,b: Q_{2,3} \rightarrow \F$ such that \begin{equation} \dim \bigg\langle T_{y,z} + a(y,z) 1_{x=y} + b(y,z) 1_{x=z}: (y,z) \in Q_{2,3} \bigg\rangle \le 10ml, \label{vertical dimension bound for order-3 tensors} \end{equation} which establishes $\efrank_1 T \le 10ml$. We define the following sets, which we will use throughout. \begin{align*}
X_1 &= \{x_1, \dots, x_l\}\\
W_{\a} &= Q_{\a} \setminus \{x_1, \dots, x_l, y_1, \dots, y_l, z_1, \dots, z_l \} \text{ for each } \a \in \{1,2,3\} \\
U_1 &= \{(y_1,z_1), \dots, (y_l,z_l)\} \\
W_{2,3} &= (W_2 \times W_3) \setminus E(\{2,3\}). \end{align*} We first show that we can define functions $a,b: W_{2,3} \rightarrow \F$ such that \[\dim \bigg\langle (T_{y,z} + a(y,z) 1_{x=y} + b(y,z) 1_{x=z})(W_1): (y,z) \in W_{2,3} \bigg\rangle \le l. \] In other words, we want to show that the matrix $A: Q_1 \times Q_{2,3} \rightarrow \F$ defined by \[ A(x,(y,z)) = T(x,y,z) \] is such that the rank of the restriction $A(W_1 \times W_{2,3})$ is at most $ l$ after some modifications of the entries of the type \[ A(u,(u,z)), A(u,(y,u)). \] We construct $a,b$ as follows. Let $x \in W_1$ and let $(y,z) \in W_{2,3} $. We distinguish two situations.

\begin{enumerate}

\item If $ x \neq y$ and $x \neq z$ then by maximality of $ l$, the $ (l+1) \times (l+1)$ submatrix \[ A((X_1 \cup \{x\}) \times (U_1 \cup \{(y,z)\})) \] has the same rank as the full-rank matrix $ A(X_1 \times U_1)$, so we can write \[ A(x,(y,z)) = A(\{x\} \times U_1)  A(X_1 \times U_1)^{-1} A(X_1 \times \{(y,z)\}). \]

\item If $ x=y$ or $x=z$ then because $A((X_1 \cup \{x\}) \times (U_1 \cup \{(y,z)\}))$ is an $(l+1) \times (l+1)$ matrix and $A(X_1 \times U_1)$ has full rank $l$, there exists if $x=y$ a unique $a(y,z) \in \F$ such that \[ A(x,(y,z)) + a(y,z) = A(\{x\} \times U_1) A(X_1 \times U_1)^{-1} A(X_1 \times \{(y,z)\}) \] and if $x=z$ a unique $b(y,z) \in \F$ such that \[ A(x,(y,z)) + b(y,z) = A(\{x\} \times U_1) A(X_1 \times U_1)^{-1} A(X_1 \times \{(y,z)\}). \]

\end{enumerate}

The values $a(y,z)$, $b(y,z)$ obtained in the second situation define functions $a,b: W_{2,3} \rightarrow \F$. Let $B: (X_1 \cup W_1) \times (U_1 \cup W_{2,3}) \rightarrow \F$ be the matrix defined for all $(x,(y,z)) \in W_1 \times W_{2,3}$ by \[ B(x,(y,z)) = A(x,(y,z)) + a(y,z) 1_{x=y} + b(y,z) 1_{x=z} \] and by $B(x,(y,z)) = A(x,(y,z))$ for all other $(x,(y,z))$. By construction of $a,b$ we have \[ B(W_1 \times W_{2,3}) = B(W \times U_1) B(X_1 \times U_1)^{-1} B(X_1 \times W_{2,3}) \] whence $\rk B (W_1 \times W_{2,3}) \le l$. For $j \in Q_1$ let us write $\{x=j\}$ for the slice $\{j\} \times Q_2 \times Q_3$, and similarly define the notations $\{y=j\}$ with $j \in Q_2$ and $\{z=j\}$ with $j \in Q_3$. Defining the slices \begin{align*}
S_{1,j} &= \{x = j\} \setminus (Q_1 \times E(\{2,3\}))\\
S_{2,j} &= \big(\{y = j\} \setminus (Q_1 \times E(\{2,3\}))\big) \cap \{x \in W_1\}\\
S_{3,j} &= \big(\{z = j\} \setminus (Q_1 \times E(\{2,3\}))\big) \cap \{x \in W_1\} \cap \{y \in W_2\} \end{align*} for each $j \in Q_1 \setminus W_1$, each $j \in Q_2 \setminus W_2$, each $j \in Q_3 \setminus W_3$, respectively, the set $Q_1 \times Q_{2,3}$ can be written as the disjoint union \[\Bigg( W_1 \times W_{2,3}\Bigg) \bigcup \left(\bigcup_{j \in Q_1 \setminus W_1} S_{1,j}\right) \bigcup \left(\bigcup_{j \in Q_2 \setminus W_2} S_{2,j}\right) \bigcup \left(\bigcup_{j \in Q_3 \setminus W_3} S_{3,j}\right).\] By definition of the sets $W_1$, $W_2$, $W_3$ the sets $Q_1 \setminus W_1$, $Q_2 \setminus W_2$, $Q_3 \setminus W_3$ each have size at most $3l$. Regardless of how we choose the functions $a,b: Q_{2,3} \to \F$, something that we are about to do, we have \[ \dim \bigg\langle (T_{y,z} + a(y,z) 1_{x=y} + b(y,z) 1_{x=z})(Q_1 \setminus W_1): (y,z) \in Q_{2,3} \bigg\rangle \] to be at most $|Q_1 \setminus W_1| \le 3l$. By assumption, for each $j \in Q_2 \setminus W_2$ we have $\erk T_{j} \le m$, so $a,b$ can be defined on $S_{2,j}$ such that \[ \dim \bigg\langle (T_{y,z} + a(y,z) 1_{x=y} + b(y,z) 1_{x=z})(W_1): (y,z) \in Q_{2,3}, y = j \bigg\rangle \le m. \] Similarly, for each $j \in Q_3 \setminus W_3$ we can define $a,b$ on $S_{3,j}$ such that \[ \dim \bigg\langle (T_{y,z} + a(y,z) 1_{x=y} + b(y,z) 1_{x=z})(W_1): (y,z) \in Q_{2,3}, y \in W_2, z = j \bigg\rangle \le m. \] We note that by construction, the successive choices of values of $(a,b)$ pertain to disjoint sets of pairs $(a,b)$, so are all compatible. The bound \eqref{vertical dimension bound for order-3 tensors} follows by subadditivity. \end{proof}

We are now ready to prove Theorem \ref{Disjoint rank subtensors theorem} for the tensor rank of order-$3$ tensors.

\begin{proposition} \label{Disjoint rank for order-3 tensors}

Let $T$ be an order-$3$ tensor, and let $l \ge 1$ be an integer. If $\etr T \ge 17500 l^3$ then $\dtr T \ge l$. \end{proposition}

\begin{proof} We distinguish two cases.
\medskip

\noindent \textbf{Case 1}: There exists $x \in Q_1$ or $y \in Q_2$ or $z \in Q_3$ such that $\erk T_x$ or $\erk T_y$ or $\erk T_z$ is at least $3l+2$. Without loss of generality we can assume that this occurs for some $x \in Q_1$. Then \[ \erk T_x((Q_2 \setminus \{x\}) \times (Q_3 \setminus \{x\})) \ge 3l.\]  By Proposition \ref{Disjoint rank for matrices} there exist disjoint subsets $Y \subset Q_2 \setminus \{x\}$ and $Z \subset Q_3 \setminus \{x\}$ such that $\rk T_x(Y \times Z) \ge l$. Letting $X = \{x\}$ we obtain $\tr T(X \times Y \times Z) \ge l$.
\medskip

\noindent \textbf{Case 2}: We have $\erk T_x, \erk T_y, \erk T_z \le 3l+2$ for respectively all $x \in Q_1$, all $y \in Q_2$, all $z \in Q_3$. Since $\etr T \ge (3l+4)(10(3l+2)l)^2$, by Proposition \ref{Equivalence between the 3D flattening rank and tensor rank} we have $\efrank_1 T \ge 10(3l+2)l$. By Proposition \ref{Disjoint $frank_1$ assuming all $rk$ of slices of all sizes are bounded} there exist pairwise disjoint subsets $X \subset Q_1$, $Y \subset Q_2$, $Z \subset Q_3$ such that $\frank_1 T(X \times Y \times Z) \ge l$, so in particular $\tr T(X \times Y \times Z) \ge l$ by Lemma \ref{dimension of columns at most tensor rank for 3D tensors}. \end{proof}

\subsection{The general case}

We now generalise the proof to order-$d$ tensors for an arbitrary fixed $d \ge 3$. The structure of the proof will be similar to that of the $d=3$ case, although with an additional complication compared to the case of order-$3$ tensors which we will now summarise. Although removing an order-$1$ slice from an order-$2$ tensor can only decrease the rank (and essential rank) of that order-$2$ tensor by at most $1$, for $d \ge 4$ removing an order-$(d-2)$ slice from an order-$(d-1)$ tensor can decrease the tensor rank (and essential tensor rank) of that order-$(d-1)$ tensor by an unbounded amount. This occurs for instance for the order-$4$ tensor \[T(x,y,z,w) = 1_{x=1} 1_{y=1} 1_{w=z+1},\] for which the order-$3$ slice obtained by fixing $x=1$ has unbounded tensor rank but removing the order-$2$ slice obtained by fixing $x=1,y=1$ from the order-$3$ slice reduces the tensor rank of that order-$3$ slice to $0$. As a result, for $d \ge 4$ we cannot proceed by imitating Case 1 of the proof of Proposition \ref{Disjoint rank for order-3 tensors} by replacing the matrix rank of order-$2$ tensors by the tensor rank of order-$(d-1)$ tensors. We resort to using the partition rank of these lower-order slices as an auxiliary tool instead, since the defect of the tensor rank that we just discussed does not arise for the partition rank. 

\begin{definition} \label{Flattening rank definition general case}

For $T: \prod_{\a=1}^d Q_{\a} \rightarrow \mathbb{F}$ an order-$d$ tensor, let the \emph{1-flattening rank} of $T$ be the $R$-rank of $T$ for $R = \{\{\{1\},\{2,\dots ,d\}\}\}$, and let the \emph{essential 1-flattening rank} of $T$ be the essential $R$-rank of $T$ for that same $R$. We write $\frank_1 T$ and $\efrank_1 T$ for the $1$-flattening rank and essential $1$-flattening rank of $T$ respectively.

\end{definition}

We note that the 1-flattening rank of T is the rank of the matrix $A: Q_1  \times \left( \prod_{\a=2}^d Q_{\a} \right) \rightarrow \mathbb{F}$ defined by \[ A(x_1,(x_2,\dots ,x_d)) = T(x_1,x_2,\dots ,x_d).\]

As we have not managed to find a simpler argument we prove the following two propositions by recognising them as special cases of Corollary \ref{Equivalence of all essential ranks}, a common generalisation of them. The proof of Corollary \ref{Equivalence of all essential ranks} only uses Proposition \ref{Essential equivalence}, which in turn has a self-contained proof. As Proposition \ref{Essential equivalence} will be used in full in Section \ref{section: General case}, its proof and that of Corollary \ref{Equivalence of all essential ranks} are deferred to that section. Proposition \ref{Equivalence between $efrank_1$ and $etr$, assuming that $epr$ is bounded for slices of all sizes} and Proposition \ref{A tensor with bounded essential partition ranks of all slices of all sizes has bounded essential tensor rank} are similar, except that the assumption that the $1$-flattening rank of $T$ is small in Proposition \ref{Equivalence between $efrank_1$ and $etr$, assuming that $epr$ is bounded for slices of all sizes} is replaced by the assumption that the essential partition rank of $T$ is small in Proposition \ref{A tensor with bounded essential partition ranks of all slices of all sizes has bounded essential tensor rank}.

\begin{proposition} \label{Equivalence between $efrank_1$ and $etr$, assuming that $epr$ is bounded for slices of all sizes}

Let $d \ge 2, l \ge d^2$ be integers, and let $T: \prod_{\a=1}^d Q_{\a} \rightarrow \mathbb{F}$ be an order-$d$ tensor. Assume that $\epr T_y \le l$ for each $I \subset \lbrack d \rbrack$ with $|I| \in \{1,\dots ,d-2\}$ and each $y \in \left(\prod_{\a \in I} Q_\a \right) \setminus E(I)$. If $\efrank_1 T \le l$ then $\etr T \le (4l^3)^{2^{d}}$. \end{proposition}

We shall also apply the following result to order-$(d-1)$ tensors in the proof of Proposition \ref{Disjoint $frank_1$ assuming all $epr$ of slices of all sizes are bounded}.

\begin{proposition} \label{A tensor with bounded essential partition ranks of all slices of all sizes has bounded essential tensor rank}

Let $d \ge 2, l \ge d^2$ be integers and let $T: \prod_{\a=1}^d Q_{\a} \rightarrow \mathbb{F}$ be an order-$d$ tensor. Assume that $\epr T \le l$ and $\epr T_y \le l$ for each $I \subset \lbrack d \rbrack$ with $|I| \in \{1,\dots ,d-2\}$ and each $y \in \left(\prod_{\a \in I} Q_{\a}\right) \setminus E(I)$. Then $\etr T \le (4l^3)^{2^d}$. \end{proposition}

Analogously to the order-$3$ case we shall use the following upper bound in the proof of Proposition \ref{Disjoint tensor rank}, that is, Theorem \ref{Disjoint rank subtensors theorem} in the case of the tensor rank (for all orders). We will, however, additionally use it in a moment in the proof of Proposition \ref{Disjoint $frank_1$ assuming all $epr$ of slices of all sizes are bounded}.

\begin{lemma} \label{dimension of columns at most tensor rank} Let $d \ge 2$ be an integer, and let $T: \prod_{\a=1}^d Q_{\a} \rightarrow \mathbb{F}$ be an order-$d$ tensor. Then \[\frank_1(T) = \dim \bigg\langle T_{x_2,\dots,x_d}: (x_2,\dots,x_d) \in \prod_{\a=2}^d Q_{\a} \bigg\rangle \le \tr T.\] \end{lemma}

\begin{proof}

The equality is immediate by definition of the $1$-flattening rank. We write $l$ for $\tr T$. By definition of the tensor rank, we may write a decomposition \[T(x_1,\dots,x_d) = \sum_{i=1}^l \prod_{\a=1}^d a_{\a,i}(x_\a)\] for some functions $a_{\a,i}: Q_\a \to \F$ with $\a \in [d]$ and $i \in [l]$. For each $(x_2,\dots,x_d) \in \prod_{\a=2}^d Q_\a$, this provides the expression \[T_{x_2,\dots,x_d}(x_1) = \sum_{i=1}^l \sigma_{i,x_2,\dots,x_d} a_i(x_1)\] for some $\sigma_{i,x_2,\dots,x_d} \in \F$. The functions $a_1,\dots,a_l$ therefore span all $T_{x_2,\dots,x_d}$ with $(y,z) \in \prod_{\a = 2}^d Q_\a$, concluding the proof. \end{proof}

In the following proposition and its proof we will write $Q_{2, \dots, d}$ for \[ \left(\prod_{\a=2}^d Q_{\a}\right) \setminus E(\{2, \dots, d\}). \]

\begin{proposition} \label{Disjoint $frank_1$ assuming all $epr$ of slices of all sizes are bounded} 

Let $d \ge 2, l \ge d^2, m \ge 1$ be integers, and let $T: \prod_{\a=1}^d Q_{\a} \rightarrow \mathbb{F}$ be an order-$d$ tensor such that $\epr T_y \le m$ for every $I \subset \lbrack d \rbrack$ with $|I| \in \{1,\dots ,d-2\}$ and each $y \in (\prod_{\a \in I} Q_{\a}) \setminus E(I)$. If $\efrank_1 T \ge d^2(4m^3)^{2^{d-1}} l$ then there exist pairwise disjoint $X_1 \subset Q_1,\dots ,X_d \subset Q_d$ such that $\frank_1 T(X_1 \times \dots \times X_d) \ge l$.

Suppose that  is satisfied for all functions $a_2,\dots,a_d: Q_{2, \dots, d} \rightarrow \F$.  \end{proposition}

\begin{proof}

Let $l$ be the largest integer such that there exist pairwise distinct $x_{\a,j} \in Q_{\a}$ for each $\a \in \lbrack d \rbrack, j \in \lbrack l \rbrack$ such that the restrictions \[ T_{x_{2,1},\dots ,x_{d,1}}(\{x_{1,1},\dots ,x_{1,l}\}),\dots ,T_{x_{2,l},\dots ,x_{d,l}}(\{x_{1,1},\dots ,x_{1,l}\}) \] are linearly independent. The pairwise disjoint sets $X_1=\{x_{1,1}, \dots, x_{1,l}\}, \dots, X_d=\{x_{d,1}, \dots, x_{d,l}\}$ then satisfy $\frank_1 T(X_1 \times \dots \times X_d) \ge l$. In the remainder of the proof we will show that there exist functions $a_2,\dots,a_d: Q_{2, \dots, d} \rightarrow \F$ such that \begin{equation} \dim \left\langle \left(T_{x_2,\dots ,x_d} + \sum_{\a=2}^d a_{\a}(x_2,\dots,x_d) 1_{x_1=x_\a}\right): (x_2,\dots ,x_d) \in Q_{2, \dots, d} \right\rangle \le d^2(4m^3)^{2^{d-1}} l, \label{vertical dimension bound for order-d tensors} \end{equation} which establishes $\efrank_1 T \le d^2(4m^3)^{2^{d-1}} l$. We define the following sets, which we will use throughout. \begin{align*}
W_{\a}&= Q_{\a} \setminus \{x_{\a,j}, \a \in \lbrack d \rbrack, j \in \lbrack l \rbrack \} \text{ for each } \a \in \lbrack d \rbrack \\
U_1&= \{(x_{2,1},\dots ,x_{d,1}),\dots ,(x_{2,l},\dots ,x_{d,l})\} \\
W_{2, \dots, d}&= \left(\prod_{\a = 2}^d W_{\a}\right) \setminus E(\{2, \dots, d\}). \end{align*}

We first show that we can define functions $a_2,\dots,a_d: W_{2, \dots, d} \rightarrow \F$ such that \[\dim \left\langle \left(T_{x_2,\dots ,x_d} + \sum_{\a=2}^d a_{\a}(x_2,\dots,x_d) 1_{x_1=x_{\a}}\right)(W_1): (x_2,\dots ,x_d) \in W_{2, \dots, d} \right\rangle \le l. \] In other words we want to show that the matrix $A: Q_1 \times Q_{2,\dots,d} \rightarrow \F$ defined by \[A(x_1,(x_2,\dots, x_d)) = T(x_1,x_2, \dots,x_d)\] is such that the rank of the restriction $A(W_1 \times W_{2,\dots,d})$ is at most $l$ after some modifications of the entries of the type \[ A(u,(u,x_3,\dots ,x_d)), A(u,(x_2,u,x_4,\dots ,x_d)),\dots ,A(u,(x_2,\dots ,x_{d-1},u)). \] We construct $a_2,\dots ,a_d$ as follows. Let $x_1 \in W_1$ and $(x_2,\dots ,x_d) \in W_{2, \dots, d}$. We distinguish two situations.

\begin{enumerate}

\item If $ x_1 \notin \{x_2,\dots ,x_d\}$ then by maximality of $ l$, the $ (l+1) \times (l+1)$ submatrix \[ A((X_1 \cup \{x_1\}) \times (U_1 \cup \{(x_2,\dots x_d)\})) \] has the same rank as the full rank matrix $ A(X_1 \times U_1)$, so we can write \[ A(x_1, (x_2,\dots ,x_d)) = A(\{x_1\} \times U_1)  A(X_1 \times U_1)^{-1} A(X_1 \times \{(x_2,\dots x_d)\}). \]

\item If $ x_1 = x_{\a}$ for some $\a \in \{2,\dots, d \}$ then because $A((X_1 \cup \{x_1 \}) \times (U_1 \cup \{(x_2,\dots x_d)\}))$ is an $(l+1) \times (l+1)$ matrix and $A(X_1 \times U_1)$ has full rank $l$, there exists a unique $a_{\a}(x_2,\dots,x_d) \in \F$ such that \[ A(x_1, (x_2,\dots,x_d)) + a_{\a}(x_2,\dots,x_d) = A(\{x_1\} \times U_1) A(X_1 \times U_1)^{-1} A(X_1 \times \{(x_2,\dots x_d)\}). \]

\end{enumerate}

The values $a_2(x_2, \dots, x_d), \dots, a_d(x_2, \dots, x_d)$ obtained in the second situation define functions $a_2,\dots,a_d: W_{2, \dots, d} \rightarrow \F$.  Let $B: (X_1 \cup W_1) \times (U_1 \cup W_{2, \dots, d}) \rightarrow \F$ be the matrix defined for all $(x_1,(x_2,\dots,x_d)) \in W_1 \times W_{2,\dots,d}$ by \[ B(x_1,(x_2,\dots,x_d)) = A(x_1,(x_2,\dots,x_d)) + \sum_{\a=2}^d a_{\a}(x_2,\dots,x_d) 1_{x_1=x_{\a}} \] and by $B(x_1,(x_2,\dots,x_d)) = A(x_1,(x_2,\dots,x_d))$ for all other $(x_1,(x_2,\dots,x_d))$. By construction of $a_2,\dots,a_d$ we have \[ B(W_1 \times W_{2, \dots, d}) = B(W_1 \times U_1) B(X_1 \times U_1)^{-1} B(X_1 \times W_{2, \dots, d}) \] and therefore $\rk B (W_1 \times W_{2, \dots, d}) \le l$. For each $\a \in \lbrack d \rbrack$, we define the slices \[ S_{\a,j} = \big(\{x_{\a} = j \} \setminus (Q_1 \times E(\{2, \dots, d\}))\big) \cap \{x_1 \in W_1 \} \cap \dots \cap \{x_{\a-1} \in W_{\a-1} \} \] for each $j \in Q_{\a} \setminus W_{\a}$, where we have written $\{x_{\a} = j\}$ for $\{x \in Q_1 \times \dots \times Q_d: x_{\a} = j\}$. Then the set $Q_1 \times Q_{2, \dots, d}$ can be written as the disjoint union \[ \bigg(W_1 \times W_{2, \dots, d}\bigg) \bigcup \left(\bigcup_{1 \le \alpha \le d} \bigcup_{j \in Q_{\alpha} \setminus W_{\alpha}} S_{\a,j}\right). \] By definition all the sets $Q_{\alpha} \setminus W_{\alpha}$ have size at most $dl$. Regardless of how we choose the functions $a_2,\dots,a_d: Q_{2,\dots,d} \to \F$, something that we are about to do, we have \[ \dim \left\langle \left(T_{x_2,\dots ,x_d} + \sum_{\a=2}^d a_{\a}(x_2,\dots,x_d) 1_{x_1=x_{\a}}\right)(Q_1 \setminus W_1): (x_2,\dots, x_d) \in Q_{2, \dots, d}\right\rangle \] to be at most $|Q_1 \setminus W_1| \le dl$. For each $\alpha \in \{2,\dots,d\}$ and each $j \in Q_{\a}\setminus W_{\a}$, by Proposition \ref{A tensor with bounded essential partition ranks of all slices of all sizes has bounded essential tensor rank} we have $\etr T_{x_{\a} = j}\left(\prod_{\a'\neq \a} (Q_{\a'} \setminus \{j\})\right) \le (4m^3)^{2^{d-1}}$, so using Remark \ref{dimension of columns at most tensor rank}, $a_2,\dots,a_d$ can be defined on $S_{\a,j}$ such that \begin{multline} \dim \bigg\langle \left(T_{x_2,\dots, x_d} + \sum_{\a=2}^d a_{\a}(x_2,\dots,x_d) 1_{x_1=x_{\a}}\right)(W_1): (x_2,\dots, x_d) \in Q_{2, \dots, d} ,\\ x_2 \in W_2, \dots, x_{\a - 1} \in W_{\a - 1}, x_{\alpha} = j \bigg\rangle \le (4m^3)^{2^{d-1}}. \nonumber \end{multline} We note that by construction, the successive choices of values of $(a_2,\dots,a_d)$ pertain to disjoint sets of $(d-1)$-tuples $(a_2,\dots,a_d)$, so are all compatible. The bound \eqref{vertical dimension bound for order-d tensors} follows by subadditivity.\end{proof}

Let $G_{d, \tr}'(l)$ be defined by $G_{2,\tr}'(l) = 3l$ and for each $d \ge 3$, \[G_{d, \tr}'(l) = (32(G'_{d-1, \tr}(l))^3)^{3 \cdot 4^d}.\] The following proposition establishes Theorem \ref{Disjoint rank subtensors theorem} in the case of the tensor rank.

\begin{proposition} \label{Disjoint tensor rank}

Let $d \ge 2$, $l \ge d^2$ be integers. If $T$ is an order-$d$ tensor such that $\etr T \ge G_{d,\tr}'(l)$, then $\dtr T \ge l$. \end{proposition}

\begin{proof} We prove the result by induction on $d$. The result holds for $d=2$, as that case is proved by Proposition \ref{Disjoint rank for matrices}. Let $d \ge 3$.

We distinguish two cases.
\medskip

\noindent \textbf{Case 1:} There exists $I \subset \lbrack d \rbrack$ with $|I| \in \{1,\dots ,d-2\}$ and $y \in \prod_{\a \in I} Q_{\a}$ with the $y_i$, $i \in I$ pairwise distinct and such that $\epr T_y \ge G'_{d-1,\tr}(l) + d^2$. Let $d'=|I|$. Without loss of generality we can assume that $I = \{1,\dots ,d'\}$. Hence \[ \epr T_y\left(\prod_{\a=d'+1}^d (Q_{\a} \setminus \{y_1,\dots ,y_{d'}\})\right) \ge G_{d-1,\tr}'(l) \ge G_{d',\tr}'(l) \] where the last inequality holds by monotonicity in $d$ of $G_{d,\tr}'$. By the inductive hypothesis applied to the previous tensor there exist pairwise disjoint sets $X_{\a} \subset Q_{\a} \setminus \{y_1,\dots ,y_{d'}\}$ for each $\a \in \{d'+1,\dots ,d\}$ such that \[ \tr T_y\left(\prod_{\a=d'+1}^d X_{\a}\right) \ge l. \] Letting $X_{\a}= \{y_{\a}\}$ for each $\a \in \{1,\dots ,d'\}$, the sets $X_1,\dots ,X_d$ are pairwise disjoint and \[ \tr T\left(\prod_{\a=1}^d X_{\a}\right) \ge l. \]

\noindent \textbf{Case 2:} For all $I \subset \lbrack d \rbrack$ with $|I| \in \{1,\dots ,d-2\}$ and $y \in \prod_{\a \in I} Q_{\a}$ with the $y_{\a}$, $\a \in I$ pairwise distinct we have $\epr T_y \le G'_{d-1, \tr}(l) + d^2$. Since \[\etr T \ge G'_{d,\tr}(l) = (32(G'_{d-1, \tr}(l))^3)^{3 \cdot4^d},\] applying Proposition \ref{Equivalence between $efrank_1$ and $etr$, assuming that $epr$ is bounded for slices of all sizes} we get $\efrank_1 T \ge l^2 (32(G'_{d-1, \tr}(l))^3)^{2^{d-1}}$. By Proposition \ref{Disjoint $frank_1$ assuming all $epr$ of slices of all sizes are bounded} there exist $X_1,\dots ,X_d$ pairwise disjoint such that \[\frank_1 T\left(\prod_{\a=1}^d X_{\a}\right) \ge l.\] Using Lemma \ref{dimension of columns at most tensor rank} we conclude \[\tr T\left(\prod_{\a=1}^d X_{\a}\right) \ge l. \qedhere\]  \end{proof}

\section{Proofs for the slice rank of order-3 tensors} \label{section: Slice rank for order-3 tensors}

\subsection{Proof for order-3 slice rank subtensors} \label{subsection: Slice rank subtensors for order-3 tensors}

We here prove Theorem \ref{Subtensors theorem} in the case of the slice rank of order-3 tensors. Our proof can be summarised as follows: given an order-$3$ tensor $T$ with high slice rank we distinguish two cases. In the first case, there exists a large set of slices $T_x$ which are separated, in the sense that every non-trivial linear combination of these slices has high rank. In this case, we can find sets $Y,Z$ such that the restrictions of these slices to $Y \times Z$ are still separated in the above sense, which suffices to guarantee a high rank subtensor. In the second case, there does not exist such a large separated set of slices. This allows us to decompose the tensor $T$ as the sum of some tensors $S^1, S^2, S^3$ with low slice rank and of a tensor $U$ for which the slice rank and tensor rank are equivalent in the sense that they are large simultaneously. We then conclude by applying to $U$ our subtensors result, Theorem \ref{Subtensors theorem}, for the tensor rank.

We begin by proving a lemma showing that having a large separated set of slices guarantees a high slice rank.

\begin{lemma} \label{Spanning lemma for order-3 slice rank subtensors} Let $T: Q_1 \times Q_2 \times Q_3 \rightarrow \mathbb{F}$ be an order-$3$ tensor, and $l \ge 1$ be an integer. If there exist $x_1 \dots,x_l \in Q_1$ such that \[ \rk \left(\sum_{i=1}^l a_i T_{x_i}\right) \ge l \] for every $(a_1, \dots,a_l) \in \F^l \setminus \{0\}$, then $\sr T \ge l$. \end{lemma}

\begin{proof} Assume for contradiction that $\sr T \le l-1$. Then there exist nonnegative integers $r,s,t$ with $r+s+t \le l-1$ and functions $a_i: Q_1 \rightarrow \mathbb{F}$, $b_i: Q_2 \times Q_3 \rightarrow \mathbb{F}$, $i \in \lbrack r \rbrack$, $c_j: Q_2 \rightarrow \mathbb{F}$, $d_j: Q_1 \times Q_3 \rightarrow \mathbb{F}$, $j \in \lbrack s \rbrack$, $e_k: Q_3 \rightarrow \mathbb{F}$, $ f_k: Q_1 \times Q_2 \rightarrow \mathbb{F}$, $k \in \lbrack t \rbrack$ such that \[ T(x,y,z) = \sum_{i=1}^r a_i(x) b_i(y,z) + \sum_{j=1}^s c_j(y) d_j(x,z) + \sum_{k=1}^t e_k(z) f_k(x,y) \] Let $x_1, \dots, x_l \in Q_1$. Then we can write \[ T_{x_h}(y,z) = \sum_{i=1}^r a_i(x_h) b_i(y,z) + \sum_{j=1}^s c_j(y) d_j(x_h,z) + \sum_{k=1}^t e_k(z) f_k(x_h,y) \] for each $h \in \lbrack l \rbrack$. Because $l \ge r+1$ there exists a function $a: Q_1 \rightarrow \mathbb{F}$ supported inside $\{x_1, \dots, x_l \}$ such that $a \neq 0$ but $\sum_{h=1}^l a(x_h) a_i(x_h) = 0$ for each $i \in \lbrack r \rbrack$, so we can write \[\sum_{h=1}^l a(x_h)T_{x_h}(y,z)  = \sum_{j=1}^s c_j(y) \left(\sum_{h=1}^l a(x_h) d_j(x_h,z)\right) + \sum_{k=1}^t e_k(z) \left(\sum_{h=1}^l a(x_h) f_k(x_h,y)\right).\] The right-hand side has rank at most $s+t \le l-1$, a contradiction.\end{proof}

We next show a partial converse to the inequality $\sr T \le \tr T$ which holds in the situation where all slices of $T$ of all three kinds have bounded rank.

\begin{proposition} \label{Equivalence between slice rank and tensor rank}

Let $T: Q_1 \times Q_2 \times Q_3 \rightarrow \mathbb{F}$ be an order-$3$ tensor. Let $m \ge 1$ be an integer. Assume that for all $x \in Q_1, y \in Q_2, z \in Q_3$ we have $\rk T_x, \rk T_y, \rk T_z \le m$.
Then $\tr T \le m (\sr T)^2$. \end{proposition}

\begin{proof}There exists a decomposition \begin{equation} T(x,y,z) = \sum_{i=1}^r a_i(x) b_i(y,z) + \sum_{j=1}^s c_j(y) d_j(x,z) + \sum_{k=1}^t e_k(z) f_k(x,y) \label{decomposition for 3D slice rank} \end{equation} of $ T$, with $r,s,t$ nonnegative integers such that $r+s+t = \sr T$. The family $\{a_1, \dots, a_r\}$ is necessarily linearly independent. By Gaussian elimination there exists a set $X \subset Q_1$ of size at most $ r$ and for each $i \in \lbrack r \rbrack$ a function $a_i^*: Q_1 \rightarrow \mathbb{F}$ supported inside $X$ such that $ a_i^*.a_{i'} = \delta_{i,i'}$ for all $i,i' \in \lbrack r \rbrack$. For a given $i \in \lbrack r \rbrack$, taking the dot product with $ a_i^*$ in both sides of \eqref{decomposition for 3D slice rank} we get \[ b_i(y,z) = (a_i^*.T)(y,z) - \sum_{j=1}^s c_j(y)(a_i^*.d_j)(z) - \sum_{k=1}^t e_k(z)(a_i^*.f_k)(y). \] Because $\rk T_x \le m$ for each $x \in Q_1$ and $a_i^*$ is supported in a set of size at most $ r$ we get by subadditivity of the rank that $\rk (a_i^*.T) \le mr$ and hence $ \rk b_i \le mr + s + t$. Similarly $ \rk d_j \le ms + r + t$ for every $ j \in \lbrack s \rbrack$ and $ \rk f_k \le mt + r + s$ for every $k \in \lbrack t \rbrack$. By subadditivity \[ \tr T \le r(mr + s + t) + s(ms + r + t) + t(mt + r + s) \le m(\sr T)^2. \qedhere\] \end{proof}

We are now ready to prove Theorem \ref{Subtensors theorem} for the slice rank of order-$3$ tensors.

\begin{proposition} \label{Order-3 slice rank subtensors} Let $T: Q_1 \times Q_2 \times Q_3 \rightarrow \mathbb{F}$ be an order-$3$ tensor, and let $l \ge 1$ be an integer. If $\sr T \ge 51 l^3$ then there exist $X \subset Q_1$, $Y \subset Q_2$, $Z \subset Q_3$ with size at most $48l^3$ such that $\sr T(X \times Y \times Z) \ge l$. \end{proposition}

\begin{proof} Assume $\sr T \ge 51l^3$. We distinguish two cases.

\medskip
\noindent \textbf{Case 1}: For at least one of the three coordinates $x,y,z$, which without loss of generality we can take to be the first coordinate $x$, there exist $x_1,\dots ,x_l \in Q_1$ such that \[ \rk \left(\sum_{i=1}^l a_i T_{x_i}\right) \ge l \] for every $(a_1,\dots,a_l) \in \mathbb{F}^l \setminus \{0\}$.
By multi-tensor order-$2$ subtensors (Proposition \ref{Multi-tensor matrix rank}) there are sets $Y \subset Q_2$, $Z \subset Q_3$ with size at most $l^2$ such that \[ \rk \left(\sum_{i=1}^l a_i T_{x_i}\right)(Y \times Z) \ge l \] for every $(a_1,\dots,a_l) \in \mathbb{F}^l \setminus \{0\}$. By Lemma \ref{Spanning lemma for order-3 slice rank subtensors}, taking $X = \{x_1,\dots ,x_l\}$ we get $\sr T(X \times Y \times Z) \ge l$.

\noindent \textbf{Case 2}: We are not in Case 1. Then we construct a decomposition \[ T = S^1 + S^2 + S^3 + U \] with $S^1,S^2,S^3,U$ order-$3$ tensors $Q_1 \times Q_2 \times Q_3 \rightarrow \mathbb{F}$ as follows. Consider a maximal family $\{x_1,\dots,x_r\}$ of elements of $Q_1$ such that \[ \rk \left(\sum_{i=1}^l a_i T_{x_i}\right) \ge l \] for every $(a_1,\dots,a_r) \in \F^r \setminus \{0\}$. Because we are not in Case 1, we have $r \le l$. For every $x \in Q_1$, by maximality of $\{x_1,\dots,x_r\}$, there exist coefficients $a_1(x),\dots,a_{r}(x) \in \mathbb{F}$ satisfying \[ \rk \left( T_x - \sum_{i=1}^{r} a_i(x) T_{x_i} \right) \le l. \] The tensor $S^1$ defined by $S^1(x,y,z) = \sum_{i=1}^{r} a_i(x) T_{x_i}(y,z)$ hence satisfies $\rk (T_x - S^1_x) \le l$ for every $x \in Q_1$. Similarly we can find $s, t \le l$, $y_1,\dots,y_{s} \in Q_2$, $z_1,\dots,z_t \in Q_3$ and functions $c_1,\dots,c_s: Q_2 \rightarrow \mathbb{F}$, $e_1,\dots,e_t: Q_3 \rightarrow \mathbb{F}$ such that \[ \rk \left( T_y - \sum_{j=1}^{s} c_j(y) T_{y_j} \right) \le l \] for every $y \in Q_2$ and \[ \rk \left( T_z - \sum_{k=1}^{t} e_k(z) T_{z_k} \right) \le l \] for every $z \in Q_3$. We define the tensors $S^2$ and $S^3$ by \[ S^2(x,y,z) = \sum_{j=1}^{s} c_j(y) T_{y_j}(x,z) \text{ and } S^3(x,y,z) = \sum_{k=1}^{t} e_k(z) T_{z_k}(x,y).\] Let $U= T - (S^1 + S^2 + S^3)$. For each $x \in Q_1$, $\rk (T_x - S^1_x) \le l$, and $\rk S^2_x, \rk S^3_x \le l$, so by subadditivity, $\rk U_x \le 3l$. Similarly for each $y \in Q_2$, $\rk U_y \le 3l$ and for each $z \in Q_3$, $\rk U_z \le 3l$. Since $\sr T \ge (3l)(4l)^2 + 3l$ and $\sr S^1, \sr S^2, \sr S^3 \le l$, by subadditivity $\sr U \ge (3l)(4l)^2$. Since $\tr U \ge \sr U$, we have $\tr U \ge (3l)(4l)^2$. By Proposition \ref{Subtensors for order-$d$ tensor rank} in the order-3 case we can find $X \subset Q_1$, $Y \subset Q_2$, $Z \subset Q_3$ each with size at most $(3l)(4l)^2$ such that $\tr U(X \times Y \times Z) \ge (3l)(4l)^2$. As taking subtensors cannot increase the rank, it is still the case that $\rk U(X \times Y \times Z)_x, \rk U(X \times Y \times Z)_y, \rk U(X \times Y \times Z)_z \le 3l$ for all $x \in X$, $y \in Y$, $z \in Z$. By Proposition \ref{Equivalence between slice rank and tensor rank}, $\sr U(X \times Y \times Z) \ge 4l$, and hence (since $\sr (S^1+S^2+S^3)(X \times Y \times Z) \le 3l$) we conclude that $\sr T(X \times Y \times Z) \ge l$. \end{proof}

\subsection{Proof for order-3 disjoint slice rank} \label{subsection:Proof of order-3 disjoint slice rank}

The previous proof can be adapted to a proof for order-$3$ disjoint slice rank, using the same Lemma \ref{Spanning lemma for order-3 slice rank subtensors} and Proposition \ref{Equivalence between slice rank and tensor rank} as the previous proof did.

\begin{proposition} \label{Order-3 disjoint slice rank} Let $T: Q_1 \times Q_2 \times Q_3 \rightarrow \mathbb{F}$ be an order-$3$ tensor, and let $l \ge 1$ be an integer. If $\esr T \ge 17500 ((l^2+l+5)l(4l)^2)^3 + 3l$ then there exist $X \subset Q_1$, $Y \subset Q_2$, $Z \subset Q_3$ pairwise disjoint such that $\sr T(X \times Y \times Z) \ge l$. \end{proposition}

\begin{proof} Assume that $\esr T \ge 17500 ((l^2+l+5)l(4l)^2)^3 + 3l$. We distinguish two cases.
\medskip

\noindent \textbf{Case 1}: For at least one of the three coordinates $x,y,z$ which without loss of generality we can take to be the first coordinate $x$, there exist $x_1,\dots,x_l \in Q_1$ such that \[ \erk \left( \sum_{i=1}^l a_i T_{x_i}\right) \ge (l^2+l+3)l  \] for every $(a_1,\dots,a_l) \in \mathbb{F}^l \setminus \{0\}$. Using that a single row or column has rank at most $1$, by subadditivity \[ \erk \left(\sum_{i=1}^l a_i T_{x_i}\right)\big((Q_2 \setminus \{x_1,\dots,x_l\}) \times (Q_3 \setminus \{x_1,\dots,x_l\})\big) \ge (l^2+l+1)l \] so by Proposition \ref{Multi-tensor matrix disjoint rank} there exist disjoint subsets $Y \subset Q_2 \setminus \{x_1,\dots ,x_l\}$, $Z \subset Q_3 \setminus \{x_1,\dots ,x_l\}$ such that \[ \rk \left(\sum_{i=1}^l a_i T_{x_i}\right)(Y \times Z) \ge l \] for every $(a_1,\dots,a_l) \in \mathbb{F}^l \setminus \{0\}$. Taking $X = \{x_1,\dots ,x_l\}$, the sets $X,Y,Z$ are pairwise disjoint and by Lemma \ref{Spanning lemma for order-3 slice rank subtensors} we get $\sr T(X \times Y \times Z) \ge l$.
\medskip

\noindent \textbf{Case 2}: We are not in Case 1. Then we construct a decomposition \[ T = S^1 + S^2 + S^3 + U \] as follows. Because we are not in Case 1, there exists $r \le l$ and $x_1,\dots ,x_r \in Q_1$ such that for every $x \in Q_1$ there exist coefficients $a_1(x),\dots ,a_{r}(x)$ such that \[ \erk \left( T_x - \sum_{i=1}^{r} a_i(x) T_{x_i} \right) \le (l^2+l+3)l. \] The tensor $S^1$ defined by $S^1(x,y,z) = \sum_{i=1}^{r} a_i(x) T_{x_i}(y,z)$ therefore satisfies $\erk (T_x - S^1_x) \le (l^2+l+3)l$ for every $x \in Q_1$. Similarly we can find $s, t \le l$, $y_1,\dots ,y_{s} \in Q_2$, $z_1,\dots ,z_t \in Q_3$ and functions $c_1,\dots ,c_s: \mathbb{F}^{Q_2} \rightarrow \mathbb{F}$, $e_1,\dots ,e_t: \mathbb{F}^{Q_3} \rightarrow \mathbb{F}$ such that \[ \erk \left(T_y - \sum_{j=1}^{s} c_j(y) T_{y_j} \right) \le (l^2+l+3)l \] for every $y \in Q_2$ and \[ \erk \left( T_z - \sum_{k=1}^{t} e_k(z) T_{z_k} \right) \le (l^2+l+3)l \] for every $z \in Q_3$. We define the tensors $S^2$ and $S^3$ by \[ S^2(x,y,z) = \sum_{j=1}^{s} c_j(y) T_{y_j}(x,z)\text{ and }S^3(x,y,z) = \sum_{k=1}^{t} e_k(z) T_{z_k}(x,y). \] Let $U= T - (S^1 + S^2 + S^3)$. For each $x \in Q_1$, $\erk (T_x - S^1_x) \le (l^2+l+3)l$ and $\rk S^2_x$, $\rk S^3_x \le l$, so by subadditivity, $\erk U_x \le (l^2+l+5)l$. Similarly for each $y \in Q_2$, $\erk U_y \le (l^2+l+5)l$ and for each $z \in Q_3$, $\erk U_z \le (l^2+l+5)l$. Since \[\esr T \ge 17500 ((l^2+l+5)l(4l)^2)^3 + 3l\] and $\sr S^1, \sr S^2, \sr S^3 \le l$, by subadditivity and then using that the essential slice rank is at most the essential tensor rank we obtain \[\etr U \ge \esr U \ge 17500 ((l^2+l+5)l(4l)^2)^3 \cdot\] By Proposition \ref{Disjoint rank for order-3 tensors} we can find $X \subset Q_1$, $Y \subset Q_2$, $Z \subset Q_3$ pairwise disjoint such that $\tr U(X \times Y \times Z) \ge (l^2+l+5)l(4l)^2$. As taking subtensors cannot decrease the essential rank, we still have $\rk U(X \times Y \times Z)_x, \rk U(X \times Y \times Z)_y, \rk U(X \times Y \times Z)_z \le (l^2+l+5)l$ for all $x \in X$, $y \in Y$, $z \in Z$. By Proposition \ref{Equivalence between slice rank and tensor rank}, $\sr U(X \times Y \times Z) \ge 4l$, and hence (since $\sr (S^1+S^2+S^3)(X \times Y \times Z) \le \sr (S^1+S^2+S^3) \le 3l$) we conclude $\sr U(X \times Y \times Z) \ge l$. \end{proof}

\section{Deducing subtensors and disjoint rank for several tensors from the one-tensor case} \label{section: Multi-tensor statements}

Throughout this section, we fix $d \ge 2$ a positive integer and $R$ a non-empty family of partitions of $\lbrack d \rbrack$. We assume that Theorem \ref{Subtensors theorem} and Theorem \ref{Disjoint rank subtensors theorem} hold for this choice of pair $(d,R)$ and deduce from them generalisations to several tensors for this same choice of pair $(d,R)$. In Section \ref{section: General case} we will use these multi-tensor generalisations as an important part of the inductive argument that proves Theorem \ref{Subtensors theorem} and Theorem \ref{Disjoint rank subtensors theorem}.

\subsection{Multi-tensor subtensors}

Let $d \ge 2$ be an integer and let $R$ be a non-empty family of partitions of $\lbrack d \rbrack$. For every positive integer $s \ge 1$, we define the functions $F_{d,R,s}: \mathbb{N} \rightarrow \mathbb{N}$ and $G_{d,R,s}: \mathbb{N} \rightarrow \mathbb{N}$ by \begin{align*} 
F_{d,R,1}(l) &= F_{d,R}(l) \textit{ and for each } s \ge 2 \textit{, }F_{d,R,s}(l) = F_{d,R}(sl) + F_{d,R,s-1}(l)\\
G_{d,R,1}(l) &= G_{d,R}(l) \textit{ and for each } s \ge 2 \textit{, }G_{d,R,s}(l) = G_{d,R}(sl).
\end{align*}

\begin{proposition} \label{Multi-tensor Rrk subtensors} Let $d \ge 2, s \ge 1, l \ge 1$ be integers and let $R$ be a non-empty family of partitions of $\lbrack d \rbrack$. Assume that Theorem \ref{Subtensors theorem} holds for $(d,R)$. If $T_1,\dots ,T_s: \prod_{\a=1}^d Q_{\a} \rightarrow \mathbb{F}$ are order-$d$ tensors such that \[ \Rrk \left(\sum_{i=1}^s a_i T_i\right) \ge G_{d,R,s}(l) \] for every $(a_1,\dots ,a_s) \in \F^s \setminus \{0\}$, then there exist $X_1 \subset Q_1,\dots ,X_d \subset Q_d$, each with size at most $F_{d,R,s}(l)$, such that \[ \Rrk \left(\sum_{i=1}^s a_i T_i\right)\left(\prod_{\a=1}^d X_{\a}\right) \ge l \] for every $(a_1,\dots ,a_s) \in \F^s \setminus \{0\}$.

\end{proposition}

\begin{proof} We proceed by induction on $s$. If $s=1$ then the result holds by the one-tensor case. Let $s \ge 2$ and assume that the result holds for $s-1$. By the one-tensor case there exist $X_1' \subset Q_1$, $\dots$, $X_d' \subset Q_d$ all with size at most $F_{d,R}(sl)$ such that \begin{equation} \Rrk T_1(X_1' \times \dots \times X_d') \ge sl. \label{At least sl again} \end{equation} If there existed $s$ linearly independent values of $a \in \F^s$ for which \[ \Rrk (a.T)(X_1' \times \dots \times X_d') < l \] then, as these $s$ values of $a$ would constitute a basis of $\F^s$, then this would contradict \eqref{At least sl again} by subadditivity. Therefore, there exists a strict linear subspace $W$ of $\mathbb{F}^s$ such that \[ \Rrk (a.T)(X_1' \times \dots  \times X_d') \ge l \] for every $a \in \F^s \setminus W$. Applying Proposition \ref{Multi-tensor Rrk subtensors} for $s-1$ and $W \setminus \{0\}$, there exist $X_1'' \subset Q_1, \dots, X_d'' \subset Q_d$ all with size at most $F_{d,R,s-1}(l)$ such that \[ \Rrk (a.T)(X_1'' \times \dots  \times X_d'') \ge l \] for every $a \in W \setminus \{0\}$. The sets $X_1, \dots X_d$ defined by $X_1= X_1' \cup X_1'', \dots, X_d= X_d' \cup X_d''$ all have size at most $F_{d,R}(sl) + F_{d,R,s-1}(l)$, and furthermore \[ \Rrk (a.T)(X_1 \times \dots  \times X_d) \ge l \] for every $a \in \F^s \setminus \{0\}$.\end{proof}

\subsection{Multi-tensor disjoint rank}

Is it the case that for every $l \ge 1$ there exists $g(l)$ such that if $T_1$ and $T_2$ are two order-$3$ tensors such that for all $(a,b) \in \mathbb{F}^2 \setminus \{0\}$, $\etr (a_1 T_1 + a_2 T_2) \ge g(l)$ then there exist $X,Y,Z$ pairwise disjoint such that for all $(a,b) \in \mathbb{F}^2 \setminus \{0\}$, $\tr (a_1 T_1 + a_2 T_2)(X \times Y \times Z) \ge l$? Without an additional assumption this is false, as the following counterexample shows.

\begin{example}\label{No multi-tensor disjoint tensor rank for 3D} Let $k \ge 1$ be an integer, $b_1: Q_2 \times Q_3 \rightarrow \F$, $b_2: Q_1 \times Q_3 \rightarrow \F$, be matrices with $\erk b_1, \erk b_2 \ge k+3$ and let $T_1$, $T_2$ be the order-$3$ tensors defined by $ T_1(x,y,z) = 1_{x=1} b_1(y,z)$, $ T_2(x,y,z) = 1_{y=1} b_2(x,z)$. For any $(a_1,a_2) \in \mathbb{F}^2 \setminus \{0\}$, assuming without loss of generality that $a_1 \neq 0$ we have \begin{align*}
\etr (a_1 T_1 + a_2 T_2) & \ge \erk (a_1 T_1 + a_2 T_2)_{x=1}((Q_2 \setminus \{1\}) \times (Q_3 \setminus \{1\}))\\
& \ge \erk (a_1 T_1 + a_2 T_2)_{x=1} - 2\\
&\ge \erk (b_1) - 3\\
& \ge k. \end{align*} where the third inequality comes from the fact that the order-$2$ slice $(T_2)_{x=1}$ has support contained inside the row $y=1$. On the other hand whenever $ X \subset Q_1, Y \subset Q_2, Z \subset Q_3$ are pairwise disjoint the element $ 1$ cannot be contained in both of the sets $ X$ and $ Y$. If $ 1$ is outside $ X$, then $ T_1(X \times Y \times Z) = 0$ and similarly if $ 1$ is outside $ Y$, then $ T_2(X \times Y \times Z) = 0$, so there exists some $ (a_1,a_2) \in \mathbb{F}^2 \setminus \{0\}$ such that $ (a_1T_1 + a_2T_2)(X \times Y \times Z) = 0$. \end{example}

This kind of counterexample does not show up if we consider matrices instead of order-$3$ tensors, because for matrices it is not possible to fit a  matrix of arbitrarily large essential rank into a single row or column, whereas it is possible to fit an order-$3$ tensor of arbitrarily large essential tensor rank into a single order-$2$ slice. However, the statement becomes true if we require in addition that the high essential rank assumption on the linear combinations $a.T$ holds even after removing sufficiently many order-$(d-1)$ slices of each of the $d$ kinds.

Let $d \ge 2$ be an integer and let $R$ be a non-empty family of partitions of $\lbrack d \rbrack$. For every positive integer $s \ge 1$, we define the functions $H_{d,R,s}: \mathbb{N} \rightarrow \mathbb{N}$ and $G_{d,R,s}': \mathbb{N} \rightarrow \mathbb{N}$ by \begin{align*} 
H_{d,R,1}(l) &= 0 \text{ and for each }s \ge 2 \text{, } H_{d,R,s}(l) = dF_{d,R}(sl) + H_{d,R,s-1}(l)\\
G_{d,R,1}''(l) &= G_{d,R}'(l) \text{ and for each }s \ge 2 \text{, } G_{d,R,s}''(l) = G_{d,R}''(G_{d,R}(sl)).
\end{align*} For each $d \ge 2$, we can deduce multi-tensor order-$d$ disjoint $R$-rank from one-tensor order-$d$ disjoint $R$-rank and one-tensor order-$d$ $R$-rank subtensors.

\begin{proposition} \label{Multi-tensor disjoint $Rrk$} Let $d \ge 2, s \ge 1, l \ge 1$ be integers, let $R$ be a non-empty family of partitions of $\lbrack d \rbrack$. Assume that Theorem \ref{Disjoint rank subtensors theorem} holds for $(d,R)$. Let $T_1,\dots ,T_s: \prod_{\a=1}^d Q_{\a} \rightarrow \mathbb{F}$ be order-$d$ tensors and suppose that for all $Y_1 \subset Q_1$, \dots, $Y_d \subset Q_d$ all with size at most $H_{d,R,s}(l)$ we have \[ \eRrk (a.T)\left(\prod_{\a=1}^d (Q_{\a} \setminus Y_{\a})\right) \ge G{d,R,s}''(l) \] for every $(a_1,\dots ,a_s) \in \F^s \setminus \{0\}$. Then there exist $X_1 \subset Q_1$, $\dots$, $X_d \subset Q_d$ pairwise disjoint such that \[ \Rrk (a.T)\left(\prod_{\a=1}^d X_{\a}\right) \ge l \]  for every $(a_1,\dots ,a_s) \in \F^s \setminus \{0\}$. \end{proposition}

\begin{proof} We proceed by induction on $s$. If $s = 1$ then the result holds by the one-tensor case. Let $s \ge 2$ and assume that the result holds for $s-1$. By the one-tensor case of Theorem \ref{Disjoint rank subtensors theorem} there exist $X_1^0,\dots ,X_d^0$ pairwise disjoint such that \[ \Rrk T_1\left(\prod_{\a=1}^d X_{\a}^0 \right) \ge G_{d,R}(sl).\] By the one-tensor case of Theorem \ref{Subtensors theorem} we can find subsets $X_{\a}'$ of $X_{\a}^0$ for each $\a=1,\dots ,d$ with size at most $F_{d,R}(sl)$ such that \[ \Rrk T_1\left(\prod_{\a=1}^d X_{\a}'\right) \ge sl. \] By subadditivity of the $R$-rank, there exists a strict linear subspace $W$ of $\mathbb{F}^s$ such that \[ \Rrk (a.T)\left(\prod_{\a=1}^d X_{\a}'\right) \ge l \] for every $a \in \F^s \setminus W$. Taking $Z = \bigcup_{1 \le \a' \le d} X_{\a'}'$ and applying the current proposition for $s-1$ with $Y_{\a} \cup Z$ for each $\a=1,\dots, d$ there exist subsets $X_{\a}'' \subset Q_{\a} \setminus Y_{\a}$ with the $X_{\a}''$ all pairwise disjoint and such that \[ \Rrk (a.T)\left(\prod_{\a=1}^d X_{\a}''\right) \ge l \] for every $a \in W \setminus \{0\}$. Letting $X_{\a} = X_{\a}' \cup X_{\a}''$ for each $\a = 1,\dots ,d$ the sets $X_{\a}$ are pairwise disjoint, and \[ \Rrk (a.T)\left(\prod_{\a=1}^d X_{\a}\right) \ge l \] for every $a \in \F^s \setminus \{0\}$ as desired.\end{proof}

We note that Example \ref{No multi-tensor disjoint tensor rank for 3D} occurs neither for the slice rank nor for the partition rank: this is because for these notions of rank for order-$d$ tensors, every tensor supported inside an order-$(d-1)$ slice has rank at most $1$. As a result, Proposition \ref{Multi-tensor disjoint $Rrk$} specialises to the following more concise multi-tensor version for the partition rank, which will be the only multi-tensor version of Theorem \ref{Disjoint rank subtensors theorem} that we will use throughout the remainder of the paper. Let $G_{d,\pr,s}' = G_{d,\pr,s}'' + dH_{d,\pr,s}$.

\begin{proposition} \label{Multi-tensor disjoint $pr$} 

Let $d \ge 2, s \ge 1, l \ge 1$ be integers. Assume that Theorem \ref{Disjoint rank subtensors theorem} holds for $(d,R_{\pr})$. Let $T_1,\dots ,T_s: \prod_{\a=1}^d Q_{\a} \rightarrow \mathbb{F}$ be order-$d$ tensors and suppose that \[ \epr (a.T) \ge G'_{d,\pr,s}(l) \] for every $(a_1,\dots ,a_s) \in \F^s \setminus \{0\}$. Then there exist $X_1 \subset Q_1$, $\dots$, $X_d \subset Q_d$ pairwise disjoint such that \[ \pr (a.T)\left(\prod_{\a=1}^d X_{\a}\right) \ge l \] for every $(a_1,\dots,a_s) \in \F^s \setminus \{0\}$. 

\end{proposition}

\begin{proof}
Because each order-$(d-1)$ slice of an order-$d$ tensor has partition rank at most $1$, by our choice of function $G_{d,\pr,s}'$ guarantees that the assumption of Proposition \ref{Multi-tensor disjoint $pr$} implies the assumption of Proposition \ref{Multi-tensor disjoint $Rrk$}. Applying Proposition \ref{Multi-tensor disjoint $Rrk$} then concludes the proof. \end{proof}

\section{Extending the proof for order-3 slice rank subtensors to order-4 partition rank subtensors} \label{section: Additional difficulties in order 4 and higher}

In this section we adapt the proof of Proposition \ref{Order-3 slice rank subtensors} to a proof for order-$4$ partition rank subtensors. Compared with the proof of Proposition \ref{Order-3 slice rank subtensors} the main novelty in the proofs will be how to navigate between the different notions of rank, but many of the central ideas will be the same as for the proof of Proposition \ref{Order-3 slice rank subtensors}. The following notions of rank will be relevant to the argument. We will write
\begin{align*}
\pr T & \text{ for the order-4 partition rank of } T \\
\pr_{(2,2)} T & \text{ for} \Rrk T \text{ with } R = \{\{\{1,2\},\{3,4\}\},\{\{1,3\},\{2,4\}\},\{\{1,4\},\{2,3\}\}\} \\
(1 \times \sr) T & \text{ for} \Rrk T \text{ with } R = \{\{\{1\},\{2\},\{3,4\}\},\{\{1\},\{3\},\{2,4\}\}, \{\{1\},\{4\},\{2,3\}\}. 
\end{align*}We will refer to the second notion as the \emph{$(2,2)$-partition rank} and to the third notion as the \emph{$1$-enhanced slice rank}.

The overall structure of the proof will be as follows: starting from any of the first two previously listed notions of $R$-rank, either we can identify a simple lower-dimensional structure that guarantees high $R$-rank, in which case we can conclude by Proposition \ref{Multi-tensor Rrk subtensors} applied to slices of strictly smaller order, or we cannot, in which case the task at hand reduces to proving subtensors for the notion of rank listed immediately thereafter. If we reduce twice and end up considering the $1$-enhanced slice rank then we will able to conclude in rather short order due to the particular form of $R$ in this case and the fact that we already know how to prove Theorem \ref{Subtensors theorem} for the slice rank of order-$3$ tensors.

We begin by deducing Theorem \ref{Subtensors theorem} for the order-$4$ partition rank from Proposition \ref{Multi-tensor Rrk subtensors} for the order-$3$ slice rank and from Theorem \ref{Subtensors theorem} for the order-$4$ $(2,2)$-partition rank. The proof will distinguish two cases depending on the existence of a large set of order-$3$ slices that is sufficiently separated for the notion of order-3 slice rank. The first lemma that we will use is an analogue of Lemma \ref{Spanning lemma for order-3 slice rank subtensors}.

\begin{lemma} \label{Separation lemma for 4D pr} Let $T: Q_1 \times Q_2 \times Q_3 \times Q_4 \rightarrow \mathbb{F}$ be an order-$4$ tensor, and let $l \ge 1$ be an integer. Suppose that there exist $x_1,\dots ,x_l \in Q_1$ such that \[ \sr \left(\sum_{i=1}^l a_i T_{x_i} \right) \ge l \] for every $a \in \mathbb{F}^l \setminus \{0\}$. Then $\pr T \ge l$. \end{lemma}

\begin{proof} Assume for contradiction that $\pr T \le l-1$. Then there exist nonnegative integers $r_1,r_2,r_3,r_3,r'$ with $r_1+r_2+r_3+r_4+r' = \pr T$ such that we can write a decomposition \begin{align} T(x,y,z,w) & = \sum_{i=1}^{r_1} a_{1,i}(x) b_{1,i}(y,z,w) + \sum_{i=1}^{r_2} a_{2,i}(y) b_{2,i}(x,z,w) \nonumber \\
& + \sum_{i=1}^{r_3} a_{3,i}(z) b_{3,i}(x,y,w) + \sum_{i=1}^{r_4} a_{4,i}(w) b_{4,i}(x,y,z) + U(x,y,z,w) \label{partition rank equation} \end{align} for some order-$4$ tensor $U$ with $\pr_{(2,2)} U = r'$. Since $r_1 \le l-1$ there exists a non-zero function $a: Q_1 \rightarrow \F$ supported inside $\{x_1,\dots ,x_l \}$ such that $a.a_{1,i} = 0$ for each $i \in \lbrack r_1 \rbrack$. Taking the dot product with $a$ in both sides of the decomposition \eqref{partition rank equation} we obtain $\sr a.T \le l-1$, a contradiction. \end{proof}

Our next statement is an analogue of Proposition \ref{Equivalence between slice rank and tensor rank}. It provides us with an equivalence between the partition rank and the (2,2)-partition rank of an order-4 tensor under the assumption that all order-3 slices of the tensor have bounded slice rank.

\begin{proposition} \label{Equivalence between order-$4$ $pr$ and order-$4$ $(2-2)-pr$} Let $l,m \ge 1$ be integers, let $T: Q_1 \times Q_2 \times Q_3 \times Q_4 \rightarrow \mathbb{F}$ be an order-$4$ tensor, and suppose that $\sr T_x \le m$, $\sr T_y \le m$, $\sr T_z \le m$, $\sr T_w \le m$ for each $x \in Q_1$, $y \in Q_2$, $z \in Q_3$, $w \in Q_4$. If $\pr T \le l$, then $\pr_{(2,2)} T \le ml^2$. \end{proposition}

\begin{proof} Starting from the decomposition \eqref{partition rank equation} we can find a subset $X \subset Q_1$ with size at most $r_1$ and functions $a_{1,i}^*: Q_1 \rightarrow \F$ supported inside $X$ such that $a_{1,i}^*.a_i = \delta_{i,i'}$ for all $i,i' \in \lbrack r \rbrack$. For a fixed $i \in \lbrack r \rbrack$, taking the dot product with $a_i^*$ in both sides of \eqref{partition rank equation} yields the bound $\sr b_{1,i} \le mr_1+r_2+r_3+r_4+r'$. We obtain similar bounds for the quantities $\sr b_{2,i}, \sr b_{3,i}, \sr b_{4,i}$. Substituting slice rank decompositions of the functions $b_{1,i}, b_{2,i}, b_{3,i}, b_{4,i}$ with smallest length in \eqref{partition rank equation} provides a decomposition in which the first four sums consist in sums of products of three functions, one in two variables and two in one variable, which may in particular be written as products of two functions in two variables. We conclude the desired inequality by subadditivity of the $(2,2)$-partition rank. \end{proof}

We are now ready to show that to obtain order-$4$ partition rank subtensors it suffices to show order-$4$ $(2,2)$-partition rank subtensors together with multi-tensor order-$3$ slice rank subtensors.

\begin{proposition} \label{Up to pr} If Theorem \ref{Subtensors theorem} holds for $(d,R) = (4,\pr_{(2,2)})$ and Proposition \ref{Multi-tensor Rrk subtensors} holds for $(d,R) =  (3, \sr)$ then Theorem \ref{Subtensors theorem} holds for $(d,R) = (4,\pr)$ with \[F_{4,\pr}(l) = \max\big(F_{3,\sr,l}(l), F_{4,\pr_{(2,2)}}((G_{3,\sr,l}(l)+4l)(5l)^2)\big)\] \[G_{4,\pr}(l) = G_{4,\pr_{(2,2)}}((G_{3,\sr,l}(l)+4l)(5l)^2) + 4l.\] \end{proposition}

\begin{proof}
Let $T$ be an order-$4$ tensor with $\pr T \ge G_{4,\pr_{(2,2)}}((G_{3,\sr,l}(l)+4l)(5l)^2) + 4l$. We distinguish two cases.
\medskip

\noindent \textbf{Case 1}: For one of the possible four coordinates, which without loss of generality we can assume to be the first coordinate, there exist $x_1,\dots ,x_l \in Q_1$ such that \[ \pr \left(\sum_{i=1}^l b_i T_{x_i}\right) \ge G_{3,\sr,l}(l) \] for every $b \in \mathbb{F}^l \setminus \{0\}$. Then by Proposition \ref{Multi-tensor Rrk subtensors} for $(3, \sr)$ there exist $Y \subset Q_2, Z \subset Q_3,W \subset Q_4$ each with size at most $F_{3,\sr,l}(l)$ such that \[ \pr \left(\sum_{i=1}^l b_i T_{x_i}\right)(Y \times Z \times W) \ge l \] for every $b \in \mathbb{F}^l \setminus \{0\}$. Letting $X = \{x_1,\dots,x_l\}$ and using Lemma \ref{Separation lemma for 4D pr} we conclude $\pr T(X \times Y \times Z \times W) \ge l$.
\medskip

\noindent \textbf{Case 2}: If we are not in Case 1, then, proceeding as in the proof of Proposition \ref{Order-3 slice rank subtensors}, we can write \[ T = S^1 + S^2 + S^3 + S^4 + U \] with \[ S^1(x,y,z,w) = \sum_{i=1}^{r_1} a_{i,1}(x) T_{x_i}(y,z,w) \] for some $r_1 \le l$ and $x_1, \dots x_l \in Q_1$, $a_1, \dots, a_l: Q_1 \rightarrow \F$, with $S^2, S^3, S^4$ order-$4$ tensors with similar expressions (with the singled out coordinate being $y,z,w$ respectively) and with $U$ an order-$4$ tensor such that every order-$3$ slice of $U$ (of any of the four types) has slice rank at most $G_{3,\sr,l}(l)+4l$. In particular $S^1,\dots ,S^4$ have partition rank at most $l$, so by subadditivity $\pr U \ge k-4l$, so $\pr_{(2,2)} U \ge k-4l$. Since $k = G_{4,\pr_{(2,2)}}((G_{3,\sr,l}(l)+4l)(5l)^2) + 4l$, by Theorem \ref{Subtensors theorem} for $(4, \pr_{(2,2)})$ there exist $X,Y,Z,W$ with size at most $F_{4,\pr_{(2,2)}}((G_{3,\sr,l}(l)+4l)(5l)^2)$ such that \[ \pr_{(2,2)} U(X \times Y \times Z \times W) \ge (G_{3,\sr,l}(l)+4l)(5l)^2. \] Moreover every order-$3$ slice of $U$ (of any of the four types) has slice rank at most $G_{3,\sr,l}(l) + 4l$ (so in particular, this is still true for the order-3 slices of the restriction $U(X \times Y \times Z \times W)$). By Proposition \ref{Equivalence between order-$4$ $pr$ and order-$4$ $(2-2)-pr$} we have \[ \pr U(X \times Y \times Z \times W) \ge 5l \] and hence \[ \pr T(X \times Y \times Z \times W) \ge l \] by subadditivity.\end{proof}

We now begin the second phase of the arguments of this section, where we deduce Theorem \ref{Subtensors theorem} for $(d,R) = (4,\pr_{(2,2)})$ from Proposition \ref{Multi-tensor Rrk subtensors} for $(d,R) = (2,\rk)$ and from Theorem \ref{Subtensors theorem} for $(d,R) = (4, 1 \times \sr)$. Before the deduction we will again prove a separation lemma and an equivalence. Their bounds will be slightly worse than those of the respective corresponding statements that we have encountered so far and we will highlight the relevant computations.

\begin{lemma} \label{Spanning lemma for $(2-2)-pr$} Let $T: Q_1 \times Q_2 \times Q_3 \times Q_4 \rightarrow \F$ be an order-$4$ tensor, and let $l \ge 1$ be an integer. If there exist $(x_1,y_1), \dots, (x_l,y_l) \in Q_1 \times Q_2$ such that \[ \rk \left(\sum_{i=1}^l a_i T_{x_i,y_i}\right) \ge l(l-1)+1 \] for every $(a_1, \dots, a_l) \in \F^l \setminus \{0\}$ then $\pr_{(2,2)} T \ge l$. \end{lemma}

\begin{proof} Assume for contradiction that $\pr_{(2,2)} T \le l-1$. Then there exist nonnegative integers $r,s,t \ge 0$ with $r+s+t=\pr_{(2,2)} T$ such that we can write a decomposition \[T(x,y,z,w) = \sum_{i=1}^r a_i(x,y) b_i(z,w) + \sum_{j=1}^s c_j(x,z) d_j(y,w) + \sum_{k=1}^t e_k(x,z) f_k(y,w).\] For each $h \in \lbrack l \rbrack$ we can write \[T_{x_h,y_h}(z,w) = \sum_{i=1}^r a_{i,x_h,y_h} b_i(z,w) + \sum_{j=1}^s c_{j,x_h}(z) d_{j,y_h}(w) + \sum_{k=1}^t e_{k,x_h}(z) f_{k,y_h}(w).\] Because $l \ge r+1$ there exists a function $a: Q_1 \times Q_2 \rightarrow \F$ supported inside $\{(x_1, y_1) \dots, (x_l,y_l) \}$ such that $a \neq 0$ but $\sum_{h=1}^l a(x_h,y_h) a_i(x_h,y_h) = 0$ for each $i \in \lbrack r \rbrack$. Hence, \[\sum_{h=1}^l \mu_h T_{x_h,y_h}(z,w) = \sum_{j=1}^s \left(\sum_{h=1}^l \mu_h c_{j,x_h}(z) d_{j,y_h}(w)\right) + \sum_{k=1}^t \left(\sum_{h=1}^l  \mu_h e_{k,x_h}(z) f_{k,y_h}(w)\right)\] where we write $\mu_h = a(x_h,y_h)$. The right-hand side has rank at most $sl+tl \le (l-1)l$, a contradiction. \end{proof}

We next prove our equivalence statement.

\begin{proposition} \label{Equivalence between 4D (2-2)-pr and 1 x sr} Let $T: Q_1 \times Q_2 \times Q_3 \times Q_4 \rightarrow \mathbb{F}$ be an order-$4$ tensor, and let $l,m \ge 1$ be integers. Assume that $\rk T_{z,w}, \rk T_{y,w}, \rk T_{y,z} \le m$ for respectively all $(z,w) \in Q_3 \times Q_4$, all $(y,w) \in Q_2 \times Q_4$, and all $(y,z) \in Q_2 \times Q_3$. If $\pr_{(2,2)} T \le l$, then $(1 \times \sr) T \le (m+l)l^2$. \end{proposition}

\begin{proof} There exists a decomposition

\[ T(x,y,z,w) = \sum_{i=1}^r a_i(x,y) b_i(z,w) + \sum_{j=1}^s c_j(x,z) d_j(y,w) + \sum_{k=1}^t e_k(x,z) f_k(y,w)\] of $T$ with $r+s+t = \pr_{(2,2)} T$. The family $\{ b_1, \dots, b_r \}$ is linearly independent. By Gaussian elimination there exist functions $b_i^*: Q_3 \times Q_4 \rightarrow \F$ all supported inside a set $U \subset Q_3 \times Q_4$ of size $r$ such that $b_i^*.b_i = \delta_{i,i'}$ for all $i,i' \in \lbrack r \rbrack$. For a given $i \in \lbrack r \rbrack$, taking the dot product with $b_i^*$ of $T$ we get \[ a_i = b_i^*.T - \sum_{j=1}^s b_i^*.(c_j d_j) - \sum_{k=1}^t b_i^*.(e_k f_k). \] Because for each $(z,w) \in U$, the order-$2$ slice $T_{z,w}$ has rank at most $m$ and $b_i^*$ is supported inside a set of size at most $r$, by subadditivity of the rank we have $\rk b_i^*.T \le mr$. Let $(z_1,w_1),\dots,(z_r,w_r)$ be the elements of $U$. For each $j \in \lbrack s \rbrack$, \[(b_i^*.(c_j d_j))(z,w) = \sum_{h=1}^r b_i^*(z_h,w_h) c_j(x,z_h) d_j(y,w_h) \] so $\rk (b_i^*.(c_j d_j)) \le r \le l$. Similarly for each $k \in \lbrack t \rbrack$, $\rk (b_i^*.(e_k f_k)) \le l$. Therefore by subadditivity $\rk a_i \le mr + ls+lt$. Similarly $\rk c_j \le lr + ms+lt$ and $\rk e_k \le lr + ls+mt$, so \[(1 \times \sr) T \le r(mr + ls+lt) + s(lr + ms+lt) + t(lr + ls+mt) \le (m+l)l^2. \qedhere\] \end{proof}

We are again ready to write the step leading to Theorem \ref{Subtensors theorem} for order-$4$ $(2,2)$-partition rank subtensors.

\begin{proposition} \label{Up to (2-2)-pr} If Theorem \ref{Subtensors theorem} holds for $(d,R) = (4,1 \times \sr)$ then Theorem \ref{Subtensors theorem} holds for $(d,R) = (4, \pr_{(2,2)})$ with \[ F_{4, \pr_{(2,2)}}(l) = F_{4,1 \times \sr}(112 l^4)\] \[G_{4, \pr_{(2,2)}}(l) = G_{4,1 \times \sr}(112 l^4) + 3l.\] \end{proposition}

\begin{proof}

Let $T$ be an order-$4$ tensor such that $\pr_{(2,2)} T \ge G_{4,1 \times \sr}(112 l^4) + 3l$. 
\medskip

\noindent \textbf{Case 1}: For at least one of the three types of slices $T_{z,w}, T_{y,w}, T_{y,z}$, which without loss of generality we can assume to be the direction with $z,w$ as fixed coordinates there exist $(z_1,w_1),\dots ,(z_l,w_l) \in Q_3 \times Q_4$ such that \[ \rk \left(\sum_{i=1}^l a_i T_{z_i,w_i}\right) \ge l^2 \] for every $a \in \mathbb{F}^l \setminus \{0\}$. Then by Proposition \ref{Multi-tensor matrix rank} there exist $X \subset Q_1$, $Y \subset Q_2$ with size at most $l^2$ such that \[ \rk \left(\sum_{i=1}^l a_i T_{z_i,w_i} \right)(X \times Y) \ge l^2 \] for every $a \in \mathbb{F}^l \setminus \{0\}$. Letting $Z = \{z_i: 1 \le i \le l\}$, $W = \{w_i: 1 \le i \le l\}$, by Lemma \ref{Spanning lemma for $(2-2)-pr$} we get $\pr_{(2,2)} T(X \times Y \times Z \times W) \ge l$.
\medskip

\noindent \textbf{Case 2}: We are not in Case 1. Then there exist $(z_i,w_i)$, $i=1,\dots ,r$ with $r \le l$ such that for all $(z,w) \in Q_3 \times Q_4$, there exist $A_1(z,w),\dots, A_{r}(z,w)$ such that \[ \rk \left(T_{z,w} - \sum_{i=1}^{r} A_i(z,w) T_{z_i,w_i}\right) \le l^2. \]

We write $S^{34}(x,y,z,w)= \sum_{i=1}^{r} A_i(z,w) T_{z_i,w_i}(x,y)$. Similarly we define two other tensors $S^{24}$ and $S^{23}$ (constructed using linear combinations of slices of the types $T_{y,w}$ and $T_{y,z}$ respectively), and write \[ T = S^{34} + S^{24} + S^{23} + U. \] The tensors $S^{34}, S^{24}, S^{23}$ have $(2,2)$-partition rank at most $l$, so $\pr_{(2,2)} U \ge k-3l$ and hence $(1 \times \sr) U \ge k - 3l$. Since $k \ge G_{4,1 \times \sr}(112 l^4) + 3l$, by $1 \times \sr$ subtensors there exist $X \subset Q_1, Y \subset Q_2, Z \subset Q_3, W \subset Q_4$ of size at most $F_{4,1 \times \sr}(112 l^4)$ such that \[ (1 \times \sr ) U(X \times Y \times Z \times W) \ge 112 l^4 \ge ((l^2+2l)+4l)(4l)^2. \] The order-2 slices of $U$ with pairs of fixed coordinates $(z,w)$, $(y,w)$, and $(y,z)$ all have rank at most $l^2+2l$, so this is still the case for the corresponding order-2 slices of $U(X \times Y \times Z \times W)$. By Proposition \ref{Equivalence between 4D (2-2)-pr and 1 x sr} \[ \pr_{(2,2)} U(X \times Y \times Z \times W) \ge 4l \] and hence \[ \pr_{(2,2)} T(X \times Y \times Z \times W) \ge l. \qedhere\] \end{proof}

Since the proof of Proposition \ref{Up to pr} had as its second case a reduction to the notion of $(2,2)$-partition rank, where all parts of all partitions of the corresponding family $R$ have size at most $2$, we can ask whether the proof of Proposition \ref{Up to (2-2)-pr} could not have reduced to a family where these parts have size at most $1$, in other words reduced to the tensor rank. What stands in the way of doing so is that given $l$ a fixed nonnegative integer and $S^{12}$, $S^{34}$ order-$4$ tensors of the type \[S^{12}(x,y,z,w) = \sum_{i=1}^{r} a_i^{12}(x,y) T_{x_i,y_i}(z,w)\] \[S^{34}(x,y,z,w) = \sum_{i=1}^{s} a_i^{34}(z,w) T_{z_i,w_i}(x,y)\] with $r,s \le l$ such that for each $(x,y) \in Q_1 \times Q_2$, $\rk (T - S^{12})_{(x,y)} \le l$ and for each $(z,w) \in Q_3 \times Q_4$, $\rk (T - S^{34})_{(z,w)} \le l$, it is not necessarily true that the ranks of the slices of the type $(T-S^{12}-S^{34})_{(x,y)}$ and $(T-S^{12}-S^{34})_{(z,w)}$ are bounded: to conclude this we would also want to  know that the functions $a_i^{12}$ and $a_i^{34}$ have bounded rank. This difficulty led us to to the structure of the proof in Section \ref{section: Additional difficulties for the order-$4$ tripartition rank}.

We finish by proving Theorem \ref{Subtensors theorem} for the $1$-enhanced slice rank using its product structure.

\begin{proposition} \label{Subtensors for $1$-enhanced slice rank} Theorem \ref{Subtensors theorem} holds for $(d,R) = (4, 1 \times \sr)$, with \[F_{4, 1 \times \sr}(l) = F_{3,\sr}(l)\] \[G_{4, 1 \times \sr}(l) = l \cdot G_{3, \sr}(l).\] \end{proposition}

\begin{proof} Let $T$ be an order-$4$ tensor with $(1 \times \sr) T \ge k$. Let $A: Q_1 \times (Q_2 \times Q_3 \times Q_4) \rightarrow \F$ be the matrix defined by $A(x,(y,z,w)) = T(x,y,z,w)$. We distinguish two cases depending on the value of $\frank_1 T = \rk A$.
\medskip

\noindent \textbf{Case 1}: $\frank_1 T \ge l$. Then by the standard result on full-rank submatrices of matrices there exist subsets $X \subset Q_1$ and $U \subset Q_2 \times Q_3 \times Q_4$, both with size $l$, such that $\rk A(X_1 \times U) = l$. Then the canonical projections $Y,Z,W$ of $U$ on the second, third and fourth coordinate axes respectively have size at most $l$, and we have $\frank_1 T(X \times Y \times Z \times W) \ge l$ whence $({1} \times \sr) T(X \times Y \times Z \times W) \ge l$.
\medskip

\noindent \textbf{Case 2}: $\frank_1 T \le l$. Letting $l'=\frank_1 T$ there exist $x_1, \dots, x_{l'} \in Q_1$ and functions $a_1,\dots,a_{l'}: Q_1 \rightarrow \F$ such that \[T(x,y,z,w) = \sum_{i=1}^{l'} a_i(x) T_{x_i}(y,z,w). \] By this decomposition we have $(1 \times \sr) T \le \sum_{i=1}^{l'} \sr T_{x_i}$, so there exists $i \in \lbrack l' \rbrack$ such that $\sr T_{x_i} \ge G_{3,\sr}(l)$. By Proposition \ref{Order-3 slice rank subtensors} there exist $Y \subset Q_2, Z \subset Q_3, W \subset Q_4$ with size at most $F_{3,\sr}(l)$ such that $\sr T_{x_i}(Y \times Z \times W) \ge l$. Letting $X = \{x_i\}$ we obtain $(1 \times \sr) T(X \times Y \times Z \times W) \ge l$. \end{proof}

\section{Additional difficulties for order-4 tripartition rank subtensors} \label{section: Additional difficulties for the order-$4$ tripartition rank}

For $T$ an order-$4$ tensor, let the \emph{tripartition rank} $\trp T$ of $T$ be the value of $\Rrk T$ for $R$ the set of partitions of $\{1,2,3,4\}$ into three (non-empty) parts.

Throughout this section we prove Theorem \ref{Subtensors theorem} for $(d,R) = (4, \trp)$. The structure of the proof will be as follows. Let $T$ be an order-$4$ tensor with large tripartition rank. Rather than the two cases involved in the proofs of the previous section, the proof will distinguish between three cases (the second of our steps below is a preparation for the last two cases).

\begin{enumerate}

\item If the tensor $T$ has a large separated set of order-$2$ slices, without loss of generality of the type $T_{x_i,y_i}$, then we can restrict them to a product $Z \times W$ with $Z,W$ of bounded size such that the restrictions are still separated, and containing the $(x_i,y_i)$ in a box $X \times Y$ with bounded size ensures that the tripartition rank of $T(X \times Y \times Z \times W)$ is large.

\item If the tensor $T$ does not have such a large separated set, then we find a bounded number of $(x_i,y_i)$ and of functions $A_i: Q_1 \times Q_2 \rightarrow \F$ such that we can approximate each slice $T_{x,y}$ by \[ \sum_{i} A_i(x,y) T_{x_i,y_i}(z,w) \] and similarly for the five other kinds of order-$2$ slices. Requiring the quality of the approximation to improve with every new function $A_i$ by some well-chosen factor ensures that regardless of the errors in the approximation, the rank of the difference of two slices $(x',y')$, $(x'',y'')$ with equal families $(A_i(x',y'))$, $(A_i(x'',y''))$ is always smaller than the rank of the difference of two slices $(x',y')$, $(x'',y'')$ with different families $(A_i(x',y'))$, $(A_i(x'',y''))$.

\item If at least one of the functions $A_i$ has high rank, or similarly for their analogues for one of the five other kinds of order-$2$ slices, then taking subtensors $X \times Y$ and $Z \times W$ such that $A_i(X \times Y)$ has high rank and the slices $T_{x_i,y_i}(Z \times W)$ are still separated ensures that the tripartition rank of $T(X \times Y \times Z \times W)$ is large.

\item If all the functions $A_i$ have bounded rank, and similarly for their analogues for all five other kinds of order-$2$ slices, then we can decompose $T = S + U$ with $S$ a tensor with bounded tripartition rank and $U$ a tensor such that all slices $U_{(x,y)}$ with $(x,y) \in Q_1 \times Q_2$ have bounded rank. Using an equivalence between the tripartition rank and the (order-$4$) tensor rank we are then able to conclude by applying order-$4$ tensor rank subtensors.

\end{enumerate}

We begin as usual with the relevant separation result. Because the tripartition rank is at least as big as the $(2,2)$-partition rank, the following lemma follows from Lemma \ref{Spanning lemma for $(2-2)-pr$}.

\begin{lemma} \label{Separation lemma for 4D trp} Let $l \ge 1$ be an integer, let $T$ be an order-$4$ tensor, and suppose that there exist $(x_1,y_1),\dots,(x_l,y_l) \in Q_1 \times Q_2$ such that \[\rk \left(\sum_{i=1}^l a_i T_{x_i,y_i}\right) \ge l(l-1)+1\] for every $a \in \mathbb{F}^l \setminus \{0\}$. Then $\trp T \ge l$. \end{lemma}

We next want to obtain another condition which guarantees that a tensor has high tripartition rank and will allow us to conclude in the second of the three cases of the proof. To prepare for the proof of this condition we first prove a lemma that states that the order-$2$ slices of a tensor with bounded tripartition rank can be well-approximated by a linear combination of a bounded number of order-$2$ tensors.

\begin{lemma} \label{Spanning for trp} Let $l \ge 1$ be an integer and let $T$ be an order-$4$ tensor with $\trp T \le l$. Then there exist functions $C_1, \dots, C_l$: $Q_1 \times Q_2 \rightarrow \F$, each with rank at most $1$, and functions $D_1, \dots, D_l: Q_3 \times Q_4 \rightarrow \F$ such that \[\rk \left(T_{x,y} - \sum_{i=1}^l C_i(x,y) D_i\right) \le l \] for every $(x,y) \in Q_1 \times Q_2$. \end{lemma}

\begin{proof} There exists a decomposition \[T(x,y,z,w) = \sum_{i=1}^r a_i(x) b_i(y) c_i(z,w) + \sum_{i=1}^{r'} F_i(x,y,z,w)\] for some nonnegative integers $r,r'$ with $r+r' = \trp T \le l$, where for each $i \in \lbrack r' \rbrack$, the function $F_i$ is of one of the five types \begin{align*} a(x) b(z) c(y,w)\text{, } a(x) b(w) c(y,z) \text{, } a(y) b(z) c(x,w) \text{, } a(y) b(w) c(x,z) \text{, } a(z) b(w) c(x,y). \end{align*} We take $C_i = a_i b_i$ and $D_i = c_i$ for each $i \in \lbrack r \rbrack$. For any given $(x,y) \in Q_1 \times Q_2$ and each $i \in \lbrack r' \rbrack$ the rank of $(F_i)_{(x,y)}$ is at most $1$, so the desired inequality follows by subadditivity. \end{proof}

We can now deduce our second condition that guarantees that an order-$4$ tensor has high tripartition rank.

\begin{proposition}\label{Case 2 proposition} Let $M,m,l',l \ge 1$ be integers, let $T: Q_1 \times Q_2 \times Q_3 \times Q_4 \rightarrow \mathbb{F}$ be an order-$4$ tensor, and let functions $A_1,\dots ,A_{l'}: Q_1 \times Q_2 \rightarrow \mathbb{F}$, functions $B_1,\dots ,B_{l'}: Q_3 \times Q_4 \rightarrow \mathbb{F}$ be such that the three following conditions hold.
\begin{enumerate}[(i)]
\item For all $a \in \mathbb{F}^{l'} \setminus \{0\}$,  $\rk \left(\sum_{i=1}^{l'} a_i B_i\right) \ge M$.
\item There exists $j \in \lbrack l' \rbrack$ such that $\rk A_j \ge l$.
\item For all $(x,y) \in Q_1 \times Q_2$,  $\rk \left(T_{x,y} - \sum_{i=1}^{l'} A_i(x,y) B_i\right) \le m$.
\end{enumerate}
If $M \ge (m+l)l$, then $\trp T \ge l$.
\end{proposition}

\begin{proof}

Assume for contradiction that $\trp T < l$. By Lemma \ref{Spanning for trp} there exist functions $C_1,\dots ,C_{l-1}: Q_1 \times Q_2 \rightarrow \mathbb{F}$ and functions $D_1,\dots ,D_{l-1}: Q_3 \times Q_4 \rightarrow \mathbb{F}$ such that $\rk C_i = 1$ for each $i \in \lbrack l-1 \rbrack$ and such that \begin{equation} \rk \left(T_{x,y} - \sum_{i=1}^{l-1} C_i(x,y) D_i\right) \le l-1 \label{rank approximation} \end{equation} for every $(x,y) \in Q_1 \times Q_2$. By (ii), there exists some $j \in \lbrack l' \rbrack$ such that $\rk A_j \ge l$. Since $\rk C_i = 1$ for all $i \in \lbrack l-1 \rbrack$, $A_j$ does not belong to the the linear span of $C_1,\dots ,C_{l-1}$. It follows that there exists a function $u: Q_1 \times Q_2 \rightarrow \mathbb{F}$ supported in a subset $U$ of $Q_1 \times Q_2$ with size at most $l$ such that $u.C_i = 0$ for each $i \in \lbrack l-1 \rbrack$ but $u.A_j \neq 0$. We can write $u.T = \sum_{(x,y) \in U} u(x,y) T_{x,y}$. On the one hand, applying \eqref{rank approximation}, subadditivity, and the fact that $u.C_i = 0$ for each $i \in \lbrack l-1 \rbrack$ we obtain \begin{equation} \rk (u.T) \le (l-1) |U| \le (l-1)l \label{bound on rk u.T}. \end{equation} On the other hand, by assumption (iii) and subadditivity of the rank, we have \[ \rk \left(u.T - \sum_{i=1}^{l'} (u.A_i) B_i\right) \le m |U| \le ml. \] Since $u.A_j \neq 0$, by applying assumption (i) to $(a_1,\dots,a_{l'})= (u.A_1,\dots,u.A_{l'})$ we have $\rk \sum_{i=1}^{l'} (u.A_i) B_i \ge M$, and therefore by subadditivity \[ \rk u.T \ge M - ml. \] From this inequality and \eqref{bound on rk u.T} we obtain $(l-1)l \ge M - ml$ and therefore $M \le (m+l-1)l$, a contradiction. \end{proof}

In order to conclude in the third case of our argument we prove an equivalence between the tripartition rank and the tensor rank for order-$4$ tensors.

\begin{proposition} \label{Equivalence between trp and 4D tr subtensors} Let $l,m$ be integers. Let $T$ be an order-$4$ tensor such that $\rk T_{x,y} \le m$, $\rk T_{x,z} \le m$, $\rk T_{x,w} \le m$, $\rk T_{y,z} \le m$, $\rk T_{y,w} \le m$, $\rk T_{z,w} \le m$ for respectively all $(x,y) \in Q_1 \times Q_2$, all $(x,z) \in Q_1 \times Q_3$, all $(x,w) \in Q_1 \times Q_4$, all $(y,z) \in Q_2 \times Q_3$, all $(y,w) \in Q_2 \times Q_4$, all $(z,w) \in Q_3 \times Q_4$. If $\trp T \le l$, then $\tr T \le (m+l) l^2$. \end{proposition}

\begin{proof} We start with a decomposition of $T(x,y,z,w)$ as \begin{align*} & \sum_{i=1}^{r_{12}} A_{12,i}(x,y) B_{12,i}(z,w) + \sum_{i=1}^{r_{13}} A_{13,i}(x,z) B_{12,i}(y,w) + \sum_{i=1}^{r_{14}} A_{14,i}(x,w) B_{12,i}(y,z) \\
+ & \sum_{i=1}^{r_{23}} A_{23,i}(y,z) B_{12,i}(x,w) + \sum_{i=1}^{r_{24}} A_{24,i}(y,w) B_{12,i}(x,z) + \sum_{i=1}^{r_{34}} A_{34,i}(z,w) B_{12,i}(x,y)
\end{align*} where the functions $A_{12,i},A_{13,i},A_{14,i},A_{23,i},A_{24,i},A_{34,i}$ all have rank $1$, and $r_{12} + r_{13} + r_{14} + r_{23} + r_{24} + r_{34} = \trp T$.

The functions $A_{12,i}$, $i \in \lbrack r_{12} \rbrack$ are linearly independent, so we can find functions $A_{12,i}^*, i \in \lbrack r_{12} \rbrack$ all supported inside a subset $U \subset Q_1 \times Q_2$ with size at most $r_{12}$ and such that $A_{12,i'}^*.A_{12,i} = \delta_{i,i'}$ for every $i,i' \in \lbrack r_{12} \rbrack$. Taking the dot product with $A_i^*$ in the decomposition we obtain that for each $i \in \lbrack r \rbrack$, $\rk B_{12,i} \le r_{12}m+ (r_{13} + r_{14} + r_{23} + r_{24} + r_{34})l$, and similar bounds for the ranks of the functions $B_{13,i}, B_{14,i},B_{23,i},B_{24,i},B_{34,i}$. The result follows by subadditivity. \end{proof}

We are now ready to deduce Theorem \ref{Subtensors theorem} for the tripartition rank of order-$4$ tensors.

\begin{proposition} \label{Up to trp} Theorem \ref{Subtensors theorem} holds for $(d,R) = (4, \trp)$, and for each $l \ge 2$ we may take \[F_{4,\trp}(l) = G_{4,\trp}(l) = 1000 l^{l+6}.\] \end{proposition}

\begin{proof}

Let $l \ge 2$ and let $T$ be an order-$4$ tensor wth $\trp T \ge 1000 l^{l+6}$. We distinguish three cases, as follows.
\medskip

\noindent \textbf{Case 1}: For at least one of the six types of order-$2$ slices $T_{x,y}$, $T_{x,z}$, $T_{x,w}$, $T_{y,z}$, $T_{y,w}$, $T_{z,w}$, which without loss of generality we can assume to be of the first type, there exist $(x_1,y_1), \dots, (x_l,y_l) \in Q_1 \times Q_2$ such that \begin{equation} \rk \left(\sum_{i=1}^l a_i T_{x_i,y_i}\right) \ge l^2 \label{rk separated set equation} \end{equation} for every $a \in \mathbb{F}^l \setminus \{0\}$. By Proposition \ref{Multi-tensor matrix rank} we can find $Z \subset Q_3$, $W \subset Q_4$ with size at most $l^3$ such that \[ \rk \left(\sum_{i=1}^l a_i T_{x_i,y_i}\right)(Z \times W) \ge l^2 \] for every $a \in \mathbb{F}^l \setminus \{0\}$. Letting $X = \{x_i: 1 \le i \le l\}$ and $Y = \{y_i: 1 \le i \le l\}$, the sets $X,Y$ have size at most $l$ and by Lemma \ref{Separation lemma for 4D trp} we get \[\trp T(X \times Y \times Z \times W) \ge l. \] 
\bigskip

Let us now assume that we are not in Case 1. We here prepare Cases 2 and 3. For each of the six types of order-$2$ slices we can define a process as follows: we will here describe it for the first and second coordinates taken to be the fixed coordinates. If there exists $(x,y)$ such that \[\rk T_{x,y} > l^2 (l^{l}-1)/(l-1)\] then we let $(x_1,y_1) = (x,y)$, otherwise we stop the process. Thereafter, if there exists $(x,y)$ such that for all $a_1 \neq 0$, \[ \rk (T_{x,y} - a_1 T_{x_1,y_1}) > l^2 (l^{l-1}-1)/(l-1)\] then we let $(x_2,y_2) = (x,y)$, and otherwise we stop. More generally, at the $i$th step of the process, if there exists $(x,y)$ such that for all $(a_1,\dots ,a_{i-1}) \neq 0$, \[\rk \big(T_{x,y} - (a_1 T_{x_1,y_1} + \dots + a_{i-1} T_{x_{i-1},y_{i-1}})\big) > l^2 (l^{l-i+1}-1)/(l-1)\] then we let $(x_i,y_i) = (x,y)$, and otherwise we stop.

The process must necessarily stop after at most $l-1$ iterations: if it did not, then we would have $(x_i,y_i)$, $i=1,\dots ,l$ satisfying \eqref{rk separated set equation} for all $(a_1,\dots ,a_l) \in \mathbb{F}^l \setminus \{0\}$; in other words we would be in Case 1, which we have assumed that we are not. Let $l' \le l-1$ be the number of iterations after which the process stops. The family $\{T_{x_i,y_i}: 1 \le i \le l' \}$ satisfies the following two properties: the inequality \begin{equation} \rk \left(\sum_{i=1}^l a_i T_{x_i,y_i}\right) \ge l^2 (l^{l-l'+1}-1)/(l-1) \label{trp property 1} \end{equation} holds for every $a \in \F^{l'} \setminus \{0\}$, and for each $(x,y) \in Q_1 \times Q_2$ there exists an element $(A_1(x,y),\dots ,A_{l'}(x,y)) \in \mathbb{F}^{l'}$ such that \begin{equation}\rk \left(T_{x,y} - \sum_{i=1}^{l'} A_i(x,y) T_{x_i,y_i}\right) \le l^2 (l^{l-l'}-1)/(l-1). \label{trp property 2} \end{equation}

We now distinguish two further cases (Cases 2 and 3 below) depending on whether for at least one of the six types of order-$2$ slices, there exists $j \in \lbrack l' \rbrack$ such that $A_j$ has high rank.
\medskip

\noindent \textbf{Case 2}: For at least one of the six types of order-$2$ slices $T_{x,y}$, $T_{x,z}$, $T_{x,w}$, $T_{y,z}$, $T_{y,w}$, $T_{z,w}$, which without loss of generality we can assume to be the first type, there exists $j \in \{1,\dots ,l'\}$ such that $\rk A_j \ge l$. By the standard result on full-rank submatrices of matrices there exist $X,Y$ with size at most $l$ such that $\rk A_j(X \times Y) \ge l$. Moreover, by \eqref{trp property 1} and by Proposition \ref{Multi-tensor matrix rank} there exist $Z,W$ with size at most $l' l^2 (l^{l+1}-1)/(l-1) \le l^{l+4}$ such that \[ \rk \left(\sum_{i=1}^{l'} a_i T_{x_i,y_i}\right)(Z \times W) \ge l^2 (l^{l-l'+1}-1)/(l-1)\] for every $(a_1,\dots ,a_{l'}) \in \mathbb{F}^{l'} \setminus \{0\}$. For each $i \in [l']$ let $B_i = T_{x_i,y_i}$. Applying Proposition \ref{Case 2 proposition} to the tensor $T(X \times Y \times Z \times W)$, the functions $A_i(X \times Y)$ for each $i \in \lbrack l' \rbrack$, the functions $B_i(Z \times W)$ for each $i \in \lbrack l' \rbrack$, and the parameters $m = l^2 (l^{l-l'}-1)/(l-1)$ and $M = l(m+l) = l^2 (l^{l-l'+1}-1)/(l-1)$, we obtain that $\trp T(X \times Y \times Z \times W) \ge l$.
\medskip

\noindent \textbf{Case 3}: We are not in Case 1 or Case 2. We define a decomposition \[ T = S^{12} + S^{13} + S^{14} + S^{23} + S^{24} + S^{34} + U \] as follows. We define the tensor $S^{12}$ by \[ S^{12}(x,y,z,w) = \sum_{i=1}^{l'^{12}} A_{i,12}(x,y) T_{x_i,y_i}(z,w) \] where $l'^{12}$ is the value of $l'$ obtained in the process above for the slices of the type $T_{x,y}$, and we define the tensors $S^{13}, S^{14}, S^{23}, S^{24}, S^{34}$ similarly for the slices of the type $T_{x,z}$, $T_{x,w}$, $T_{y,z}$, $T_{y,w}$, $T_{z,w}$ respectively.

For each $i \in \{1,\dots, l'^{12}\}$ we have $\rk A_{i,12} \le l$, so it follows from the decomposition defining $S^{12}$ that $\trp S^{12} \le l'^{12}l \le l^2$. We similarly see that this upper bound holds for the tripartition ranks of the tensors $S^{12}, S^{13}, S^{14}, S^{23}, S^{24}, S^{34}$.

Let $(x,y) \in Q_1 \times Q_2$ be fixed. By \eqref{trp property 2} and using that $l^2 (l^l-1)/(l-1) \le l^{l+2}$ we have $\rk (T-S^{12})_{x,y} \le l^{l+2}$; moreover, $\rk S^{13}_{(x,y)}, \rk S^{14}_{(x,y)}, \rk S^{23}_{(x,y)}, \rk S^{24}_{(x,y)}, \rk S^{34}_{(x,y)}$ are at most $\trp S^{13}$, $\trp S^{14}$, $\trp S^{23}$, $\trp S^{24}$, $\trp S^{34}$ respectively, so are each at most $l^2$, so by subadditivity $\rk U_{x,y} \le 6 l^{l+2}$. We similarly see that this upper bound holds for the ranks of the five other types of order-$2$ slices of $U$.

Since $\trp T \ge 1000 l^{l+6}$, by subadditivity $\trp U \ge 1000 l^{l+6} - 6l^{l+2} \ge 900 l^{l+6}$, so $\tr U \ge 900 l^{l+6}$. Applying Proposition \ref{Subtensors for order-$d$ tensor rank} to $U$ we can find $X,Y,Z,W$ with size at most $900 l^{l+6}$ such that \[ \tr U(X \times Y \times Z \times W) \ge 900 l^{l+6} \ge (6l^{l+2}+7l^2)(7l^2)^2. \] The slices $U(X \times Y \times Z \times W)_{(x,y)}$ and all other order-$2$ slices all have rank at most $6 l^{l+2}$, so applying Proposition \ref{Equivalence between trp and 4D tr subtensors} we obtain that \[ \trp U (X \times Y \times Z \times W) \ge 7l^2, \] and hence that \[ \trp T(X \times Y \times Z \times W) \ge l^2 \ge l. \qedhere\] \end{proof}

\section{The inductive proof in the general case}\label{section: General case}

\subsection{Definitions}

In this section we prove Theorem \ref{Subtensors theorem} in full. For $d \ge 2$ an integer and $R$ a non-empty family of partitions of $\lbrack d \rbrack$, we shall say that a set $C \subset \lbrack d \rbrack$ is a \emph{largest part} for $R$ if there exists $P \in R$ with $C \in P$ and if moreover $|C| = \max_{P \in R, I \in P} |I|$. We shall write $R_{\tr}=\{\{\{1\}, \dots, \{d\}\}\}$ for the family of partitions corresponding to the tensor rank.

Our proof strategy is to proceed by induction, with the base case of the induction being the case $R=R_{\tr}$, for which we already know by Proposition \ref{Subtensors for order-$d$ tensor rank} that Theorem \ref{Subtensors theorem} holds. The inductive step will consist of a double induction on $d$ and the size of the largest part for $R$. If we are not in the base case, we choose a largest part $C \subset \lbrack d \rbrack$ for $R$ that will remain fixed throughout the inductive step. The following notions of rank will be relevant to the proof. (In the third notion, “comp” stands for “compatible”.) \begin{align*}
R_+ & = \{P \in R: C \in P\} \\
R_- & = \{P \in R: C \notin P\} \\
\Rcomp & = \{P \setminus \{C\}: P \in R_+\} \\
\Rnew & = \{P \cup \{I,J\}: P \in \Rcomp, \{I,J\} \text{ a bipartition of }C \} \\
R' & = R_- \cup \Rnew.
\end{align*} We emphasize that although $R_+,R_-, R', \Rnew$ are notions of rank for order-$d$ tensors $\prod_{\a=1}^d Q_{\a} \rightarrow \F$, $\Rcomp$ is a notion of rank for tensors $\prod_{\a \in C^c} Q_{\a} \rightarrow \F$. We will refer to $R'$ as the \emph{down-shadow} of $R$ (with respect to $C$). Informally, the set $R'$ is the set $R$, modified such that whenever $C$ appears in a partition $P$ of $R$, the partition $P$ is replaced by all partitions which are identical to $P$ in all ways except that the part $C$ is split into two non-empty parts. Starting from a non-empty family $R$ of partitions on $\lbrack d \rbrack$, choosing a largest part $C$ for $R$, building a down-shadow $R^1$ of $R$ with respect to $C$, then iterating by again choosing a largest part $C^1$ for $R^1$, building a down-shadow $R^2$ of $R^1$ with respect to $C^1$, and so forth we end up at the family $R_{\tr}$ corresponding to the tensor rank after at most $2^{d}$ iterations, because every iteration forbids one additional subset of $\lbrack d \rbrack$ (the selected largest part) from belonging to any later family $R^i$, and the subsets of $\lbrack d \rbrack$ are forbidden by decreasing size.

We now summarise the structure of the inductive step, which can largely be viewed as a generalisation of the structure of the proof in Section \ref{section: Additional difficulties for the order-$4$ tripartition rank}. Let $R \neq R_{\tr}$ be a non-empty family of partitions on $\lbrack d \rbrack$, let $T$ be an order-$d$ tensor with large $R$-rank, and let $C$ be a largest part for $R$.

\begin{enumerate}

\item If the tensor $T$ has a large separated set of slices $T_{y_1}, \dots, T_{y_l}$ with $y_1, \dots, y_l \in \prod_{\a \in C^c} Q_{\a}$, then by applying Proposition \ref{Multi-tensor Rrk subtensors} for $(|C|,\pr)$ we can restrict this set of slices to a product $\prod_{\a \in C} X_{\a}$ such that the restrictions of the slices to this product are still separated and the sets $X_{\a}$, $\a \in C$ have bounded size. Containing $\{y_1, \dots, y_l\}$ in a box $\prod_{\a \in C^c} X_{\a}$ with bounded size ensures that the $R$-rank of $T\left(\prod_{\a = 1}^d X_{\a}\right)$ is large.

\item If the tensor $T$ does not have such a large separated set, then we can find a bounded number of $y_i \in \prod_{\a \in C^c} Q_{\a}$ and functions $A_i: \prod_{\a \in C^c} Q_{\a} \rightarrow \F$ such that for each $y \in \prod_{\a \in C^c} Q_{\a}$ we can approximate $T_y$ by \[ \sum_{i} A_i(y) T_{y_i} \] Requiring the quality of the approximation to improve with every new function $A_i$ by some well-chosen factor ensures that regardless of the errors in the approximation, the partition rank of the difference of two slices $y'$, $y''$ with equal families $(A_i(y'))$, $(A_i(y''))$ is always smaller than the partition rank of the difference of two slices $y'$, $y''$ with different families $(A_i(y'))$, $(A_i(y''))$.

\item If at least one of the functions $A_i$ has high $\Rcomp$-rank, then by applying Theorem \ref{Subtensors theorem} for $(d-|C|, \Rcomp)$ and Proposition \ref{Multi-tensor Rrk subtensors} for $(|C|,\pr)$ we can find $\prod_{\a \in C^c} X_{\a}$ and $\prod_{\a \in C} X_{\a}$, respectively, such that $A_i\left(\prod_{\a \in C^c} X_{\a}\right)$ has high $\Rcomp$-rank and the slices $T_{y_i}\left(\prod_{\a \in C} X_{\a}\right)$ are still separated. This ensures that the $R$-rank of $T\left(\prod_{\a = 1}^d X_{\a}\right)$ is large.

\item If all the functions $A_i$ have bounded $\Rcomp$-rank, then we have a decomposition $T = S+U$ where $S$ is a tensor with bounded $R$-rank and $U$ is a tensor such that every slice $U_y$ with $y \in \prod_{\a \in C^c} X_{\a}$ has bounded partition rank. We can then conclude using an equivalence statement between the $R$-rank and the $R'$-rank together with Theorem \ref{Subtensors theorem} for $R'$-rank.

\end{enumerate}

Whenever $T: \prod_{\a=1}^d Q_{\a} \rightarrow \F$ is an order-$d$ tensor with $\Rrk T = k$, there exist nonnegative integers $r,r'$ with $r+r'=k$, functions $A_i: \prod_{\a \in C^c} Q_{\a} \rightarrow \F$, $B_i: \prod_{\a \in C} Q_{\a} \rightarrow \F$ with $\Rcomprk A_i = 1$ for each $i \in \lbrack r \rbrack$, and for each $i \in \lbrack r' \rbrack$ a function $F_i: \prod_{\a=1}^d Q_{\a} \rightarrow \F$ with $\Rminusrk F_i = 1$ such that for all $(x_1,\dots,x_d) \in \prod_{\a=1}^d Q_{\a}$ we can write \begin{equation} T(x_1,\dots,x_d) = \sum_{i=1}^r A_i(x(C^c)) B_i(x(C)) + \sum_{i=1}^{r'} F_i(x). \label{general form} \end{equation} Furthermore the family $\{A_1, \dots, A_r \}$ is linearly independent: if this were not the case then we would obtain an $R$-rank decomposition of $T$ with length strictly smaller than $k$. We will repeatedly use the following fact.

\begin{lemma}\label{rank at most 1}

In the decomposition \eqref{general form}, we have $\pr (F_i)_y \le 1$ for each $y \in \prod_{\a \in C^c} Q_{\a}$ and $i \in \lbrack r' \rbrack$.

\end{lemma}

\begin{proof}

Let $i \in \lbrack r' \rbrack$. Since $\Rminusrk F_i = 1$, there exists $P \in R_-$ such that we can write \[F_i(x) = \prod_{I \in P} a_{I}(x(I))\] for some functions $a_{I}: \prod_{\a \in I} Q_{\a} \rightarrow \F$. Because $C$ is a largest part for $R$, $|C|$ is an upper bound on the sizes of all sets $I \in P$; since $C \notin P$, there exist distinct $I_1, I_2 \in P$ such that $C \cap I_1$ and $C \cap I_2$ are both non-empty, so we have $\pr (F_i)_y \le 1$ for each $y \in \prod_{\a \in C^c} Q_{\a}$.

\end{proof}

\subsection{The general case for \texorpdfstring{$R$}{R}-rank subtensors}

In the case of a general non-empty family $R$ of partitions, the formulation of the separation statement that we shall use will be the following.

\begin{lemma} \label{separation lemma} Let $d \ge 2$ be an integer, let $R \neq R_{\tr}$ be a non-empty family of partitions of $\lbrack d \rbrack$, and let $C$ be a largest part for $R$. Let $l \ge 1$ be an integer. If $T$ is an order-$d$ tensor and $y_1,\dots,y_l \in \prod_{\a \in C^c} Q_{\a}$ are such that \[\pr \left(\sum_{i=1}^l a_i T_{y_i}\right) \ge l(l-1)+1 \] for every $a \in \F^l \setminus \{ 0 \}$, then $\Rrk T \ge l$. \end{lemma}

\begin{proof} Assume that $\Rrk T \le l-1$. We consider an $R$-rank decomposition of $T$ as in \eqref{general form}. Because the family $\{A_1, \dots A_r \}$ is linearly independent and $l \ge r+1$, there exists a function $u: \prod_{\a \in C^c} Q_{\a} \rightarrow \F$ supported inside $\{y_1, \dots, y_l \}$ such that $u \neq 0$ but $u.A_i = 0$ for each $i \in \lbrack r \rbrack$. Taking the dot product with $u$ in the decomposition \eqref{general form} we find that \[ u.T = \sum_{i=1}^{r'} u.F_i(x). \] Using Lemma \ref{rank at most 1} we have $\pr u.F_i \le l$ for each $i \in \lbrack r' \rbrack$. We therefore conclude $\pr u.T \le lr' \le l(l-1)$.\end{proof}

To prepare for the proof of Proposition \ref{Case 2, general case}, we write the following approximation of the slices of a tensor with bounded $R$-rank.

\begin{lemma} \label{Spanning result, general case} Let $d \ge 2$ be an integer, let $R \neq R_{\tr}$ be a non-empty family of partitions of $\lbrack d \rbrack$, and let $C$ be a largest part for $R$. Let $l \ge 1$ be an integer. If $T$ is an order-$d$ tensor with $\Rrk T \le l$ then there exist functions $C_1,\dots ,C_l: \prod_{\a \in C^c} Q_{\a} \rightarrow \mathbb{F}$ with $\Rcomp$-rank at most $1$ and functions $D_1,\dots ,D_l: \prod_{\a \in C} Q_{\a} \rightarrow \mathbb{F}$ such that \[ \pr \left(T_{y} - \sum_{i=1}^l C_i(y) D_i\right) \le l \] for every $y \in \prod_{\a \in C^c} Q_{\a}$.\end{lemma}

\begin{proof} We consider an $R$-rank decomposition of $T$ as in \eqref{general form}. As shown by Lemma \ref{rank at most 1}, for each $i \in \lbrack r' \rbrack$ and each $y \in \prod_{\a \in C^c} Q_{\a}$ we have $\pr (F_i)_y \le 1$. The result follows by subadditivity of the partition rank. \end{proof}

We now deduce the following condition, which ensures that a tensor has high $R$-rank and which we will use in the second of the three main cases of our proof. That condition generalises Proposition \ref{Case 2 proposition}.

\begin{proposition}\label{Case 2, general case} Let $d \ge 2$ be an integer, let $R \neq R_{\tr}$ be a non-empty family of partitions of $\lbrack d \rbrack$, let $C$ be a largest part for $R$, and let $M,m,l',l \ge 1$ be integers such that $M\ge(m+l)l$. Let $T: \prod_{\a=1}^d Q_{\a} \rightarrow \mathbb{F}$ be an order-$d$ tensor, and suppose that $A_1,\dots ,A_{l'}: \prod_{\a \in C^c} Q_{\a} \rightarrow \mathbb{F}$ and $B_1,\dots ,B_{l'}: \prod_{\a \in C} Q_{\a} \rightarrow \mathbb{F}$ are functions such that the three following conditions hold.
\begin{enumerate}[(i)]
\item For all $a \in \mathbb{F}^{l'} \setminus \{0\}$,  $\pr \left(\sum_{i=1}^{l'} a_i B_i \right) \ge M$.
\item There exists $j \in \lbrack l' \rbrack$ such that $\Rcomprk (A_j) \ge l$.
\item For all $y \in \prod_{\a \in C^c} Q_{\a}$,  $\pr \left(T_y - \sum_{i=1}^{l'} A_i(y) B_i \right) \le m$.
\end{enumerate}
Then $\Rrk T \ge l$.

\end{proposition}

\begin{proof}

Assume for contradiction that $\Rrk T < l$. Then by Lemma \ref{Spanning result, general case} there exist functions $C_1,\dots ,C_{l-1}: \prod_{\a \in C^c} Q_{\a} \rightarrow \mathbb{F}$ and $D_1,\dots ,D_{l-1}: \prod_{\a \in C} Q_{\a} \rightarrow \mathbb{F}$ such that $\Rcomprk C_i = 1$ for each $i \in \lbrack l-1 \rbrack$ and such that \begin{equation} \pr \left(T_y - \sum_{i=1}^{l-1} C_i(y) D_i \right) \le l-1 \label{approximation used in proof of case 2 in general proof} \end{equation} for every $y \in \prod_{\a \in C^c} Q_{\a}$. By (ii), there exists some $j \in \lbrack l' \rbrack$ such that $\Rcomprk A_j \ge l$. Since $\Rcomprk C_i = 1$ for all $i \in \lbrack l-1 \rbrack$, $A_j$ does not belong to the the linear span of $C_1,\dots ,C_{l-1}$, so there exists a function $u:\prod_{\a \in C^c} Q_{\a} \rightarrow \mathbb{F}$ supported in a subset $U$ of $\prod_{\a \in C^c} Q_{\a}$ with size at most $l$ such that $u.C_i = 0$ for each $i \in \lbrack l-1 \rbrack$ but $u.A_j \neq 0$. We can write $u.T = \sum_{y \in U} u(y) T_y$. On the one hand, applying \eqref{approximation used in proof of case 2 in general proof}, subadditivity, and the fact that $u.C_i = 0$ for each $i \in \lbrack l-1 \rbrack$ we obtain the bound \begin{equation} \pr (u.T) \le (l-1) |U| \le (l-1)l. \label{upper bound on projected rank in case 2 lemma in general case} \end{equation} On the other hand by assumption (iii) and by subadditivity of the partition rank, we have \[ \pr \left( u.T - \sum_{i=1}^{l'} (u.A_i) B_i \right) \le m |U| \le ml. \] Since $u.A_j \neq 0$, by applying assumption (i) to $(a_1,\dots,a_{l'})=(u.A_1,\dots,u.A_{l'})$ we have $\pr \sum_{i=1}^{l'} (u.A_i) B_i \ge M$, so by subadditivity \[ \pr (u.T) \ge M - ml. \] From this inequality and \eqref{upper bound on projected rank in case 2 lemma in general case} we find that $(l-1)l \ge M - ml$, so $M \le (m+l-1)l$, a contradiction.\end{proof}

We next formulate our equivalence statement between the $R$-rank and $R'$-rank.

\begin{proposition} \label{Equivalence} Let $d \ge 2$ be an integer and let $R \neq R_{\tr}$ be a non-empty family of partitions of $\lbrack d \rbrack$. Let $C$ be a largest part for $R$, let $R^\prime$ be the down-shadow of $R$ with respect to $C$, and let $l,m \ge 1$ be integers. Let $T: \prod_{\a=1}^d Q_{\a} \rightarrow \mathbb{F}$ be an order-$d$ tensor such that $\Rrk T \le l$. Assume that \[ \pr T_y \le m\] for every $y \in \prod_{\a \in C^c} Q_{\a}$. Then $\Rdashrk T \le l(lm+l^2+1)$.

\end{proposition}

\begin{proof} We consider an $R$-rank decomposition of $T$ with minimal length as in \eqref{general form}. Because $\{A_1,\dots ,A_r\}$ is linearly independent we can find a subset $U$ of $\prod_{\a \in C^c} Q_{\a}$ with size $r \le l$ and functions $u_i: \prod_{\a \in C^c} Q_{\a} \rightarrow \mathbb{F}$ for each $i \in \lbrack r \rbrack$ all supported inside $U$ and such that $u_i.A_{i'} = \delta_{i,i'}$ for all $i,i' \in \lbrack r \rbrack$. For each $i \in \lbrack r \rbrack$, taking the dot product with $u_i$ in \eqref{general form} we obtain \begin{equation} B_i = u_i.T - \sum_{i'=1}^{r'} u_i.F_{i'} = \sum_{y \in U} u_i(y) T_{y} - \sum_{i'=1}^{r'} \sum_{y \in U} u_i(y) (F_{i'})_y. \label{expression of $B_i$ in general equivalence proof} \end{equation} By the assumption of the proposition we have that $\pr T_{y} \le m$ for each $y \in U$, so \[ \pr \sum_{y \in U} u_i(y) T_{y} \le |U|m \le lm. \] For each $i' \in \lbrack r' \rbrack$, by Lemma \ref{rank at most 1} we have $\pr (F_{i'})_y \le 1$, and hence \[ \pr \sum_{i'=1}^{r'} \sum_{y \in U} u_i(y) (F_{i'})_y \le r' |U| \le l^2. \] So by \eqref{expression of $B_i$ in general equivalence proof},  $\pr B_i \le lm + l^2$. For each $i \in \lbrack r \rbrack$ we have $\Rrk A_i B_i \le 1$ and furthermore $\pr B_i \le lm+l^2$, so $\Rdashrk A_i B_i \le lm+l^2$. For each $i' \in \lbrack r \rbrack$ we have $\Rminusrk F_{i'} \le 1$ so $\Rdashrk F_{i'} \le 1$. Using \eqref{general form} and subadditivity we obtain $\Rdashrk T \le r(lm + l^2) + r' \le l(lm+l^2) + l = l(lm+l^2+1)$ as desired. \end{proof}

The next lemma encapsulates an approximation process that we shall use in the proof of Proposition \ref{subtensors inductive step, general case}.

\begin{lemma}\label{process with radii} Let $d' \ge 2, l \ge 1$ be integers, and let $(T_y)_{y \in Y}$ be a family of order-$d'$ tensors indexed by some finite set $Y$. Let $D: \lbrack l \rbrack \rightarrow \mathbb{R}_+$ be a decreasing function. Assume that there do not exist $y_1, \dots, y_l \in Y$ satisfying \[ \pr \left( \sum_{i=1}^{l} a_i T_{y_i} \right) \ge D(l) \] for all $a \in \F^l \setminus \{ 0 \}$. Then there exist $l' \in \{0, \dots, l-1\}$ and a subset $Y'=\{y_1,\dots ,y_{l'}\}$ of $Y$ such that the two following conclusions hold.
\begin{enumerate}[(i)]
\item For all $(a_1,\dots ,a_{l'}) \in \mathbb{F}^{l'} \setminus \{0\}$ we have $\pr \left(\sum_{i=1}^{l'} a_i T_{y_i}\right) \ge D(l')$.
\item For all $y \in Y$, there exist $a_1(y),\dots a_{l'}(y) \in \F$ such that $\pr \left(T_y - \sum_{i=1}^{l'} a_i(y) T_{y_i}\right) \le D(l'+1)$.
\end{enumerate}
Furthermore, all the above remains true if the partition rank is replaced in both the assumption and the conclusions by the essential partition rank.

\end{lemma}

\begin{proof} We construct the set $Y'$ inductively as follows. If we can find $y_1 \in Y$ such that $\pr T_{y_1} \ge D(1)$ then we continue, and otherwise we stop. Thereafter, if we can find $y_2$ such that \[ \pr (a_1 T_{y_1} + a_2 T_{y_2}) \ge D(2) \] for every $(a_1,a_2) \in \mathbb{F}^2 \setminus \{0\}$ then we continue, and otherwise we stop. For each $j \in \lbrack l \rbrack$, at the $j$th step, provided that we have continued up to this step, if we can find $y_j$ such that \begin{equation} \pr \left (\sum_{i=1}^j a_i T_{y_i} \right) \ge D(j) \label{condition for the process with radii} \end{equation} for all $(a_1,\dots ,a_j) \in \mathbb{F}^j \setminus \{0\}$ then we continue, and otherwise we stop. The process stops after at most $l-1$ steps: if it did continue up to and including the $l$th step then since $D$ is decreasing the resulting family $\{y_1, \dots y_l\}$ would contradict the assumption. Let $l'$ be the total number of iterations completed. That (i) is satisfied follows from the fact that $D$ is decreasing, and that (ii) is satisfied follows from the criterion which led to stopping at step $l'+1$. This finishes the proof for all $y \in Y \setminus Y'$. For $y \in Y'$, the result is immediate. This finishes the proof for partition rank. For the essential partition rank, the proof is identical. \end{proof}

Together with Proposition \ref{Multi-tensor Rrk subtensors}, Proposition \ref{subtensors inductive step, general case} will constitute the main inductive step used to establish Theorem \ref{Subtensors theorem}.

\begin{proposition} \label{subtensors inductive step, general case} Let $d \ge 2$ be an integer, let $R \neq R_{\tr}$ be a non-empty family of partitions of $\lbrack d \rbrack$, let $C$ be a largest part for $R$, and let $R'$ be the down-shadow of $R$ with respect to $C$. If Proposition \ref{Multi-tensor Rrk subtensors} holds for $(|C|, \pr)$, Theorem \ref{Subtensors theorem} holds for $(d-|C|,\Rcomp)$ and Theorem \ref{Subtensors theorem} holds for $(d,R')$ then Theorem \ref{Subtensors theorem} holds for $(d,R)$. \end{proposition}

\begin{proof} Let $l$ be a fixed positive integer, and let $k$ be a large positive integer that we will fix later depending on $l$ and let $T$ be an order-$d$ tensor with $\Rrk T \ge k$.

We distinguish three cases.
\medskip

\noindent \textbf{Case 1}: There exist $y_1,\dots ,y_l \in \prod_{\a \in C^c} Q_{\a}$ such that the order-$|C|$ slices $T_{y_1},\dots ,T_{y_l}$ of $T$ satisfy \[ \pr \left(\sum_{i=1}^l a_i T_{y_i}\right) \ge G_{|C|,\pr, l}(l^2) \] for all $a \in \mathbb{F}^{l} \setminus \{0\}$. By Proposition \ref{Multi-tensor Rrk subtensors} for $(|C|, \pr)$ we can find sets $X_{\a}: \a \in C$ all with size at most $F_{|C|,\pr, l}(l^2)$ such that \[ \pr \left( \sum_{i=1}^l a_i T_{y_i}\right) \left(\prod_{\a \in C} X_{\a}\right) \ge l^2 \] for every $a \in \mathbb{F}^{l} \setminus \{0\}$. For each $\a \in C^c$, let the set $X_{\a}$ be the image of the canonical projection of $\{y_1,\dots ,y_l\}$ on to the $\a$th coordinate. Then we obtain sets $X_{\a}$, $\a \in C^c$ all with size at most $l$. By Lemma \ref{separation lemma} we conclude that $\Rrk T\left(\prod_{\a=1}^d X_{\a}\right) \ge l$.
\medskip

Let us now assume that we are not in Case 1. We will then split into two more cases, and to prepare for that we need to define some auxiliary functions. Let $D: \lbrack l \rbrack \rightarrow \mathbb{R}_+$ be the decreasing function defined by $D(l) = G_{|C|,\pr, l}(l^2)$ and for each $l' \in \lbrack l-1 \rbrack$, $D(l') = G_{|C|, \pr, l}(lD(l'+1) + l^2)$. Since we are not in Case 1, by Lemma \ref{process with radii} applied to $Y= \prod_{\a \in C^c} Q_{\a}$ there exist $l' \in \{0,\dots, l-1\}$ and $y_1,\dots,y_{l'} \in Y$ such that the two following statements hold.
\begin{enumerate}[(i)]
\item For all $(a_1,\dots ,a_{l'}) \in \mathbb{F}^{l'} \setminus \{0\}$ we have the inequality $\pr \left( \sum_{i=1}^{l'} a_i T_{y_i}\right) \ge D(l')$.
\item For each $y \in \prod_{\a \in C^c} Q_{\a}$ there exist $A_1(y),\dots A_{l'}(y) \in \F$ satisfying the inequality $\pr \left( T_{y} - \sum_{i=1}^{l'} A_i(y) T_{y_i}\right) \le D(l'+1)$.
\end{enumerate}
\medskip

\noindent \textbf{Case 2}: There exists $j \in \lbrack l' \rbrack$ satisfying \[ \Rcomprk A_j \ge G_{d-|C|, \Rcomp}(l). \] By Theorem \ref{Subtensors theorem} for $(d-|C|, \Rcomp)$ subtensors, there exist $X_{\a}$, $\a \in C^c$ each with size at most $F_{d-|C|, \Rcomp}(l)$ such that \[\Rcomprk A_j \left( \prod_{\a \in C^c} X_{\a} \right) \ge l. \] Moreover by (i) and Proposition \ref{Multi-tensor Rrk subtensors} for $(|C|, \pr)$ there exist $X_{\a}: \a \in C$ each with size at most $F_{|C|, \pr, l}(lD(l'+1) + l^2)$ such that \[ \pr \left(\sum_{i=1}^{l'} a_i T_{y_i}\right)\left(\prod_{\a \in C} X_{\a}\right) \ge lD(l'+1) + l^2\] for every $(a_1,\dots ,a_{l'}) \in \mathbb{F}^{l'} \setminus \{0\}$. For each $i \in [l']$ let $B_i = T_{y_i}$. Applying Proposition \ref{Case 2, general case} to the tensor $T\left(\prod_{\a=1}^d X_{\a}\right)$, the functions $A_i\left(\prod_{\a \in C^c} X_{\a}\right)$ for each $i \in \lbrack l' \rbrack$, the functions $B_i\left(\prod_{\a \in C} X_{\a}\right)$ for each $i \in \lbrack l' \rbrack$, and the parameters $m = D(l'+1)$ and $M = l(m+l)$, we obtain $\Rrk T' \ge l$.
\medskip

\noindent \textbf{Case 3}: We are not in Case 1 or in Case 2. Let $S$ be the tensor defined by \[ S(x)= \sum_{i=1}^{l'} A_i(x(C^c)) T_{y_i}(x(C)). \] Since $l' \le l$ and for each $i \in \lbrack l' \rbrack$ we have $\Rcomprk A_i \le G_{d-|C|, \Rcomp}(l)$, we get $\Rrk S \le l \cdot G_{d-|C|, \Rcomp}(l)$. The tensor $U=T-S$ is such that for each $y \in \prod_{\a \in C^c} Q_{\a}$ the slice $U_y$ has partition rank at most $D(l'+1) \le D(1)$. Moreover by subadditivity $\Rrk U \ge k - lG_{d-|C|, \Rcomp}(l)$. Since every tensor with $R'$-rank at most $1$ also has $R$-rank at most $1$, we have $\Rdashrk U \ge \Rrk U \ge k - lG_{d-|C|, \Rcomp}(l)$. Applying Theorem \ref{Subtensors theorem} for $(d, R')$ to $U$ we obtain that for \[ k \ge G_{d,R'}(4 (l + lG_{d-|C|, \Rcomp}(l))^3 D(1)) + lG_{d-|C|, \Rcomp}(l) \] we can find $X_1,\dots ,X_d$ with size at most \[ F_{d,R'}(4 (l + lG_{d-|C|, \Rcomp}(l))^3 D(1)) \] such that \[ \Rdashrk U\left(\prod_{\a=1}^d X_{\a}\right) \ge 4 (l + lG_{d-|C|, \Rcomp}(l))^2 D(1). \]  Because $U\left(\prod_{\a=1}^d X_{\a}\right)$ is a restriction of $U$, it is still the case that the slices $U\left(\prod_{\a=1}^d X_{\a}\right)_y$ with $y \in \prod_{\a \in C^c} X_{\a}$ of this tensor all have partition rank at most $D(1)$. Applying Proposition \ref{Equivalence} we obtain \[ \Rrk U \left(\prod_{\a=1}^d X_{\a}\right) \ge l + lG_{d-|C|, \Rcomp}(l) \] and hence \[ \Rrk T\left(\prod_{\a=1}^d X_{\a}\right) \ge l. \qedhere\] \end{proof}

Everything is now in place for the proof of Theorem \ref{Subtensors theorem}.

\begin{proof}[Proof of Theorem \ref{Subtensors theorem}]

We proceed using a double induction on $d$ and $R$. The base case consists in the tensor rank case for all values of $d \ge 2$, which has already been established in Proposition \ref{Subtensors for order-$d$ tensor rank}. Suppose that for some $d \ge 2$, $R \neq R_{\tr}$ is a non-empty family of partitions of $[d]$, with a largest part $C$. Suppose that Theorem \ref{Subtensors theorem} holds for all non-empty families $R'$ of partitions of $[d]$ where \[|\cup_{P \in R'} \{J \in P: |J|=|C|\}| < |\cup_{P \in R} \{J \in P: |J|=|C|\}|.\] We take $R'$ to be down-shadow of $R$ with respect to $C$. By our assumption, Theorem \ref{Subtensors theorem} holds for $(|C|, \pr)$, for $(d-|C|, R_{\mathrm{comp}})$, and for $(d,R')$. We then apply Proposition \ref{Multi-tensor Rrk subtensors} to obtain the multi-tensor version of Theorem \ref{Subtensors theorem} for $(|C|, \pr)$. Proposition \ref{Multi-tensor Rrk subtensors} then establishes that Theorem \ref{Subtensors theorem} holds for $(d,R)$, concluding the inductive step.

\end{proof}

\subsection{The general case for disjoint \texorpdfstring{$R$}{R}-rank}

We begin by proving our equivalence result between the essential $R$-rank and the essential $R'$-rank.

\begin{proposition} \label{Essential equivalence}

Let $d \ge 2$ be an integer, let $R \neq R_{\tr}$ be a non-empty family of partitions of $\lbrack d \rbrack$, let $C$ be a largest part for $R$, and let $d'=|C|$. Let $l,m \ge 1$ be integers. Let $T: \prod_{\a=1}^d Q_{\a} \rightarrow \mathbb{F}$ be an order-$d$ tensor such that $\eRrk T \le l$. Assume that \[\epr T_y \le m\] for every $y \in \left(\prod_{\a \in C^c} Q_{\a}\right) \setminus E(C^c)$. Then $\eRdashrk T \le l^2(m + d'(d-d')) + l^3 + l$.

\end{proposition}

\begin{proof} Because $\eRrk T \le l$ there exists a tensor $V$ supported inside $E$ such that $\Rrk (T - V) \le l$. We consider an $R$-rank decomposition of $T - V$ with minimal length, which by \eqref{general form} we can write \begin{equation} (T-V)(x) = \sum_{i=1}^{r} A_i(x(C^c)) B_i(x(C)) + \sum_{i=1}^{r'} F_{i}(x) \label{decomposition of T-V in essential equivalence proof} \end{equation} where $r+r' \le l$, with $\Rcomprk A_i = 1$ for each $i \in \lbrack r \rbrack$ and $\Rminusrk F_i = 1$ for each $i \in \lbrack r' \rbrack$. 

Let us establish that the restrictions of the functions $A_i$ to $\left(\prod_{\a \in C^c} Q_{\a}\right) \setminus E(C^c)$ are linearly independent. Indeed, these functions are linearly independent if and only if the functions $\tilde{A_i}: \prod_{\a \in C^c} Q_{\a} \to \F$ defined by $\tilde{A_i}(x(C^c)) = 0$ if $x(C^c) \in E(C^c)$ and $\tilde{A_i}(x(C^c)) = A_i(x(C^c))$ otherwise are linearly independent. The differences $A_i - \tilde{A_i}$ are supported inside $E(C^c)$, so the tensor $\tilde{V}$ defined by \[\tilde{V}(x) = V(x) + \sum_{i=1}^r (A_i - \tilde{A_i})(x(C^c)) B_i(x(C)) \]is supported inside $E$. We have the decomposition \[T - \tilde{V} = \sum_{i=1}^{r} \tilde{A_i}(x(C^c)) B_i(x(C)) + \sum_{i=1}^{r'} F_{i}(x).\] If the functions $\tilde{A_i}$ are linearly dependent then the first sum above has $R$-rank at most $r-1$, and the $R$-rank of $T - \tilde{V}$ is at most $r-1+r'<l$, which entails $\eRrk T < l$. This is excluded, and the functions $\tilde{A_i}$ are therefore linearly independent.

Therefore, we can find a subset $U$ of $\left(\prod_{\a \in C^c} Q_{\a}\right) \setminus E(C^c)$ with size $r \le l$ and functions $u_i: \prod_{\a \in C^c} Q_{\a} \rightarrow \mathbb{F}$ for each $i \in \lbrack r \rbrack$ with supports all contained inside $U$ (so in particular, contained outside $E(C^c)$) such that $u_i.A_{i'} = \delta_{i,i'}$ for every $i,i' \in \lbrack r \rbrack$. For each $i \in \lbrack r \rbrack$, taking the dot product with $u_i$ in \eqref{decomposition of T-V in essential equivalence proof} we obtain \begin{equation} B_i = u_i.(T-V) - \sum_{i'=1}^{r'} u_i.F_{i'} = \sum_{y \in U} u_i(y) (T-V)_{y} - \sum_{i'=1}^{r'} \sum_{y \in U} u_i(y) (F_{i'})_y. \label{Expression of $B_i$ in the general essential equivalence proof} \end{equation} By our assumption, for each $y \in U$ we have $\epr T_{y} \le m$, so there exists an order-$d'$ tensor $V'^y$ supported in \[E(C) = \{x(C): x_{\a'} = x_{\a''} \text{ for some distinct }\a',\a'' \in C\}\] such that $\pr (T_y-V'^y) \le m$. The slice $V_y$ of $V$ has its support contained inside the union of $E(C)$ and of \[ M:=\{x(C): \text{ there exist }\a' \in C \text{ and }\a'' \in C^c \text{ such that } x_{\a'} = y_{\a''}\}.\] We can write $V_y = V_{y,\mathrm{int}} + V_{y,\mathrm{ext}}$ where $V_{y,\mathrm{int}}$ and $V_{y,\mathrm{ext}}$ are the restrictions of $V_y$ to $E(C)$ and $M$. The tensor $V_{y,\mathrm{ext}}$ has support contained in the union of $d'(d-d')$ order-$(d'-1)$ slices of $\prod_{\a \in C} Q_{\a}$, so  $\pr V_{y,\mathrm{ext}} \le d'(d-d')$. The tensor $V''^y= V'^y - V_{y, \mathrm{int}}$ is supported inside $E(C)$ and we have by subadditivity \[ \pr ((T-V)_y - V''^y) \le \pr (T_y - V'^y) + \pr (V_{y, \mathrm{int}} - V_y) = \pr (T_y - V'^y) + \pr (V_{y, \mathrm{ext}}) \le m +d'(d-d'). \] Hence, for each $y \in U$ there exists a tensor $U^y: \prod_{\a \in C} Q_{\a} \rightarrow \F$ such that we can write $(T-V)_y = V''^y + U^y$, where  $\pr U^y \le m + d'(d-d')$. For each $i \in \lbrack r \rbrack$, using \eqref{Expression of $B_i$ in the general essential equivalence proof} we get \[ B_i = \sum_{y \in U} u_i(y)(V''^y + U^y) - \sum_{i'=1}^{r'} \sum_{y \in U} u_i(y) (F_{i'})_y. \] Substituting into \eqref{decomposition of T-V in essential equivalence proof} we obtain \begin{multline} (T-V)(x) = \sum_{i=1}^{r} A_i(x(C^c)) \left( \sum_{y \in U} u_i(y)(V''^y + U^y)(x(C)) - \sum_{i'=1}^{r'} \sum_{y \in U} u_i(y) F_{i'}(y,x(C)) \right) \\ + \sum_{i=1}^{r'} F_{i}(x) = T_1 + T_2 + T_3 + T_4, \end{multline} where

\begin{align*} T_1(x) &= \sum_{i=1}^{r} A_i(x(C^c)) \sum_{y \in U} u_i(y) V''^y(x(C)) \\
T_2(x) &= \sum_{i=1}^{r} A_i(x(C^c)) \sum_{y \in U} u_i(y) U^y(x(C))\\
T_3(x) &= - \sum_{i=1}^{r} A_i(x(C^c)) \sum_{i'=1}^{r'} \sum_{y \in U} u_i(y) F_{i'}(y,x(C))\\
T_4(x) &= \sum_{i=1}^{r'} F_{i}(x). \end{align*}

For each $y \in U$, the support of $V''^y$ is contained in $E(C)$, so the support of $T_1$ is contained in $\left(\prod_{\a \in C^c} Q_{\a}\right) \times E(C)$ and hence contained in $E$. For each $y \in U$ and each $i \in \lbrack r \rbrack$ we have $\Rcomprk A_i \le 1$ and $\pr U^y \le m + d'(d-d')$, so $\Rdashrk T_2 \le r |U| (m + d'(d-d')) \le l^2(m + d'(d-d'))$. For each $y \in U$, each $i \in \lbrack r \rbrack$, and each $i' \in \lbrack r' \rbrack$ we have $\Rcomprk A_i \le 1$ and by Lemma \ref{rank at most 1} we have $\pr (F_{i'})_y \le 1$, so $\Rdashrk(A_i(F_{i'})_y) \le 1$, so $\Rdashrk T_3 \le rr' |U| \le l^3$. For each $i' \in \lbrack r' \rbrack$ we have $\Rminusrk F_{i'} \le 1$ so $\Rdashrk F_{i'} \le 1$, and therefore $\Rdashrk T_4 \le r' \le l$. It follows that $T$ coincides outside of $E$ with a tensor $(T_2 + T_3 + T_4)$ which has $R'$-rank at most $l^2(m + d'(d-d')) + l^3 + l$, so $\eRdashrk T \le l^2(m + d'(d-d')) + l^3 + l$. \end{proof}

From Proposition \ref{Essential equivalence} we can deduce the following corollary, from which Proposition \ref{Equivalence between $efrank_1$ and $etr$, assuming that $epr$ is bounded for slices of all sizes} and Proposition \ref{A tensor with bounded essential partition ranks of all slices of all sizes has bounded essential tensor rank} from Section \ref{section: Disjoint tensor rank} can be deduced.

\begin{corollary} \label{Equivalence of all essential ranks} Let $d \ge 2$ be an integer, and let $R \neq R_{\tr}$ be a non-empty family of partitions of $\lbrack d \rbrack$. Let $l \ge d^2$ be an integer. If $T: \prod_{\a=1}^d Q_{\a} \rightarrow \F$ is an order-$d$ tensor such that $\eRrk T \le l$ and furthermore $\epr T_y \le l$ for every $I \subset \lbrack d \rbrack$ with $|I| \in \{1,\dots,d-2\}$ and for every $y \in \left(\prod_{\a \in I} Q_{\a}\right) \setminus E(I)$, then $\etr T \le (4l^3)^{2^d}$. \end{corollary}

\begin{proof} We repeatedly apply Proposition \ref{Essential equivalence}: we let $R_1 = R$, then inductively define families $R_2, R_3, \dots$ of partitions of $\lbrack d \rbrack$ as follows: as long as $R_i \neq R_{\tr}$ we choose $C_i$ a largest part for $R_i$ and obtain the down-shadow $R_{i+1}=R_i'$ of $R_i$ with respect to $C_i$. We stop the process when $R_i = R_{\tr}$, which necessarily occurs after at most $2^d$ iterations, since the $i$th iteration rules out the set $C_i$ from all partitions of all $R_{i'}$ with $i' > i$, and the sets $C_i$ have non-increasing size. Using the assumption $l \ge d^2$ it is simple to check that whenever $m \ge l$ we have $l^2(m + d'(d-d')) + l^3 + l \le 4l^2m$. The claim follows. \end{proof}

For $d \ge 2$ an integer and $T: Q_1 \times \dots \times Q_d \rightarrow \F$ an order-$d$ tensor let \[Z(T)= \{(x_1, \dots, x_d) \in Q_1 \times \dots \times Q_d: T(x_1, \dots, x_d) \neq 0 \} \] be the support of $T$, and let $eZ(T)= Z(T) \setminus E$. In the proof of Proposition \ref{disjoint rank, inductive step} we shall use the following result on $d$-uniform hypergraphs.

\begin{lemma} \label{Disjoint support size}

Let $d \ge 2$, $k \ge 1$ be integers, and let $T: Q_1 \times \dots \times Q_d \rightarrow \F$ be an order-$d$ tensor such that $|eZ(T)| \ge k$. Then there exist $X_1,\dots ,X_d$ pairwise disjoint such that \[|Z(T(X_1 \times \dots \times X_d))| \ge k/d^d.\] \end{lemma}

\begin{proof} Without loss of generality we may assume that $Q_1 = \dots  = Q_d$. Indeed, letting $Q= \bigcup_{1 \le \a \le d} Q_\a$ we can extend $T$ to a tensor $T'$ supported on $Q^d$ such that $Z(T') = Z(T)$ and $Z(T') \setminus E = Z(T) \setminus E$ by setting all new entries to take the value $0$. Provided that the claim holds for $T'$, with sets $X_1',\dots ,X_d'$ we then deduce it for $T$ by taking $X_1 = X_1' \cap Q_1$,\dots, $X_d = X_d' \cap Q_d$, since $Z(T(X_1 \times \dots  \times X_d)) = Z(T'(X_1 \times \dots  \times X_d))$.

We now assume that $Q_1 = \dots = Q_d = Q$. For each $u \in Q$ we send $u$ to a set $X_{\a(u)}$ by choosing the $\a(u) \in \lbrack d \rbrack$ independently and uniformly at random. For each $x \in Z(T) \setminus E$, the probability that $x \in Z(T(X_1 \times \dots \times X_d))$ is equal to $1/d^d$, so the expected size of $Z(T(X_1 \times \dots  \times X_d))$ is $eZ(T) / d^d$, and in particular for at least one of the choices of $(\a(u): u \in Q)$ this is the case. \end{proof}

Together with Proposition \ref{Multi-tensor disjoint $pr$}, Proposition \ref{disjoint rank, inductive step} will constitute the main inductive step used to establish Theorem \ref{Disjoint rank subtensors theorem}.

\begin{proposition} \label{disjoint rank, inductive step} Let $d \ge 2$ be an integer, let $R \neq R_{\tr}$ be a non-empty family of partitions of $\lbrack d \rbrack$, let $C$ be a largest part for $R$, and let $R'$ be the down-shadow of $R$ with respect to $C$. If Proposition \ref{Multi-tensor disjoint $pr$} holds for $(|C|, \pr)$, Theorem \ref{Disjoint rank subtensors theorem} holds for $(d-|C|, \Rcomp)$, Theorem \ref{Subtensors theorem} holds for $(d-|C|, \Rcomp)$, and Theorem \ref{Disjoint rank subtensors theorem} holds for $(d,R')$, then Theorem \ref{Disjoint rank subtensors theorem} holds for $(d,R)$. \end{proposition}

\begin{proof} Let $l$ be a fixed positive integer, let $k$ be a large positive integer that we shall fix later depending on $l$, and let $T$ be an order-$d$ tensor with $\eRrk T \ge k$. We distinguish three cases.
\medskip

\noindent \textbf{Case 1}: There exist $y_1,\dots,y_{d^dl} \in \left(\prod_{\a \in C^c} Q_{\a}\right) \setminus E(C^c)$ such that the order-$|C|$ slices $T_{y_1},\dots ,T_{y_{d^dl}}$ of $T$ satisfy \[ \epr \left(\sum_{i=1}^l a_i T_{y_i}\right) \ge G'_{|C|,\pr, l}(l^2) + d^2 l \] for every $a \in \mathbb{F}^{l} \setminus \{0\}$. By Lemma \ref{Disjoint support size} there exist pairwise disjoint $X_{\a}, \a \in C^c$ such that $\{y_1,\dots ,y_{d^dl}\} \cap \prod_{\a \in C^c} X_{\a}$ contains at least $l$ elements, which without loss of generality we can assume to be $y_1, \dots, y_l$. Furthermore (by requiring for each $\a \in C^c$ the set $X_{\a}$ to be the image of $\{y_1,\dots ,y_l\}$ by the canonical projection on the $\a$th coordinate) we can require each of the $X_{\a}$, $\a \in C^c$ to have size at most $l$. By subadditivity \[ \epr \left(\sum_{i=1}^l a_i T_{y_i}\right)\left(\prod_{\a=1}^d (Q_{\a} \setminus \cup_{\a' \in C^c} X_{\a'}) \right) \ge G'_{|C|,\pr, l}(l^2) \] for every $a \in \mathbb{F}^l \setminus \{0\}$. By Proposition \ref{Multi-tensor disjoint $pr$} for $(|C|, \pr)$ we can find pairwise disjoint subsets $X_{\a} \subset Q_{\a} \setminus \left( \bigcup_{\a' \in C^c} X_{\a'} \right)$ for each $\a \in C$ such that \[ \pr \left(\sum_{i=1}^l a_i T_{y_i}\right)\left(\prod_{\a \in C} X_{\a}\right) \ge l^2 \] for every $a \in \mathbb{F}^l \setminus \{0\}$. By construction the sets $X_{\a}$, $\a \in \lbrack d \rbrack$ are pairwise disjoint and by Lemma \ref{separation lemma} we have $\Rrk T\left(\prod_{\a=1}^d X_{\a}\right) \ge l$.
\medskip

Let us now assume that we are not in Case 1. We will then split into two more cases, and to prepare for that we need to define some auxiliary functions. Let $D: \lbrack d^dl \rbrack \rightarrow \mathbb{R}_+$ be the decreasing function defined by $D(d^dl) = G'_{|C|,\pr, l}(l^2) + d^2l$ and $D(l') = G'_{|C|,\pr, d^dl}(lD(l'+1) + l^2) + d^2 F_{d-|C|, \Rcomp}(l)$ for each $l' \in \lbrack d^dl-1 \rbrack$. Since we are not in Case 1, by Lemma \ref{process with radii} applied with the essential partition rank to $Y= \left(\prod_{\a \in C^c} Q_{\a}\right) \setminus E(C^c)$ there exist $l' \in \{0,\dots, d^dl-1\}$ and $\{y_1,\dots,y_{l'}\} \in \left(\prod_{\a \in C^c} Q_{\a}\right) \setminus E(C^c)$ such that the two following statements hold.
\begin{enumerate}[(i)]
\item For all $(a_1,\dots ,a_{l'}) \in \mathbb{F}^{l'} \setminus \{0\}$ we have the inequality $\epr \left(\sum_{i=1}^{l'} a_i T_{y_i}\right) \ge D(l')$.
\item For each $y \in \left (\prod_{\a \in C^c} Q_{\a}\right) \setminus E(C^c)$ there exist $A_1(y),\dots A_{l'}(y) \in \F$ which satisfy $\epr \left(T_{y} - \sum_{i=1}^{l'} A_i(y) T_{y_i}\right) \le D(l'+1)$.
\end{enumerate}

\medskip

\noindent \textbf{Case 2}: There exists $j \in \lbrack l' \rbrack$ such that \[\eRcomprk A_j \ge G_{d-|C|, \Rcomp}'(G_{d-|C|, \Rcomp}(l)).\] By Theorem \ref{Disjoint rank subtensors theorem} for $(d-|C|, \Rcomp)$, there exist pairwise disjoint $X_{\a,pre} \subset Q_{\a}$ for each $\a \in C^c$ such that \[ \Rcomprk A_j\left(\prod_{\a \in C^c} X_{\a,pre}\right) \ge G_{d-|C|, \Rcomp}(l). \] By Theorem \ref{Subtensors theorem} for $(d-|C|, \Rcomp)$, there exist subsets $X_{\a} \subset X_{\a,pre}$ for each $\a \in C^c$ all with size at most $F_{d-|C|, \Rcomp}(l)$ such that \begin{equation} \Rcomprk A_j\left(\prod_{\a \in C^c} X_{\a}\right) \ge l. \label{subtensors for Rcomp in the general disjoint rank case} \end{equation} By (i) and subadditivity \[ \epr \left(\sum_{i=1}^{l'} a_i T_{y_i}\right) \left(\prod_{\a \in C} \left(Q_{\a} \setminus \bigcup_{\a' \in C^c} X_{\a'}\right)\right) \ge G'_{|C|,\pr, d^dl}(lD(l'+1) + l^2) \] for every $(a_1,\dots ,a_{l'}) \in \mathbb{F}^{l'} \setminus \{0\}$. By Proposition \ref{Multi-tensor disjoint $pr$} for $(|C|, \pr)$ there exist subsets $X_{\a} \subset Q_{\a} \setminus \left(\bigcup_{\a' \in C^c} X_{\a'}\right)$ for each $\a \in C$ which are pairwise disjoint and such that \begin{equation} \pr \left(\sum_{i=1}^{l'} a_i T_{y_i}\right)\left(\prod_{\a \in C} X_{\a}\right) \ge lD(l'+1) + l^2 \label{subtensor with separated case in case 2 of the general disjoint rank case} \end{equation} for every $(a_1,\dots ,a_{l'}) \in \mathbb{F}^{l'} \setminus \{0\}$. By \eqref{subtensor with separated case in case 2 of the general disjoint rank case}, \eqref{subtensors for Rcomp in the general disjoint rank case} and (ii) and applying Proposition \ref{Case 2, general case} to the tensor $T'=T\left(\prod_{\a=1}^d X_{\a}\right)$, the functions $A_i'=A_i\left(\prod_{\a \in C^c} X_{\a}\right)$ for each $i \in \lbrack l' \rbrack$, the functions $B_i'= B_i\left(\prod_{\a \in C^c} X_{\a}\right)$ for each $i \in \lbrack l' \rbrack$ and the parameters $m = D(l'+1)$ and $M = l(m+l)$, we obtain $\Rrk T' \ge l$.

\medskip

\noindent \textbf{Case 3}: We are not in Case 1 or in Case 2. Let $S$ be the tensor defined by \[ S(x)= \sum_{i=1}^{l'} A_i(x(C^c)) T_{y_i}(x(C)). \] Since for each $i \in \lbrack l' \rbrack$ we have $\eRcomprk A_i \le G_{d-|C|, \Rcomp}'(G_{d-|C|, \Rcomp}(l))$, we get \[ \eRrk S \le d^dlG_{d-|C|, \Rcomp}'(G_{d-|C|, \Rcomp}(l)). \] The tensor $U=T-S$ is such that for each $y \in \left(\prod_{\a \in C^c} X_{\a}\right) \setminus E(C^c)$ the $C$-slice $U_y$ satisfies $\epr U_y \le D(l'+1) \le D(1)$. Moreover by subadditivity \[ \eRrk U \ge k - d^dlG_{d-|C|, \Rcomp}'(G_{d-|C|, \Rcomp}(l)). \] Since every tensor with $R'$-rank at most $1$ also has $R$-rank at most $1$, we have $\eRdashrk U \ge \eRrk U \ge k - d^dlG_{d-|C|, \Rcomp}'(G_{d-|C|, \Rcomp}(l))$. Applying Theorem \ref{Disjoint rank subtensors theorem} for $(d, R')$ to $U$ we obtain that for \[ k \ge G'_{d,R'}\big(4d^2(l + d^dlG_{d-|C|, \Rcomp}'(G_{d-|C|, \Rcomp}(l)))^3D(1)\big) + d^dlG_{d-|C|, \Rcomp}'(G_{d-|C|, \Rcomp}(l)) \] we can find pairwise disjoint $X_1,\dots ,X_d$ such that \[ \Rdashrk U \left(\prod_{\a=1}^d X_{\a} \right) \ge 4d^2(l + d^dlG_{d-|C|, \Rcomp}'(G_{d-|C|, \Rcomp}(l)))^3D(1). \] Because $U(\prod_{\a=1}^d X_{\a})$ is a restriction of $U$, it is still the case that the $C$-slices of this tensor all have essential partition rank at most $D(1)$. Applying Proposition \ref{Essential equivalence} we obtain \[ \Rrk U \left(\prod_{\a=1}^d X_{\a}\right) \ge l + d^dlG_{d-|C|, \Rcomp}'(G_{d-|C|, \Rcomp}(l)) \] and hence \[ \Rrk T\left(\prod_{\a=1}^d X_{\a}\right) \ge l. \qedhere\] \end{proof}

Everything is now in place for the proof of Theorem \ref{Disjoint rank subtensors theorem}.

\begin{proof}[Proof of Theorem \ref{Disjoint rank subtensors theorem}]

We proceed using a double induction on $d$ and $R$. The base case consists in the tensor rank case for all values of $d \ge 2$, which has already been established in Proposition \ref{Disjoint tensor rank}. Suppose that for some $d \ge 2$, $R \neq R_{\tr}$ is a non-empty family of partitions of $[d]$, with a largest part $C$. Suppose that Theorem \ref{Disjoint rank subtensors theorem} holds for all non-empty families $R'$ of partitions of $[d]$ where \[|\cup_{P \in R'} \{J \in P: |J|=|C|\}| < |\cup_{P \in R} \{J \in P: |J|=|C|\}|.\] We take $R'$ to be down-shadow of $R$ with respect to $C$. By our assumption, Theorem \ref{Disjoint rank subtensors theorem} holds for $(|C|, \pr)$, for $(d-|C|, R_{\mathrm{comp}})$, and for $(d,R')$. We then apply Proposition \ref{Multi-tensor disjoint $pr$} to obtain the multi-tensor version of Theorem \ref{Disjoint rank subtensors theorem} for $(|C|, \pr)$. Proposition \ref{Multi-tensor disjoint $pr$} then establishes that Theorem \ref{Disjoint rank subtensors theorem} holds for $(d,R)$, concluding the inductive step.

\end{proof}

\section{Open problems}\label{section: Open problems}

The results in this paper still leave open a number of related strengthenings. The bounds that we obtain for Theorem \ref{Subtensors theorem} are very probably suboptimal: the functions $F_{d,R}$ and $G_{d,R}$ that we obtain are merely those resulting from the current proof, and nothing particularly suggests that they are close to the optimal bounds. On the contrary, our guess would be that it is possible to take both $F_{d,R}(l)$ and $G_{d,R}(l)$ to be linear in $l$.

\begin{conjecture} \label{Linear bounds for subtensors} Let $d \ge 2$ be an integer and let $R$ be a non-empty family of partitions of $\lbrack d \rbrack$. Then there exist constants $A(d,R), B(d,R) > 0$ such that whenever $T: \prod_{\a=1}^d Q_{\a} \rightarrow \mathbb{F}$ is an order-$d$ tensor with $\Rrk T \ge l$, there exist $X_1 \subset Q_1, \dots, X_d \subset Q_d$ with size at most $A(d,R) l$ such that

\[ \Rrk T(X_1 \times \dots  \times X_d) \ge B(d,R) l. \] \end{conjecture}

Aside from an improvement of our current bounds to linear bounds, there are two additional statements involving subtensors that we would expect to be true. The first involves obtaining a subtensor of bounded size with the same rank as the original tensor, and the second involves obtaining a full-rank subtensor (that is, of rank equal to the sizes of each of the $d$ sets) of size tending to infinity with the rank of the original tensor.

\begin{conjecture} \label{Type 2 subtensors} Let $d \ge 2$ be an integer, and let $R$ be a non-empty family of partitions of $\lbrack d \rbrack$. Then there exists a function $F_{d,R,\mathrm{same}}: \mathbb{N} \rightarrow \mathbb{N}$ such that whenever $T: \prod_{\a=1}^d Q_{\a} \rightarrow \mathbb{F}$ is an order-$d$ tensor with $\Rrk T \ge l$, there exist $X_1 \subset Q_1, \dots, X_d \subset Q_d$ with size at most $F_{d,R,\mathrm{same}}(l)$ such that

\[ \Rrk T(X_1 \times \dots  \times X_d) \ge l. \] \end{conjecture}

\begin{conjecture} \label{Type 1 subtensors} Let $d \ge 2$ be an integer, and let $R$ be a non-empty family of partitions of $\lbrack d \rbrack$. There exists a function $G_{d,R,\mathrm{\mlvar{full}}}: \mathbb{N} \rightarrow \mathbb{N}$ such that whenever $T: \prod_{\a=1}^d Q_{\a} \rightarrow \mathbb{F}$ is an order-$d$ tensor with $\Rrk T \ge G_{d,R,\mathrm{\mlvar{full}}}(l)$, there exist $X_1 \subset Q_1, \dots, X_d \subset Q_d$ with size at most $l$ such that

\[ \Rrk T(X_1 \times \dots  \times X_d) \ge l. \] \end{conjecture}

It is worth noting that Conjectures \ref{Linear bounds for subtensors}, \ref{Type 2 subtensors} and \ref{Type 1 subtensors} are still unproved for very simple cases, such as for the slice rank for order-$3$ tensors. As the methods in this paper involve losses rather quickly, we expect that new, more precise methods would be required to make progress on them: arguments with no losses at all in the direction where the bound is sharp will necessarily be involved in any attempt on Conjectures \ref{Type 2 subtensors} and \ref{Type 1 subtensors}. We next recall Conjecture \ref{Multi-tensor type 2 subtensors for the general case} from Section \ref{section: Spanning}, a multi-tensor version of Conjecture \ref{Type 2 subtensors}.

\begin{conjecture} \label{Multi-tensor type 2 subtensors for the general case recalled} Let $ d \ge 2, s \ge 1$ be integers and let $ R$ be a non-empty family of partitions on $ \lbrack d \rbrack$. Then there exists a function $ F_{d,R,s, \mathrm{same}}$ such that for every positive integer $l \ge 1$, whenever $ T_1,\dots ,T_s$ are order-$d$ tensors there exist $ X_1,\dots ,X_d$ of size at most $ F_{d,R,s, \mathrm{same}}(l)$ such that \[ \Rrk (a.T)(X_1 \times \dots  \times X_d) \ge \min(l, \Rrk a.T) \] 
for every $ a \in \mathbb{F}^s$. \end{conjecture}

Still on the topic of strengthening Theorem \ref{Subtensors theorem} but in another direction, we can ask about the dependence of the bounds in the case of several tensors. We would expect that in Proposition \ref{Multi-tensor Rrk subtensors} for fixed $d,R,l$ we can take the dependence of $F_{d,R,s}(l)$ to be linear in $s$ and that we can take $G_{d,R,s}(l)$ to be independent of $s$, as was shown in Proposition \ref{Multi-tensor tensor rank subtensors} to be the case for the tensor rank.

\begin{conjecture} \label{Linear bound in the number of tensors} Let $d \ge 2$ be an integer, and let $R$ be a non-empty family of partitions of $\lbrack d \rbrack$. For every positive integer $l \ge 1$, there exist quantities $F_{d,R}(l)$, $G_{d,R}(l)$ such that whenever $s \ge 1$ is a positive integer, if $T_1, \dots, T_s: \prod_{\a=1}^d Q_{\a} \rightarrow \mathbb{F}$ are order-$d$ tensors such that \[ \Rrk (a.T) \ge G_{d,R}(l)\] for every $a \in \F^s \setminus \{0\}$ then there exist $X_1 \subset Q_1, \dots, X_d \subset Q_d$ each with size at most $sF_{d,R}(l)$ such that \[ \Rrk (a.T)(X_1 \times \dots \times X_d) \ge l\] for every $a \in \F^s \setminus \{0\}$. \end{conjecture}

We note that the bound $F_{d,R,s}(l) \le F_{d,R}(sl) + F_{d,R,s-1}(l)$ resulting from the proof of Proposition \ref{Multi-tensor Rrk subtensors} would give only a quadratic bound in $s$ even if we assume that $F_{d,R}(l)$ can indeed be taken to be linear in $l$.

In the tensor rank case we were able to show that the optimistic bounds in the generalisation of the standard statement on full-rank submatrices of matrices holds indeed do hold. We can ask whether there is a simple characterisation of all non-empty families $R$ of partitions of $\lbrack d \rbrack$ for which this is the case.

\begin{question} \label{Strong subtensors} Let $d \ge 2$ be an integer. For which families $R$ of partitions of $\lbrack d \rbrack$ is it true that if $T: \prod_{\a=1}^d Q_{\a} \rightarrow \mathbb{F}$ is an order-$d$ tensor with $\Rrk T = l$, then there exist $X_1 \subset Q_1, \dots, X_d \subset Q_d$ all with size $l$ such that $\Rrk T(X_1 \times \dots  \times X_d) = l$? \end{question} Regarding Theorem \ref{Disjoint rank subtensors theorem} we can also ask for the optimal lower bounds on the value of the disjoint rank for a given essential rank. We believe that there again, linear bounds should hold.

\begin{conjecture} \label{Optimal bounds on disjoint rank} Let $d \ge 2$ be an integer, and let $R$ be a non-empty family of partitions of $\lbrack d \rbrack$. There exists a constant $D(d,R) > 0$ such that whenever $T: \prod_{\a=1}^d Q_{\a} \rightarrow \mathbb{F}$ is an order-$d$ tensor we have $\dRrk T \ge D(d,R) \eRrk T$. \end{conjecture}

In the case of the slice rank (and, similarly, of the partition rank) a better bound than $\dsr T \ge \esr T/d$ cannot hold, even for large values of $\esr T$: if $Q_1 = \dots = Q_d = \lbrack n \rbrack$ for some large positive integer $n$, then whenever $X_1,\dots,X_d$ are pairwise disjoint subsets of $\lbrack n \rbrack$ we have \[\sr T(X_1 \times \dots  \times X_d) \le \min(|X_1|,\dots , |X_d|) \le n/d,\] but a simple counting argument in any finite field $\F$ shows that there exist tensors $T: \lbrack n \rbrack^d \rightarrow \F$ with essential slice rank $n(1-o(1))$: there are $|\F|^{n^d}$ possible order-$d$ tensors and for every positive integer $k$ the number of possible slice rank decompositions with length at most $k$ is at most $k^3 |\F|^{k(n^{d-1}+n)}$. It however seems plausible that the essential slice (resp. partition) rank of an order-$d$ tensor is not more than approximately $d$ times greater than its disjoint slice (resp. partition) rank, at least in the regime where the essential and disjoint ranks are not both small.

In the case of the tensor rank, for any positive integer $d \ge 2$ and any finite field $\F$ there exists an order-$d$ tensor $[d+1]^d \rightarrow \F$ with disjoint tensor rank at most $1$, but with essential tensor rank at least $(d-1)!$. Indeed if $X_1, \dots, X_d$ are pairwise disjoint subsets of $[d+1]$ then all but at most one of these sets have size at most $1$, and a similar counting argument shows that the coefficients of $T$ can be taken such that its essential tensor rank is at least $(d-1)!$: there are $(d+1)!$ elements in $[d+1]^d \setminus E$ and for every positive integer $k$ the number of coefficients in a tensor rank decomposition with length at most $k$ is equal to $kd(d+1)$.

We have briefly attempted to prove that for $d=2$ we may take the constant to be equal to $2$ in Conjecture \ref{Optimal bounds on disjoint rank}, but found it to be, perhaps unexpectedly, a problem of at least moderate difficulty already.

\begin{conjecture} \label{Constant 1/2 in disjoint matrix rank} Let $A: Q_1 \times Q_2 \rightarrow \mathbb{F}$ be a matrix. Then $\drk A \ge (\erk A)/2$. \end{conjecture}

One more question that seems natural to draw attention to regarding the main theorems of this paper is the extent to which we can relax the requirement that $R$ be a family of partitions of $\lbrack d \rbrack$ to $R$ being an arbitrary non-empty element of $\mathcal{P}(\mathcal{P}(\lbrack d \rbrack))$. Although this can lead to notions of rank that may be behave in somewhat surprising ways (for instance if $d=3$ and $R = \{\{\{1,2\},\{1,3\}\}, \{\{1,2\},\{2,3\}\}, \{\{1,3\},\{2,3\}\}\}$ then the tensor $1_{x=y=z} = 1_{x=y} 1_{x=z}$ has $R$-rank one), we believe that Theorems \ref{Subtensors theorem} and \ref{Disjoint rank subtensors theorem} are probably still true in this case, at least for tensors for which the $R$-rank is well-defined.

\appendix

\section{A simple subtensors argument in the case of rank powers}

Before concluding we would like to devote a short section to discuss a technique which allows us to obtain reasonably good bounds for subtensors when the family of partitions can be factored in the sense that we will now define. (For instance, the tensor rank can be factored in a simple way in this sense, whereas the slice rank and the partition rank cannot.)

\begin{definition} \label{product of ranks definition} Let $M \ge 1$, $ d_1,\dots ,d_M \ge 2$ be integers and let $ R_1,\dots ,R_M$ be non-empty families of partitions of $ \lbrack d_1 \rbrack, \dots, \lbrack d_M \rbrack$, respectively. Let $ R= R_1 \times \dots  \times R_M$ be the family of partitions of $ \{(i,j): 1 \le i \le M, 1 \le j \le d_i\}$ defined by $ \{\{P_{i,1} \cup \dots \cup P_{i,M}\}: P_{i,1} \in R_1,\dots ,P_{i,M} \in R_M \}$, where a partition $ P_{i,j}$ of $\lbrack d_i \rbrack$ is identified with the corresponding partition of $ \{i\} \times \lbrack d_i \rbrack$.

\end{definition}

For instance, if $M=2$, $d_1 = d_2=3$, and \[R_1=R_2=\{\{\{1\},\{2,3\}\}, \{\{2\},\{1,3\}\}, \{\{3\},\{1,2\}\} \} \] both correspond to the slice rank of order-$3$ tensors, then $R$ is the collection of partitions of $[2] \times [3]$ of the type \[\{\{1,\g_1\},\{1,\g_2,\g_3\}, \{2,\tau_1\},\{2,\tau_2,\tau_3\}\},\] for some $\g_1,\g_2,\g_3, \tau_1,\tau_2,\tau_3$ satisfying \[\{\g_1,\g_2,\g_3\} = \{\tau_1, \tau_2, \tau_3\} = \{1,2,3\}.\]

We shall use the following notion of flattening rank.

\begin{definition} For $T: \prod_{\a = 1}^{d_1} Q_{1,\a} \times \prod_{\a = 1}^{d_2} Q_{2,\a} \rightarrow \F$ an order-$(d_1+d_2)$ tensor, let the \emph{flattening rank} of $T$, denoted by $ \matfrank T$, be the rank of the matrix $A: \left(\prod_{\a = 1}^{d_1} Q_{1,\a}\right) \times \left(\prod_{\a = 1}^{d_2} Q_{2,\a}\right) \rightarrow \mathbb{F}$ defined by \[ A((x_{1,1},\dots ,x_{1,d_1}),(x_{2,2},\dots ,x_{2,d_2}))= T(x_{1,1},\dots ,x_{1,d_1},x_{2,2},\dots ,x_{2,d_2}). \]  \end{definition}

We remark that if $ T$ is such that $ \Rprodrk T \le k$ then $ \matfrank T \le k$: this follows from checking that if $\Rprodrk T= 1$ then $ \matfrank T \le 1$. The proof that we are about to start can be viewed as a generalisation of the proof of Proposition \ref{Subtensors for $1$-enhanced slice rank}.

\begin{proposition}\label{product of ranks bound} Let $d_1, d_2 \ge 2$ be integers, and let $ R_1, R_2$ be non-empty families of partitions of respectively $ \lbrack d_1 \rbrack$, $ \lbrack d_2 \rbrack$. If Theorem \ref{Subtensors theorem} holds for $(d_1, R_1)$ with bounds $ F_{d_1, R_1}$ and $ G_{d_1,R_1}$ and for $(d_2, R_2)$ with bounds $ F_{d_2, R_2}$ and $ G_{d_2,R_2}$ then it holds for $(d_1+d_2,R_1 \times R_2)$ with bounds \[ F_{d_1 + d_2, R_1 \times R_2}(l) = \max(l, F_{d_1,R_1}(l), F_{d_2,R_2}(l))\] \[ G_{d_1 + d_2, R_1 \times R_2}(l) = l G_{d_1, R_1}(l) G_{d_2, R_2}(l).\] \end{proposition}

\begin{proof}

Let $ T$ be a tensor with $\Rprodrk T \ge l G_{d_1, R_1}(l) G_{d_2,R_2}(l)$. We distinguish two cases.
\medskip

\noindent \textbf{Case 1}: We have $ \matfrank T \ge l$. Then letting $ A$ be as above and using the standard result on full-rank submatrices of matrices, there exist subsets $ X^1, X^2$ of $ \prod_{\a = 1}^{d_1} Q_{1,\a} $ and $\prod_{\a = 1}^{d_2} Q_{2,\a}$, respectively, with size at most $l$, such that $ \matfrank A(X^1 \times X^2) \ge l$. Let $X_{1,\a}$ be the projection of $ X^1$ on to the $(1,\a)$th coordinate axis for each $\a = 1,\dots ,d_1$, and similarly let $X_{2,\a}$ be the projection of $ X^2$ on  to the $(2,\a)$th coordinate axis for each $\a = 1,\dots ,d_2$. The tensor \[ T'=T(X_{1,1} \times \dots  \times X_{1,d_1} \times X_{2,1} \times \dots  \times X_{2,d_2}) \] satisfies $\matfrank T' \ge l$ and hence \[ \Rprodrk T' \ge l. \]
\medskip

\noindent \textbf{Case 2}: We have $ \matfrank T \le l$. We then let $ l'= \matfrank T$. There exist tensors $ T_{1,1},\dots ,T_{1,l'}: \prod_{\a = a}^{d_1} Q_{1,\a} \rightarrow \F $ and tensors $T_{2,1},\dots ,T_{2,l'}: \prod_{\a = a}^{d_2} Q_{2,\a} \rightarrow \F$ such that we can write \[ T(x_{1,1} \dots, x_{1,d_1}, x_{2,1}, \dots, x_{2,d_2}) = \sum_{i=1}^{l'} T_{1,i}(x_{1,1}, \dots, x_{1,d_1}) T_{2,i}(x_{2,1}, \dots, x_{2,d_2}) \] for all $x_{1,1} \in Q_{1,1}, \dots, x_{1,d_1} \in Q_{1,d_1}, x_{2,1} \in Q_{2,1}, \dots, x_{2,d_2} \in Q_{2,d_2}$. Moreover, the families of tensors $\{T_{1,1}, \dots T_{1,l'}\}$ and $\{T_{2,1}, \dots T_{2,l'}\}$ are both linearly independent (if they were not, then we would have $\matfrank T < l'$). For each $ i \in \lbrack l' \rbrack$, by definition of the $(R_1 \times R_2)$-rank we have \[ \Rprodrk (T_{1,i} T_{2,i}) \le \Ronerk T_{1,i} \Rtwork T_{2,i} \] so by subadditivity \[ \Rprodrk T \le \sum_{i=1}^{l'}\Ronerk T_{1,i} \Rtwork T_{2,i}. \] Since $ \Rprodrk T \ge l G_{d_1, R_1}(l) G_{d_2,R_2}(l)$ and $ l'\le l$, there exists $ i \in \lbrack l' \rbrack$ such that $\Ronerk T_{1,i} \ge G_{d_1,R_1}(l)$ or $\Rtwork T_{2,i} \ge G_{d_2,R_2}(l)$. Without loss of generality let us assume that $\Ronerk T_{1,i} \ge G_{d_1,R_1}(l)$. Because the family $\{T_{2,1}, \dots ,T_{2,l'}\}$ is linearly independent there exists a function $u: \prod_{\a=1}^{d_2} Q_{2,\a} \rightarrow \F$ supported inside a subset $U$ of $\prod_{\a=1}^{d_2} Q_{2,\a}$ with size at most $l' \le l$ such that $ u.T_{2,i'} = \delta_{i,i'}$ for all $ i' \in \lbrack l' \rbrack$, so $ u.T = T_{1,i}$. For each $ \a \in \lbrack d_2 \rbrack$ let $ X_{2,\a}$ be the projection of $ U$ on the $ (2,\a)$th coordinate axis. The sets $ X_{2,\a}$ all have size at most $ l' \le l$. By Theorem \ref{Subtensors theorem} for $(d_1,R_1)$ there exist sets $ X_{1,\a}, \a \in \lbrack d_1 \rbrack$ with size at most $ F_{d_1,R_1}(l)$ and such that \[ \Ronerk T_{1,i} (X_{1,1} \times \dots  \times X_{1,d_1}) \ge l. \] Letting \[ T'=T(X_{1,1} \times \dots  \times X_{1,d_1} \times X_{2,1} \times \dots  \times X_{2,d_2}) \] we have $ \Ronerk u.T' \ge l$. This ensures that $ \Rprodrk T \ge l$: indeed if we can write \[ T'(x_{1,1} \dots, x_{1,d_1}, x_{2,1}, \dots, x_{2,d_2}) = \sum_{i=1}^{r} T_{1,i}'(x_{1,1} \dots, x_{1,d_1}) T_{2,i}'(x_{2,1}, \dots, x_{2,d_2}) \] for some tensors $T_{1,i}': \prod_{\a = 1}^{d_1} Q_{1,\a} \rightarrow \F$, $T_{2,i}': \prod_{\a = 1}^{d_2} Q_{2,\a} \rightarrow \F$ and some positive integer $r$ then \[ u.T' = \sum_{i=1}^{r} T_{i,1}' (u.T_{i,2}') \] and hence $ \Ronerk u.T' \le r$.\end{proof}

Using induction on $ M$ by applying the previous proposition to $ R_1 \times \dots  \times R_{M-1}$ and $ R_M$ we obtain the following bounds. We write $R^{\otimes M}$ for the $M$-fold product of $R$ with itself.

\begin{corollary} Let $ M \ge 1$, $ d_1,\dots ,d_M \ge 2$ be integers, let $ R_1,\dots ,R_M$ be families of non-empty partitions of respectively $ \lbrack d_1 \rbrack$,\dots ,$ \lbrack d_M \rbrack$, and let $R = R_1 \times \dots \times R_M$. If for each $ j \in \lbrack M \rbrack$ Theorem \ref{Subtensors theorem} holds for $(d_j, R_j)$ with bounds $ F_{d_j, R_j}$ and $ G_{d_j, R_j}$ then Theorem \ref{Subtensors theorem} holds for $(d_1 + \dots + d_M, R_1 \times \dots  \times R_M)$ with the bounds \[ F_{d_1 + \dots  + d_M,R_1 \times \dots  \times R_M}(l) = \max \left(l, \max_{1 \le j \le M}F_{d_j,R_j}(l)\right) \] \[ G_{d_1 + \dots + d_M, R_1 \times \dots  \times R_M}(l) = l^{M-1} \prod_{j=1}^d G_{d_j, R_j}(l). \] \end{corollary}

\begin{corollary} Let $ M \ge 1$, $ d \ge 2$ be integers and let $ R$ be a family of non-empty partitions of $ \lbrack d \rbrack$. If Theorem \ref{Subtensors theorem} holds for $(d,R)$ with the bounds $ F_{d,R}$ and $ G_{d,R}$, then Theorem \ref{Subtensors theorem} holds for $(Md, R^{\otimes M})$ with the bounds \[ F_{Md, R^{\otimes M}}(l) = \max(l, F_{d,R}(l))\] \[ G_{Md, R^{\otimes M}}(l) = l^{M-1} G_{d,R}^M(l). \] \end{corollary}

\section{Partition rank subtensors in the finite field case using projections and analytic rank}\label{section: Projection}

Throughout this section and only in this section we assume that the field $\mathbb{F}$ is finite. We will prove Theorem \ref{Subtensors theorem} for the partition rank in this special case, and our main result will be Theorem \ref{Partition rank subtensors for finite fields}. Although our upper bound on the function $G_{d,\pr}^{|\mathbb{F}|}$ there turns out to be independent of $|\mathbb{F}|$, we shall nonetheless keep the notation of this function in order to distinguish it from the function $G_{d, \pr}$ from Theorem \ref{Subtensors theorem}, where the finiteness of the field is not assumed. For $f_d,g_d: \mathbb{N} \rightarrow [0, \infty)$ two sequences of functions, we shall write $f_d = O_d(g_d)$ for the statement that for every $d \ge 2$ there exist quantities $C_d, L_d > 0$ (depending on $d$) such that $f_d(l) \le C_d g_d(l)$ for all $l \ge L_d$.

\begin{theorem} \label{Partition rank subtensors for finite fields} Let $d \ge 2$ be an integer and let $\F$ be a finite field. Then there exist functions $F_{d,\pr}^{|\mathbb{F}|}: \mathbb{N} \rightarrow \mathbb{N}$, $G_{d,\pr}^{|\mathbb{F}|}: \mathbb{N} \rightarrow \mathbb{N}$ satisfying \[ F_{d,\pr}^{|\mathbb{F}|}(l) \le |\F|^{d!l}\text{ and } G_{d,\pr}^{|\mathbb{F}|}(l) = O_d(l \log^{d}(l)) \] such that the following property holds. If $T: Q_1 \times \dots \times Q_d \rightarrow \F$ is an order-$d$ tensor with $\pr T \ge G_{d,\pr}^{|\mathbb{F}|}(l)$ then there exist $X_1 \subset Q_1, \dots, X_d \subset Q_d$ each with size at most $F_{d,\pr}^{|\mathbb{F}|}(l)$ and such that $\pr T(X_1 \times \dots \times X_d) \ge l$. \end{theorem}

We will rely on a connection between the partition rank and the analytic rank which we begin by explaining. 

\begin{definition} \label{Multilinear form definition}

Let $T: \prod_{\a=1}^d Q_{\a} \rightarrow \mathbb{F}$ be an order-$d$ tensor. The \emph{multilinear form $m(T): \prod_{\a=1}^d \mathbb{F}^{Q_{\a}} \rightarrow \mathbb{F}$ associated with $T$} is defined by the formula \[ m(T)(u^{(1)},\dots ,u^{(d)}) = \sum_{x_1 \in Q_1,\dots ,x_d \in Q_d} T(x_1,\dots ,x_d) (u^{(1)})_{x_1} \dots (u^{(d)})_{x_d} \] for each $u^{(1)} \in \F^{Q_1}, \dots, u^{(d)} \in \F^{Q_d}$. \end{definition} 

For each $u \in \F^{Q_1}$ let $u.T: \prod_{\a=2}^d Q_{\a} \rightarrow \mathbb{F}$ be the order-$(d-1)$ tensor defined by \[ (u.T)(x_2,\dots ,x_d) = \sum_{x_1 \in Q_1} u(x_1) T(x_1,\dots x_d) \] for every $x_2 \in Q_2$, \dots, $x_d \in Q_d$. It is straightforward to check that the $(d-1)$-linear form $m(u.T): \prod_{\a=2}^d \mathbb{F}^{Q_{\a}} \rightarrow \mathbb{F}$ satisfies \[m(u.T)(u^{(2)},\dots ,u^{(d)}) = m(T)(u, u^{(2)}, \dots ,u^{(d)})\] for every $u^{(2)} \in \F^{Q_2}, \dots, u^{(d)} \in \F^{Q_d}$.

\begin{definition}

Let $d \ge 2$ be an integer, and let $T: \prod_{\a=1}^d Q_{\a} \rightarrow \mathbb{F}$ be an order-$d$ tensor. The \emph{bias} of $T$ is defined by \begin{equation} \bias(T) = \mathbb{E}_{u^{(1)} \in \F^{Q_1}, \dots, u^{(d)} \in \F^{Q_d}} \chi(m(T)(u^{(1)},\dots ,u^{(d)}))\label{expression of bias for a multilinear form}\end{equation} for any arbitrary non-trivial additive character $\chi$ of $\mathbb{F}$, that is, for any non-trivial group homomorphism from $(\F,+)$ to $(\C, \times)$.

\end{definition}

The linear forms $m((u^{(1)}, \dots, u^{(d-1)}).T): \F^{Q_d} \rightarrow \F$ defined by \[m((u^{(1)}, \dots, u^{(d-1)}).T)(u^{(d)}) = m(T)(u^{(1)}, \dots, u^{(d)})\] for each $u^{(1)} \in \F^{Q_1}, \dots, u^{(d)} \in \F^{Q_d}$ provide an interpretation showing that the right-hand side of \eqref{expression of bias for a multilinear form} is indeed independent of the non-trivial additive character $\chi$ of $\mathbb{F}$, and furthermore that $\bias(T)$ is always a positive real number. We begin by writing \[\bias(T) = \mathbb{E}_{u^{(1)} \in \F^{Q_1}, \dots, u^{(d-1)} \in \F^{Q_{d-1}}} \left (\mathbb{E}_{u^{(d)} \in \F^{Q_d}} \chi(m((u^{(1)}, \dots, u^{(d-1)}).T)(u^{(d)}))\right).\] Using that whenever $L: \F_{Q_d} \to \F$ is a linear form the expectation $\E_{u \in \F_{Q_d}} \chi(L(u))$ is equal to $1$ if $L \equiv 0$ and to $0$ otherwise, we obtain \[\bias(T) = \mathbb{P}_{u^{(1)} \in \F^{Q_1}, \dots, u^{(d-1)} \in \F^{Q_{d-1}}} (m((u^{(1)}, \dots, u^{(d-1)}).T)\equiv 0).\] The main property of the $\bias$ that we will use is that we can write \begin{equation} \label{averaging the bias} \bias(T) = \mathbb{E}_{u^{(1)} \in \F^{Q_1}} \bias (u^{(1)}.T). \end{equation}

If $d=2$, then it follows from \eqref{averaging the bias} that $\bias T = \mathbb{F}^{-\rk T}$. Proving qualitatively that for a fixed integer $d \ge 2$ and a fixed finite field $\F$, the bias of an order-$d$ tensor $T$ over $\F$ tends to $0$ as the partition rank of $T$ tends to infinity and then quantifying this asymptotic relationship has been the topic of a significant line of research. The \emph{analytic rank} of a tensor $T$, introduced in \cite{Gowers and Wolf}, is defined to be the quantity \[ \ar T = - \log_{|\mathbb{F}|} \bias T.\] It is known since works of Lovett \cite[Theorem 1.7]{Lovett} and independently of Kazhdan and Ziegler \cite[Lemma 2.2]{Kazhdan and Ziegler 2} that \begin{equation} \label{ar at most pr} \ar T \le \pr T. \end{equation} At the time that the first version of this manuscript was written, the best bounds \begin{equation} \label{pr at most a function of ar} \pr T \le A_{d,\F}(\ar T) \end{equation} in the converse direction had been obtained in works of Janzer and Mili\'cevi\'c: Janzer showed \cite[Theorem 1.10]{Janzer},  that for all $r \ge 1$ we can take $A_{d,\F}(r) = (c \log |\mathbb{F}|)^{c'(d)} (r)^{c'(d)}$ for $c$ an absolute constant and $c'(d) = 4^{d^d}$ and Mili\'cevi\'c showed \cite[Theorem 3]{Milicevic} that we can take $A_{d,\F}(r) = 2^{d^{2^{O(d^2)}}} \left(r^{2^{2^{O(d^2)}}} + 1\right)$ for all $r \ge 0$. Since then these bounds have been improved further to $A_{d,\F}(r) \le K_1(d) (r \log^{d-1}(1+r))$ for some $K_1(d) > 0$ by Moshkovitz and Zhu \cite[Theorem 1]{Moshkovitz and Zhu}. It has been asked by Lovett \cite[Problem 1.9]{Lovett} whether $A_{d,\F}(r) \le K(d) r$ for some $K(d) > 0$, and several similar statements have been conjectured \cite[Conjecture 1.10]{Adiprasito Kazhdan Ziegler}, \cite[Conjecture 1.7]{Kazhdan and Ziegler 1}, \cite[Conjecture 1.4]{Lampert and Ziegler}. In the case $d=3$, this bound was shown to be true by Moshkovitz and Cohen \cite[Theorem 1]{Moshkovitz and Cohen} provided that $\F \neq \F_2$ and a similar result was established independently by Adiprasito, Kazhdan and Ziegler \cite[Corollary 1.14]{Adiprasito Kazhdan Ziegler}.

The averaging identity \eqref{averaging the bias} together with inequalities \eqref{ar at most pr} and \eqref{pr at most a function of ar} between the partition and analytic ranks provides a route for our inductive argument: indeed starting from \eqref{averaging the bias} we can rewrite \[ |\mathbb{F}|^{- \ar T} = \mathbb{E}_{u \in \mathbb{F}^{Q_1}} |\mathbb{F}|^{-\ar (u.T)} \] and hence obtain \begin{equation} \label{averaging inequality for analytic ranks} |\mathbb{F}|^{-A_{d,\F}^{-1}(\pr T)} \ge \mathbb{E}_{u \in \F^{Q_1}} |\mathbb{F}|^{-\pr (u.T)}. \end{equation}

We now begin the proof of our subtensors result, Theorem \ref{Subtensors theorem}. Our first step will be to use inequality \eqref{averaging inequality for analytic ranks} to show that, in a sense that we are about to make precise, if a tensor has all its projections over the first coordinate approximately spanned by a family of order-$(d-1)$ tensors that has bounded size, then the tensor has bounded partition rank.

\begin{proposition} \label{A tensor without a large set of pr-separated projections has bounded pr} Let $d \ge 3$, $q,l \ge 1$ be integers. If $T: \prod_{\a=1}^d Q_{\a} \rightarrow \mathbb{F}$ is an order-$d$ tensor over a finite field $\F$ and $u^{(1)},\dots ,u^{(l)} \in \F^{Q_1}$ are such that for every $u \in \F^{Q_1}$ there exists $ v \in \mathbb{F}^l \setminus \{0\}$ satisfying \[ \pr \left(u.T - \sum_{h=1}^l v_h u^{(h)}.T \right) \le q,\] then \[ \pr T \le A_{d,\F}(l+q).\] \end{proposition}

\begin{proof} There are only at most $ |\mathbb{F}|^l$ linear combinations of $ u^{(1)},\dots ,u^{(l)}$, so by the assumption and the pigeonhole principle there is some $ v \in \mathbb{F}^l$ such that for a proportion at least $ |\mathbb{F}|^{-l}$ of the $u \in \mathbb{F}^{Q_1}$, \[ \pr \left(u.T - \sum_{h=1}^l v_h u^{(h)}.T\right) \le q. \] The change of variables $u \mapsto u - \sum_{h=1}^l v_h u^{(h)}$ therefore shows that a proportion at least $|\mathbb{F}|^{-l}$ of the $u \in \F^{Q_1}$ is such that $ \pr (u.T) \le q$. Hence, \[\mathbb{E}_{u \in \F^{Q_1}} |\mathbb{F}|^{-\pr (u.T)} \ge |\mathbb{F}|^{-l} |\mathbb{F}|^{-q}, \] so by the inequality \eqref{averaging inequality for analytic ranks} we have \[|\mathbb{F}|^{-A_{d,\F}^{-1}(\pr T)} \ge |\mathbb{F}|^{-l} |\mathbb{F}|^{-q}, \] from which it follows that $\pr T \le A_{d,\F}(l+q)$.\end{proof}

We can without much effort deduce from Proposition \ref{A tensor without a large set of pr-separated projections has bounded pr} that a tensor with high partition rank has a large separated family of projections over the first coordinate (as we may assume without loss of generality that the last non-zero coordinate of $v$ is equal to $1$).

\begin{corollary} \label{A high-rank tensor has a large pr-separated family of projections}

Let $q,l \ge 1$ be integers. Let $T: \prod_{\a=1}^d Q_{\a} \rightarrow \mathbb{F}$ be an order-$d$ tensor over a finite field $\F$ such that $\pr T \ge A_{d,\F}(l+q)$. Then there exist $u^{(1)},\dots ,u^{(l)} \in \mathbb{F}^{Q_1}$ such that \[ \pr \left(\sum_{h=1}^l v_h u^{(h)}.T\right) \ge q\]  for every $v \in \mathbb{F}^l \setminus \{0\}$. \end{corollary}

\begin{proof} Using the contrapositive of Proposition \ref{A tensor without a large set of pr-separated projections has bounded pr} we can construct the $u^{(h)}$ by induction on $h=1,\dots ,l$ as follows: we find $u^{(1)} \in \F^{Q_1}$ such that $\pr u^{(1)}.T \ge q$ and more generally at the $h$th step we find $u^{(h)} \in \F^{Q_1}$ such that \[ \pr \left(\left(u^{(h)} - \sum_{h'=1}^{h-1} v_{h'} u^{(h)\prime}\right). T\right) \ge q\] for every $v \in \mathbb{F}^{h-1} \setminus \{0\}$.\end{proof}

In our next lemma we show that, conversely, if a tensor has a large separated set of projections then it has high partition rank. Unlike the proof of Proposition \ref{A tensor without a large set of pr-separated projections has bounded pr}, the proof of the following lemma does not resort to the connection between partition rank and analytic rank (and hence works in an arbitrary field).

\begin{lemma} \label{If a tensor has a large separated set of projections then it has high partition rank} Let $l \ge 1$ be an integer, and let $T: \prod_{\a=1}^d Q_{\a} \rightarrow \mathbb{F}$ be an order-$d$ tensor. Suppose that there exist $u^{(1)},\dots ,u^{(l)} \in \mathbb{F}^{Q_1}$ such that \[ \pr \left(\sum_{h=1}^{l} v_h u^{(h)}.T \right)\ge l \] for every $v \in \mathbb{F}^l \setminus \{0\}$. Then $\pr T \ge l$. \end{lemma}

\begin{proof} Assume for contradiction that $\pr T \le l-1$. For some nonnegative integers $r,s$ with $r+s \le l-1$, that assumption provides us with a decomposition of the type \[T(x_1,\dots ,x_d) = \sum_{i=1}^{r} a_i(x_1) b_i(x_2,\dots ,x_d) + \sum_{i=1}^{s} c_i(x(I_i)) d_i(x(J_i)) \] for every $x_1 \in Q_1, \dots, x_d \in Q_d$. In this decomposition, for each $i \in \lbrack r \rbrack$ we have two functions $a_i: Q_1 \rightarrow \mathbb{F}$, $b_i: \prod_{\a=2}^d Q_{\a} \rightarrow \mathbb{F}$, and for each $i \in \lbrack s \rbrack$ we have a bipartition $\{I_i,J_i\}$ of $\lbrack d \rbrack$ with $I_i,J_i \neq \emptyset$, $1 \in I_i$, $I_i \neq \{1\}$, and two functions $c_i: \prod_{i \in I_i} Q_i \rightarrow \mathbb{F}$, $d_i: \prod_{i \in J_i} Q_i \rightarrow \mathbb{F}$. Let $u^{(1)},\dots ,u^{(l)}$ be arbitrary fixed elements of $\mathbb{F}^{Q_1}$. For each $h \in \lbrack l \rbrack$, \[ u^{(h)}.T = \sum_{i=1}^{r} (u^{(h)}.a_i) b_i(x_2,\dots ,x_d) + \sum_{i=1}^{s} (u^{(h)}.c_i)(x(I_i \setminus \{1\})) d_i(x(J_i)) \] where \[u^{(h)}.a_i = \sum_{x_1 \in Q_1} u^{(h)}(x_1) a_i(x_1) \in \F \] and \[(u^{(h)}.c_i): x(I \setminus \{1\}) \mapsto \sum_{x_1 \in Q_1} u^{(h)}(x_1) c_i(x(I)).\] Because $l \ge r+1$, there exists $v \in \mathbb{F}^l \setminus \{0\}$ such that $\sum_{h=1}^l v_h (u^{(h)}.a_i) = 0$ for all $i \in \lbrack r \rbrack$. Hence \[ \sum_{h=1}^l v_h (u^{(h)}.T) = \sum_{i=1}^{s} \left( \sum_{h=1}^l v_h (u^{(h)}.c_i) (x(I_i \setminus \{1\})) \right) d_i(x(J_i)). \] The right-hand side has partition rank at most $s \le l-1$, a contradiction. \end{proof}

For $d\geq 2$ we define two families of functions $F_{d, \pr}^{|\mathbb{F}|}: \mathbb{N} \rightarrow \mathbb{N}$ and $G_{d, \pr}^{|\mathbb{F}|}: \mathbb{N} \rightarrow \mathbb{N}$ inductively as follows. We set \begin{align*} F_{2, \pr}^{|\mathbb{F}|}(l) &= l\text{ and for each }d \ge 3\text{, }F_{d,\pr}^{|\mathbb{F}|}(l) = (|\mathbb{F}|^l F_{d-1, \pr}^{|\mathbb{F}|}(l))^{d-1}\\ G_{2, \pr}^{|\mathbb{F}|}(l) &= l\text{ and for each }d \ge 3\text{, } G_{d,\pr}^{|\mathbb{F}|}(l) = A_{d,\F}(G_{d-1,pr}^{|\mathbb{F}|}(l) + l).\end{align*}

The following pair of results, Proposition \ref{Obtaining from a separated set a separated set where furthermore the projection vectors have bounded support} and Proposition \ref{Subtensors for partition rank over finite fields}, will be proved by induction on $d$. The base case is Proposition \ref{Subtensors for partition rank over finite fields} for $d=2$, which is already known to us, as it is the case of matrix rank. Then, for every $d \ge 3$, Proposition \ref{Subtensors for partition rank over finite fields} in order $d-1$ implies Proposition \ref{Obtaining from a separated set a separated set where furthermore the projection vectors have bounded support} in order $d$, and Proposition \ref{Obtaining from a separated set a separated set where furthermore the projection vectors have bounded support} in order $d$ implies Proposition \ref{Subtensors for partition rank over finite fields} in order $d$. 

The statements that we have gathered so far in this section allow us to do the following: starting with an order-$d$ tensor $T$ with high partition rank, we find a large separated set of projections of $T$, then apply Proposition \ref{Subtensors for partition rank over finite fields} in order $d-1$ to obtain sets $X_2, \dots, X_d$ with bounded size such that the set of projections is still separated when restricted to the order-$(d-1)$ subtensor $X_2 \times \dots \times X_d$, which then ensures that $T(Q_1 \times X_2 \times \dots \times X_d)$ has high partition rank. In order to also be able to restrict the first set of coordinates we furthermore want to be able to assume that the projections have bounded support: this step will constitute an important part of the proof of Proposition \ref{Obtaining from a separated set a separated set where furthermore the projection vectors have bounded support}.

Proposition \ref{Obtaining from a separated set a separated set where furthermore the projection vectors have bounded support} informally states that if we can find several well-separated linear combinations of order-$(d-1)$ slices of $T$, then we can also find the same number of well-separated linear combinations of order-$(d-1)$ slices of $T$ that furthermore only involve a bounded number of order-$(d-1)$ slices of $T$.

\begin{proposition}\label{Obtaining from a separated set a separated set where furthermore the projection vectors have bounded support}

Let $d \ge 3$, $q,l \ge 1$ be integers. Let $T: \prod_{\a=1}^d Q_{\a} \rightarrow \mathbb{F}$ be an order-$d$ tensor over a finite field $\F$. If there exist $u^{(1)},\dots,u^{(l)} \in \mathbb{F}^{Q_1}$ such that \[ \pr \sum_{h=1}^l v_h (u^{(h)}.T) \ge G_{d-1, \pr}^{|\mathbb{F}|}(q) \] for every $v \in \mathbb{F}^l \setminus \{0\}$, then there exist $X_2 \subset Q_2$, \dots, $X_d \subset Q_d$ with size at most $|\mathbb{F}|^l F_{d-1, \pr}^{|\mathbb{F}|}(q)$ and $ u^{(1)\prime},\dots ,u^{(l)} \in \mathbb{F}^{Q_1}$ all supported inside a subset $X_1 \subset Q_1$ with size at most $(|\mathbb{F}|^l F_{d-1, \pr}^{|\mathbb{F}|}(q))^{d-1}$ such that \[ \pr \sum_{h=1}^l v_h (u^{(h)\prime}.T(X_2 \times \dots  \times X_d)) \ge q \] for every $v \in \mathbb{F}^l \setminus \{0\}$. \end{proposition}

Proposition \ref{Subtensors for partition rank over finite fields} establishes Theorem \ref{Subtensors theorem} for the partition rank over finite fields.

\begin{proposition} \label{Subtensors for partition rank over finite fields}

Let $d \ge 2$ be an integer. If $ T: \prod_{\a=1}^d Q_{\a} \rightarrow \F$ is an order-$d$ tensor over a finite field $\mathbb{F}$ with $ \pr T \ge G_{d, \pr}^{|\mathbb{F}|}(l)$, then there exist sets $X_1 \subset Q_1, \dots, X_d \subset Q_d$ of size at most $F_{d, \pr}^{|\mathbb{F}|}(l)$ such that \[ \pr T(X_1 \times \dots \times X_d) \ge l. \] \end{proposition}

\begin{proof} [Proof of Proposition \ref{Obtaining from a separated set a separated set where furthermore the projection vectors have bounded support}] For each $v \in \mathbb{F}^l \setminus \{0\}$ we can by Proposition \ref{Subtensors for partition rank over finite fields} in dimension $ d-1$ find sets $ X_{2,v} \subset Q_2 \dots, X_{d,v} \subset Q_d$ each with size at most $ F_{d-1, \pr}^{|\mathbb{F}|}(q)$ such that \[ \pr \left(\sum_{h=1}^l v_h (u^{(h)}.T)\right)(X_{2,v} \times \dots \times X_{d,v}) \ge q. \] Let $ X_{\a} = \bigcup_{v \in \mathbb{F}^l \setminus \{0\}} X_{\a,v}$ for each $ \a=2,\dots ,d$. We obtain \[ \pr \left(\sum_{h=1}^l v_h (u^{(h)}.T)\right)(X_2 \times \dots \times X_d) \ge q \] for all $v \in \mathbb{F}^l \setminus \{0\}$ and the sets $ X_2,\dots ,X_d$ each have size at most $ |\mathbb{F}|^l F_{d-1, \pr}^{|\mathbb{F}|}(q)$. The family $ \{T_{x_1}: x_1 \in Q_1 \}$ of slices $ (x_2,\dots ,x_d) \mapsto T(x_1,x_2, \dots,x_d)$ spans $ \{u.T: u \in \mathbb{F}^{Q_1}\}$, so in particular spans all linear combinations $\sum_{h=1}^l v_h (u^{(h)}.T)$ with $v \in \F^l \setminus \{0\}$. For each $x_1 \in Q_1$ let $T_{x_1}'$ be the restriction \[T_{x_1}(X_2 \times \dots  \times X_d).\] The family of restrictions $ \{T_{x_1}': x_1 \in Q_1 \}$ spans all linear combinations \[ \sum_{h=1}^l v_h (u^{(h)}.T)(X_2 \times \dots \times X_d)\] with $v \in \F^l \setminus \{0\}$. Because $ X_2 \times \dots \times X_d$ has size at most $ (|\mathbb{F}|^l F_{d-1, \pr}^{|\mathbb{F}|}(q))^{d-1}$, there exists a set $ X_1 \subset Q_1$ with size at most $(|\mathbb{F}|^l F_{d-1, \pr}^{|\mathbb{F}|}(q))^{d-1}$ such that \[ \langle \{T'_{x_1}: x_1 \in X_1\} \rangle = \langle \{T'_{x_1}: x_1 \in Q_1\} \rangle. \] For each $x_1 \in Q_1$ there exist coefficients $a_{y}(x_1) \in \mathbb{F}$ for every $y \in X_1$ such that \[ T'_{x_1} = \sum_{y \in X_1} a_{y}(x_1) T'_{y}. \] For each $h \in \lbrack l \rbrack$ we define $u^{(h)\prime} \in \F^{Q_1}$ by \[ u^{(h)\prime}(y) = \sum_{x_1 \in Q_1} u^{(h)}(x_1) a_{y}(x_1)\] for each $y \in X_1$ and by $u^{(h)\prime}(y) = 0$ for each $y \in Q_1 \setminus X_1$.  We can check that \[u^{(h)\prime}.T' = \sum_{y \in X_1} u^{(h)\prime}(y) T_{y}' = \sum_{x_1 \in Q_1} u^{(h)}(x_1) \left(\sum_{y \in X_1} a_{y}(x_1) T_{y}'\right) = \sum_{x_1 \in Q_1} u^{(h)}(x_1) T_{x_1}'= u^{(h)}.T' \] for each $h \in \lbrack l \rbrack$. By linearity, \[\sum_{h=1}^l v_h u^{(h)\prime}.T' = \sum_{h=1}^l v_h u^{(h)}.T'\] for each $v \in \mathbb{F}^l \setminus \{0\}$. Since taking restrictions cannot increase the partition rank, the result follows by Lemma \ref{If a tensor has a large separated set of projections then it has high partition rank}. \end{proof}

\begin{proof}[Proof of Proposition \ref{Subtensors for partition rank over finite fields}] Let $T$ be an order-$d$ tensor with $\pr T \ge A_{d,\F}(G_{d-1,\pr}^{|\mathbb{F}|}(l) + l)$. By Corollary \ref{A high-rank tensor has a large pr-separated family of projections} there exist $u^{(1)},\dots ,u^{(l)} \in \F^{Q_1}$ such that \[ \pr \left(\left(\sum_{h=1}^{l} v_h u^{(h)}\right).T\right) \ge G_{d-1,\pr}^{|\mathbb{F}|}(l) \] for every $v \in \mathbb{F}^l \setminus \{0\}$. By Proposition \ref{Obtaining from a separated set a separated set where furthermore the projection vectors have bounded support} there exist $X_2 \subset Q_2,\dots ,X_d \subset Q_d$ with size at most $|\mathbb{F}|^l F_{d-1, \pr}^{|\mathbb{F}|}(l)$, a set $X_1$ with size at most $(|\mathbb{F}|^l F_{d-1, \pr}^{|\mathbb{F}|}(l))^{d-1}$ and $u^{(1)\prime},\dots ,u^{(l)\prime} \in \F^{Q_1}$ all supported inside $X_1$ such that \[ \pr \left(\left(\sum_{h=1}^{l} v_h u^{(h)\prime}\right).T\right)(X_2 \times \dots  \times X_d) \ge l \] for every $v \in \mathbb{F}^l \setminus \{0\}$. By Lemma \ref{If a tensor has a large separated set of projections then it has high partition rank} we conclude that \[ \pr T(X_1 \times \dots  \times X_d) \ge l. \qedhere \] \end{proof}

We would like to make a pair of remarks on the bounds in Theorem \ref{Partition rank subtensors for finite fields} arising from our argument. Firstly, it follows from the inductive expression of $G_{d, \pr}^{|\F|}$ that if $A_d(r) \le K(d)r$ is indeed true, then the bound on $G_{d, \pr}^{|\F|}$ in Theorem \ref{Partition rank subtensors for finite fields} can be taken to be linear in $l$ for any fixed $d \ge 2$. Secondly, the exponential bound in $l$ on the function $F_{d, \pr}^{|\F|}$ comes from the construction, for every $\a =2, \dots, d$, of $X_{\a}$ each as a union of the $|\F|^l$ sets $X_{\a,v}$; this can be avoided by using Proposition \ref{Multi-tensor Rrk subtensors} instead, which then leads to the inductive expressions $F_{d,\pr}^{|\mathbb{F}|}(l) = (l F_{d-1, \pr}^{|\mathbb{F}|}(l^2))^{d-1}$ and $G_{d,\pr}^{|\mathbb{F}|}(l) = A_{d,\F}(G_{d-1,\pr}^{|\mathbb{F}|}(l^2) + l)$ for each $d \ge 3$. These improve the bound on $F_{d, \pr}^{|\F|}(l)$ to a power bound $F_{d, \pr}^{|\F|}(l) \le l^{2^d d!}$, but at the cost of worsening the quasi-linear bound on $G_{d, \pr}^{|\F|}(l)$ to a power bound $G_{d, \pr}^{|\F|}(l) = O_d (l^{2^d})$.

Using several of the earlier ingredients of the proof of Theorem \ref{Partition rank subtensors for finite fields}, we may obtain a qualitatively weaker variant of Theorem \ref{Partition rank subtensors for finite fields} involving ``coordinate-free" restrictions of tensors where we allow for linear transformations of the associated multilinear forms, but for which we obtain quasi-linear bounds for both functions, which would improve further to linear bounds if the inequality $A_d(r) \le K(d) r$ were proved. More precisely, let $A'_{d, \F}(l) = A_{d,\F}(2l)$ and for each nonnegative integer $i \ge 0$ let $A_{d, \F}^{'\circ i}$ be the $i$-fold iteration of $A'_{d, \F}$; the bound $A_{d, \F}^{'\circ d}(l)$ that we obtain in Proposition \ref{Coordinate-free subtensors for partition rank over finite fields} for both functions is at most $O_d(l \log^{d}(l))$, and would be at most $(2K(d))^dl$ for all $l \ge 1$ provided that $A_{d,\F}(r) \le K(d)r$.

For $T: \prod_{\a=1}^d Q_{\a} \rightarrow \F$ an order-$d$ tensor, for linear subspaces $U_{\a} \subset \F^{Q_{\a}}$, subsets $Q_{\a}'$ of $\mathbb{N}$ with size $\dim U_{\a}$ and linear isomorphisms $L_{\a}: \F^{Q_{\a}'} \rightarrow U_{\a}$ for each $\a \in [d]$, we write $T(L_1, \dots, L_d): \prod_{\a=1}^d Q_{\a}' \rightarrow \F$ for the tensor defined by \[m(T(L_1, \dots, L_d))(u^{(1)}, \dots, u^{(d)}) = m(T)(L_1(u^{(1)}), \dots, L_d(u^{(d)}))\] for every $u^{(1)} \in \F^{Q_1'}, \dots, u^{(d)} \in \F^{Q{d}'}$.

\begin{proposition} \label{Coordinate-free subtensors for partition rank over finite fields} Let $d \ge 2$ be an integer, and let $l \ge 1$ be an integer. If $T: \prod_{\a=1}^d Q_{\a} \rightarrow \F$ is an order-$d$ tensor such that $\pr T \ge A_{d, \F}^{\prime \circ d}(l)$, then there exist for each $\a \in [d]$ a linear subspace $U_{\a} \subset \F^{Q_{\a}}$ with dimension at most $A_{d, \F}^{\prime \circ d}(l)$ and a linear isomorphism $L_{\a}: \F^{\dim U_{\a}} \rightarrow U_{\a}$ such that \[\pr T(L_1, \dots, L_d) \ge l.\] \end{proposition}

\begin{proof}

Since $\pr T \ge A_{d, \F}^{\prime \circ d}(l)$, by Corollary \ref{A high-rank tensor has a large pr-separated family of projections} we can find $u^{(1)}, \dots, u^{(A_{d, \F}^{\prime \circ (d-1)}(l))} \in \F^{Q_1}$ such that \[\pr \left( \left(\sum_{h=1}^l v_h u^{(h)}\right).T \right) \ge l\] for all $v \in \F^{A_{d, \F}^{\prime \circ (d-1)}(l)} \setminus \{0\}$. We select $U_1 = \bigg\langle u^{(1)}, \dots, u^{(A_{d,\F}^{\prime \circ (d-1)}(l))} \bigg\rangle$ and $L_1: \F^{\dim U_1} \rightarrow U_1$ the linear isomorphism which for every $i \in [A_{d,\F}^{\prime \circ (d-1)}(l)]$ sends the $i$th canonical basis vector of $\F^{\dim U_1}$ to $u^{(i)}$. Using Lemma \ref{If a tensor has a large separated set of projections then it has high partition rank} we obtain \[\pr T(L_1, \Id_{\F^{Q_2}}, \dots, \Id_{\F^{Q_d}}) \ge A^{\prime \circ (d-1)}(l).\] Iterating with the coordinates $\a=2,\dots,d$ successively we obtain subspaces $U_\a \subset \F^{Q_{\a}}$ with $\dim U_{\a} \le A_{d,\F}^{\prime \circ (d-\a)}(l)$ and linear isomorphisms $L_{\a}: \F^{\dim U_{\a}} \rightarrow U_{\a}$ such that \[\pr T(L_1, \dots, L_\a, \Id_{\F^{Q_{\a+1}}}, \dots, \Id_{\F^{Q_d}}) \ge A_{d,\F}^{\prime \circ (d-\a)}(l).\] In particular, all linear subspaces $U_1, \dots, U_d$ have dimension at most $A_{d,\F}^{\prime \circ d}(l)$ and \[\pr T(L_1, \dots, L_d) \ge l. \qedhere\]

\end{proof}

\section*{Acknowledgements} 

We thank Timothy Gowers for several useful conversations, for discovering Proposition \ref{Counterexample} and constructing the corresponding example, which we publish here with his permission, and for helpful comments on an earlier draft of this paper. We thank Lisa Sauermann for a sketch that led to Proposition \ref{Disjoint rank for matrices} and which we use here with her permission. We are grateful for the hospitality of the University of Cambridge, where the first version of this manuscript was written, supported by an EPSRC International Doctoral Scholarship (project reference 2114524).



\begin{aicauthors}
\begin{authorinfo}[tk]
Thomas Karam\\
School of Mathematical Sciences\\
Shanghai Jiao Tong University\\
Shanghai, China.\\
thomaskaram@sjtu.edu.cn\\
\end{authorinfo}
\end{aicauthors}


\begin{thebibliography}{9}


\bibitem{Adiprasito Kazhdan Ziegler}
K. Adiprasito, D. Kazhdan, and T. Ziegler, \textit{On the Schmidt and analytic ranks for trilinear forms}, arXiv:2102.03659 (2021).

\bibitem{Croot Lev Pach}
E. Croot, V. Lev and P. Pach, \textit{Progression-free sets in} $\mathbb{Z}_4^n$ \textit{are exponentially small}, Ann. Math. \textbf{185} (2017), 331\textendash 337.

\bibitem{Ellenberg Gijswijt}
J. Ellenberg and D. Gijswijt, \textit{On large subsets of} $\mathbb{F}_q^n$ \textit{with no three-term arithmetic progression} Ann of Math. \textbf{185} (2017), 339\textendash 343.

\bibitem{Gowers}
W. T. Gowers, \textit{The slice rank of a direct sum}, arXiv:2105.08394v2 (2022).

\bibitem{Gowers and K equidistribution}
W. T. Gowers and T. Karam, \textit{Equidistribution of high rank polynomials with variables restricted to subsets of $\F_p$}, arXiv:2209.04932 (2022).

\bibitem{Gowers and Wolf}
W. T. Gowers and J. Wolf, \textit{Linear forms and higher-degree uniformity for functions on} $\mathbb{F}_p^n$, Geom. Funct. Anal. \textbf{21} (2011), 36\textendash 69.

\bibitem{Green and Tao}
B. Green and T. Tao, \textit{The distribution of polynomials over finite fields, with applications to the Gowers norms.} Contrib. Discrete Math. \textbf{4} (2009), 1\textendash 36.

\bibitem{Janzer}
O. Janzer, \textit{Polynomial bound for the partition rank vs the analytic rank of tensors}, Discrete Anal. Paper No. 7 (2020), 18 pp.

\bibitem{Kazhdan and Ziegler 1}
D. Kazhdan and T. Ziegler, \textit{Applications of algebraic combinatorics to algebraic geometry}, Indag. Math. \textbf{39} (2021), 1412\textendash 1428.

\bibitem{Kazhdan and Ziegler 2}
D. Kazhdan and T. Ziegler, \textit{Approximate cohomology}, Selecta Math. \textbf{24} (2018), 499\textendash 509.

\bibitem{Lampert and Ziegler}
A. Lampert and T. Ziegler, \textit{Relative rank and regularization}, arXiv:2106.03933 (2021).

\bibitem{Lovett}
S. Lovett, \textit{The analytic rank of tensors and its applications}, Discrete Anal. Paper No. 7 (2019), 10 pp.

\bibitem{Milicevic}
L. Mili\'cevi\'c, \textit{Polynomial bound for partition rank in terms of analytic rank}, Geom. Funct. Anal. \textbf{29} (2019), 1503\textendash 1530.

\bibitem{Moshkovitz and Cohen}
A. Cohen and G. Moshkovitz, Structure vs. randomness for bilinear maps, Discrete Anal. \textbf{12} (2022).

\bibitem{Moshkovitz and Zhu}
G. Moshkovitz and D. G. Zhu, \textit{Quasi-linear relation between partition and analytic rank}, arXiv:2211.05780 (2022).

\bibitem{Naslund recent}
E. Naslund, \textit{The chromatic number of} $\mathbb{R}^{n}$ \textit{with multiple forbidden distances}, arXiv:2205.12312 (2022).

\bibitem{Naslund}
E. Naslund, \textit{The partition rank of a tensor and k-right corners in} $\mathbb{F}_q^n$, J. Combin. Theory Ser. A \textbf{174} (2020), 105190, 25pp.

\bibitem{Sauermann}
L. Sauermann, \textit{Finding solutions with distinct variables to systems of linear equations over} $\mathbb{F}_p$, Math. Ann. \textbf{386} (2023), 1-33.

\bibitem{Sawin and Tao}
W. Sawin and T. Tao, \textit{Notes on the “slice rank" of tensors}, https://terrytao.wordpress.com/2016/08/24/notes-on-the-slice-rank-of-tensors.

\bibitem{Tao}
T. Tao, \textit{A symmetric formulation of the Croot-Lev-Pach-Ellenberg-Gijswijt capset bound}, https://terrytao.wordpress.com/2016/05/18/a-symmetric-formulation-of-the-croot-lev-pach-ellenberg-gijswijt-capset-bound.

\end{thebibliography}
\end{document}